\documentclass[10pt,letterpaper]{amsart}
\usepackage{amsmath,amssymb,latexsym,cancel,rotating}
\usepackage{graphicx,amssymb,mathrsfs,amsmath,color,fancyhdr,amsthm}
\usepackage[all]{xy}
\usepackage{setspace}
\newtheorem{theorem}{Theorem}[section]
\newtheorem{lemma}[theorem]{Lemma}
\newtheorem{corollary}[theorem]{Corollary}
\newtheorem{definition}[theorem]{Definition}
\newtheorem{proposition}[theorem]{Proposition}

\theoremstyle{definition}
\newtheorem{remark}[theorem]{Remark}

\def\bV{\mathbf{V}}
\def\bfEVO{\mathbf{E}_{\mathbf{V},\Omega}}
\def\bfEVH{\mathbf{E}_{\mathbf{V},H}}

\def\mD{\mathcal{D}}
\def\mP{\mathcal{P}}
\def\mQ{\mathcal{Q}}
\def\mN{\mathcal{N}}
\def\mK{\mathcal{K}}
\def\mL{\mathcal{L}}
\def\Hom{{\rm{Hom}}}

\def\CC{{\rm{CC}}}

\begin{document}
\title{Lusztig sheaves, decomposition rule and restriction rule}
\author[Yixin Lan]{Yixin Lan}
\address{Academy of Mathematics and Systems Science, Chinese Academy of Sciences, Beijing 100190, P.R.China}
\email{lanyixin@amss.ac.cn (Y.Lan)}
\thanks{The author is supported by the National Natural Science Foundation of China [Grant No. 1288201] and [Grant No. 12471030].}


\date{\today}

\bibliographystyle{abbrv}


\begin{abstract}
In this article, we realize the subquotient based modules of certain tensor products or restricted modules via Lusztig's perverse sheaves on  multi-framed quivers, and provide a construction of their canonical bases. As an application, we prove that  the decomposition and restriction coefficients of symmetric Kac-Moody algebras equal to the dimensions of top Borel-Moore homology groups for certain locally closed subsets of Nakajima's quiver varieties.
\end{abstract}
\maketitle
\setcounter{tocdepth}{1}\tableofcontents

\begin{spacing}{1}

\section{Introduction}\label{introduction}
Given a symmetric Cartan matrix $C=(a_{ij})_{i,j\in I}$, the matrix $C$ defines a bilinear form $\langle -,-\rangle: \mathbb{Z}I \times \mathbb{Z} I \rightarrow \mathbb{Z}$, and there is a Kac-Moody Lie algebra $\mathfrak{g}$ associated to $C$.  The quantum group $\mathbf{U}(\mathfrak{g})$ is the $\mathbb{Q}(v)$-algebra generated by $E_i, F_i, K_{\nu}$ for $i\in I, \nu\in \mathbb{Z}I$, subject to the following relations:
\begin{itemize}
	\item {\rm{(a)}} $K_0 = 1$ and $K_{\nu}K_{\nu'} = K_{\nu+\nu'}$ for any $\nu,\nu' \in \mathbb{Z}I$;
	\item {\rm{(b)}} $K_{\nu}E_i = v^{\langle \nu,i \rangle} E_i K_{\nu}$ for any $i \in I, \nu \in \mathbb{Z}I$;
	\item {\rm{(c)}} $K_{\nu}F_i = v^{-\langle \nu,i \rangle} F_i K_{\nu}$ for any $i \in I, \nu \in \mathbb{Z}I$;
	\item {\rm{(d)}} $E_i F_j - F_j E_i = \delta_{ij} \frac{K_{i} - K_{-i}}{v - v^{-1}}$ for any $i,j\in I$;
	\item {\rm{(e)}} $\sum_{p+q=1-a_{ij}} (-1)^p E_i^{(p)} E_j E_i^{(q)} = 0$ for any $i\not=j$;
	\item {\rm{(f)}} $\sum_{p+q=1-a_{ij}} (-1)^p F_i^{(p)} F_j F_i^{(q)} = 0$ for any $i\not=j$.
\end{itemize}
Here $E_i^{(n)}= E_i^n/[n]!$, $F_i^{(n)}= F_i^n/[n]!$ are the divided power of $E_i,F_{i}$, where $[n]! := \prod _{s=1}^n \frac{v^s - v^{-s}}{v - v^{-1}}$.  Taking the classical limit $v\rightarrow 1$, we obtain the universal enveloping algebra $\mathbf{U}_{1}(\mathfrak{g})$ of $\mathfrak{g}$ from the quantum group.

Given a dominant integral  weights  $\lambda$, the module $M(\lambda)$ is realized by 
\begin{equation*}
	\begin{split}
		M(\lambda) =& \mathbf{U}(\mathfrak{g})/ ( \sum_{i\in I} \mathbf{U}(\mathfrak{g}) E_i + \sum_{\nu \in \mathbb{Z}[I]} \mathbf{U}(\mathfrak{g}) (K_{\nu} - v^{\langle \nu,\lambda \rangle})).
	\end{split}
\end{equation*}
The Verma module $M(\lambda)$ has a unique irreducible quotient,
\begin{equation}\label{canquotient}
	\begin{split}
		L(\lambda) =& \mathbf{U}(\mathfrak{g})/ ( \sum_{i\in I} \mathbf{U}(\mathfrak{g}) E_i + \sum_{\nu \in \mathbb{Z}[I]} \mathbf{U}(\mathfrak{g}) (K_{\nu} - v^{\langle \nu,\lambda \rangle}) + \sum_{i \in I} \mathbf{U}(\mathfrak{g}) F_i^{\langle i,\lambda \rangle +1})\\
		\cong & \mathbf{U}^{-}(\mathfrak{g})/ (  \sum_{i \in I} \mathbf{U}^{-}(\mathfrak{g}) F_i^{\langle i,\lambda \rangle +1}).
	\end{split}
\end{equation}
Take $v \rightarrow 1$, we can obtain an irreducible highest weight $\mathfrak{g}$-module $L_{1}(\lambda)$.

\subsection{Decomposition rule and Restriction Rule}
Given a Kac-Moody algebra $\mathfrak{g}$ or its quantized enveloping algebra $\mathbf{U}(\mathfrak{g})$, one can consider the category of integrable highest weight $\mathfrak{g}$(or $\mathbf{U}(\mathfrak{g})$)-modules. The irreducible  integrable highest weight $\mathfrak{g}$-modules $L_{1}(\lambda)$  (or $\mathbf{U}(\mathfrak{g})$-modules $L(\lambda)$) are parameterized by dominant integral  weights  $\lambda$ of $\mathfrak{g}$.

For any two dominant integral  weights  $\lambda^{1}$ and $\lambda^{2}$, the tensor product $L_{1}(\lambda^{1}) \otimes L_{1}(\lambda^{2})$ is also an integrable highest weight $\mathfrak{g}$-module, and decomposed into a direct sum of  irreducible integrable highest weight $\mathfrak{g}$-modules
\begin{equation}\label{decnumber}
	L_{1}(\lambda^{1}) \otimes L_{1}(\lambda^{2}) \cong \bigoplus_{\mu, \alpha \in S(\mu)}L_{1}(\mu)^{\alpha},
\end{equation}
here $\mu$ runs over  dominant integral  weights  of $\mathfrak{g}$ and $S(\mu)$ is the set of direct summands of $L_{1}(\lambda^{1}) \otimes L_{1}(\lambda^{2})$, which is isomorphic to $L_{1}(\mu)$. This is called the decomposition rule (or Littlewood-Richardson rule), and the order $m_{\mu}^{\lambda^{1},\lambda^{2}}$ of $S(\mu)$ is called the decomposition coefficient.

Similarly, we can consider a Levi subalgebra $\mathfrak{l}$ of $\mathfrak{g}$, and consider the restricted module $\mathbf{res}^{\mathfrak{g}}_{\mathfrak{l}} L_{1}(\lambda)$. As an integrable highest weight $\mathfrak{l}$-module,  $\mathbf{res}^{\mathfrak{g}}_{\mathfrak{l}} L_{1}(\lambda)$ is also ecomposed into a direct sum of  irreducible integrable highest weight $\mathfrak{l}$-modules
\begin{equation}\label{resnumber}
	\mathbf{res}^{\mathfrak{g}}_{\mathfrak{l}} L_{1}(\lambda)=\bigoplus_{\mu_{\mathfrak{l}}, \alpha \in S(\mu_{\mathfrak{l}})}L_{1}(\mu_{\mathfrak{l}})^{\alpha},
\end{equation}
here $\mu_{\mathfrak{l}}$ runs over  dominant integral  weights  of $\mathfrak{l}$ and $S(\mu_{\mathfrak{l}})$ is the set of direct summands of $\mathbf{res}^{\mathfrak{g}}_{\mathfrak{l}} L(\lambda)$, which is isomorphic to $L_{1}(\mu_{\mathfrak{l}})$. This is called the restriction rule, and  the order $m_{\mu_{\mathfrak{l}}}^{\lambda}$ of $S(\mu_{\mathfrak{l}})$ is called the restriction coefficient. 

It is an important question to determine the decomposition rule and restriction rule.  The author in \cite{Path} construct Demazure modules from  paths and root operators, and prove that the decomposition rule and restriction rule is determined by certain paths. Later, Kashiwara proves similar Demazure character formula for his crystals in \cite{KM}. Their results give a  combinatorial solution  to this question. In this article, we expect to give a categorical and geometric answer to this question.

\subsection{Perverse sheaves and Canonical bases for quantum groups and their representations}

For a given symmetric generalized Cartan matrix $C$, there is a finite quiver $Q=(I,H,\Omega)$ associated to $C$. We assume that $Q$ doesn't have loops but may have multiple arrows. In \cite{MR1088333},\cite{MR1227098}, \cite{MR1653038} and \cite{CBandH}, G.Lusztig has considered a certain class of  perverse sheaves on the moduli spaceof  representations of $Q$. These semisimple complexes are called Lusztig's sheaves.  The Grothendieck group of Lusztig's sheaves  has a bialgebra structure, which is canonically isomorphic to the integral form ${_{\mathcal{A}}\mathbf{U}^{+}}(\mathfrak{g})$ (or ${_{\mathcal{A}}\mathbf{U}^{-}}(\mathfrak{g})$) of the positive (or negative) part of the quantum group.  Moreover, the simple Lusztig sheaves forms the canonical basis.

Under the canonical map (\ref{canquotient}), the nonzero images of elements in the canonical basis provide a basis of $L(\lambda)$, which is called the canonical basis of $L(\lambda)$ by Lusztig. By a algebraic construction, Lusztig also defines the canonical basis for certain tensor products in \cite{lusztig1992canonical}, and his construction is generalized by \cite{bao2016canonical}. One can also see details in Section 2.1. Inspired by \cite{zheng2014categorification} and \cite{MR3177922}, the authors of \cite{fang2023lusztig} and \cite{fang2023tensor} consider certain localizations of Lusztig's category $\mathcal{Q}$ on multi-framed quivers, and give a sheaf theoretic realization of the irreducible integrable highest weight module $L(\lambda)$ and their tensor products. They also prove that the nonzero simple perverse sheaves in localization coincide with the canonical basis defined by \cite{bao2016canonical} and \cite{lusztig1992canonical}. See details in Section 3.

To determine the decomposition rule, it is a natural idea to decompose the canonical basis  of the tensor product into canonical bases of direct summands. However, the canonical basis is not compatible with direct summands in general, but only compatible with certain subquotients $M[\geqslant \mu]$, $M[>\mu]$ and $M[\geqslant \mu]/M[> \mu]$. See details in Section 2.2. Fortunately, the informations of these subquotients are enough to determine the decomposition rule. 

Our first main result is to give a  sheaf theoretic realization of those sub/quotient based modules mentioned above. More precisely, we base on the results in \cite{fang2023tensor}, and define thick subcategories $\mD^{\geqslant \mu}$ and  $\mD^{> \mu}$ for a dominant integral wight $\mu$ of $\mathfrak{g}$ by some micro-local conditions. The Grothendieck groups of localized Lusztig's sheaves $\mL^{\geqslant \mu}, \mL^{> \mu}$ and $\mL^{ \mu}$ in these thick subcategories are isomorphic to certain subquotients defined in \cite[27.1 and 27.2]{MR1227098}, and nonzero simple perverse sheaves provide the canonical bases of these subquotients. See details in Theorem \ref{main1}.

Moreover, if we focus on a subset $J \subseteq I$ and regard the vertices in $I \backslash J$ as framing, similar  micro-local conditions define thick subcategories $\mD^{\geqslant \mu_{J}}$ and  $\mD^{> \mu_{J}}$ for a dominant integral wight $\mu_{J}$ of the Levi subalgebra $\mathfrak{l}=\mathfrak{l}_{J}$, such that the Grothendieck groups of localized Lusztig's sheaves $\mL^{\geqslant \mu_{J}}, \mL^{> \mu_{J}}$ and $\mL^{ \mu_{J}}$ in these thick subcategories are isomorphic to certain subquotients of restricted modules.  Simple perverse sheaves also provide the canonical bases, which can determine the restriction rule. See details in Theorem \ref{main2}.

\subsection{Nakajima's quiver varieties and their relations with canonical bases}

Using GIT theory, Nakajima defines the quiver varieties and consider their  top Borel-Moore homology groups in \cite{MR1604167}. The Hecke correspondence induces operators on  Borel-Moore homology groups, which give  $\mathfrak{g}$-module structures on  Borel-Moore homology groups of quiver varieties. In particular, the top Borel-Moore homology group of the Lagrangian fiber $\mathfrak{L}(\omega)$ is isomorphic to the irreducible integrable highest weight  $\mathfrak{g}$-module. Later, in \cite{MR1865400}, Nakajima also defines the tensor product varieties, whose top Borel-Moore homology groups realize the tensor products of irreducible integrable highest weight  $\mathfrak{g}$-modules. Later, other homology  theory on quiver varieties are considered by many mathematicians. For example,  the authors in \cite{maulik2019quantum}  consider the quantum cohomology, and the authors in \cite{nakajima2001quiver},  \cite{KHA} and \cite{VV} consider the equivariant K-theory. Other theories like coherent sheaves, $\mathcal{D}$-modules and (monodromic) mixed Hodge modules are also considered in \cite{BL}, \cite{CKL}, \cite{BenD1} and \cite{WG}.

Following \cite{hennecart2024geometric}, the authors in \cite{fang2025lusztigsheavescharacteristiccycles} use the characteristic cycles to build a $\mathfrak{g}$-linear morphism from the Grothendieck groups of localization in \cite{fang2023lusztig} and \cite{fang2023tensor} to the Borel-Moore homology groups of Nakajima's quiver variety \cite{MR1604167} and tensor product variety \cite{MR1865400}. This result builds a relation between Lusztig's canonical basis and Nakajima's quiver varieties. See details in Section 5.

As an application, by applying the characteristic cycles defined in \cite{fang2025lusztigsheavescharacteristiccycles}, we can prove that the decomposition and restriction coefficients in (\ref{decnumber}) and (\ref{resnumber}) and the dimension of the coinvariants are equal to the top Borel-Moore homology groups of certain locally closed subset of Nakajima's quiver variety. See details in \ref{resrule}, \ref{decrule}  and    \ref{main5}.

\subsection{The structure of this article}

In Section 2, we recall the algebraic construction of based modules arising from tensor products. In Section 3, we recall the sheaf theoretic realization of the tensor products and their canonical bases. In Section 4, we define our $\mD^{\geqslant \mu},\mD^{> \mu}$ and $\mD^{\geqslant \mu_{J}},\mD^{> \mu_{J}}$ to realize  certain based modules, and prove our main theorems. In Section 5, we establish the relation between our sheaf theoretic realization and homological realization of Nakajima's quiver variety.

\subsection*{Acknowledgements.}
Y. Lan is supported by the National Natural Science Foundation of China [Grant No. 1288201] and [Grant No. 12471030]. This paper is a continuation of  collaborative works with Jiepeng Fang and Jie Xiao, I am very grateful to their detailed discussion. 

\subsection*{Convention}
Throughout this paper, all varieties are over $\mathbb{C}$ and all sheaves are constructible and with $\mathbb{C}$-coefficients. Since we only work on derived categories, we denote by $f^{\ast},f_{\ast},f^{!},f_{!}$ their derived functors for simplicity.   We denote the $n$-times shift functor in a triangulated category  by $[n]$. We denote the singular support of a complex $L$ by $SS(L)$. The subscript $\pm 1$ means the specialization of a quantized algebra or module at $v=\pm 1$, and the subscript $\mathbb{Z}$, $\mathbf{A}=\mathbb{Z}[v^{-1}]$ or $\mathcal{A}=\mathbb{Z}[v^{\pm}]$ means the $\mathbb{Z}$-form, $\mathbf{A}$-lattice or the integral form of an algebra or module.  For example, $_{\mathbb{Z}}\mathbf{U}_{1}(\mathfrak{g})$ means the $\mathbb{Z}$-form of the specialization of the quantum group at $v=1$, and $_{\mathcal{A}} L(\lambda)$ is the $\mathcal{A}$-submodule of $L(\lambda)$ spanned by its canonical basis.

\section{Preliminaries}
In this section, we briefly recall the algebraic construction of the tensor products of based modules in \cite{bao2016canonical} and \cite{lusztig1992canonical}, and the based module structure of coinvariants in \cite[27.2.5]{MR1227098}.

Recall that a based module is an integrable $\mathbf{U}$-module $M$ with a given $\mathbb{Q}(v)$-basis $B(M)$ such that
\begin{enumerate}
	\item[(1)] For any weight space $M_{\mu}$ of $M$, $B(M) \cap M_{\mu}$ is a basis of $M_{\mu}$;  
	\item[(2)] The $\mathcal{A}$-submodule $_{\mathcal{A}}M$ spanned by $B(M)$ is stable under $_{\mathcal{A}}\dot{\mathbf{U}}$;
	\item[(3)] The involution $\bar{ }:M \rightarrow M$ defined by $\bar{fb}=\bar{f}b$ for all $f\in \mathbb{Q}(v)$ satisfies $\bar{um}=\bar{u}\bar{m}$ for any $u\in \mathbf{U}$ and $m \in M$;
	\item[(4)] The $\mathbb{Q}[[v^{-1}]] \cap \mathbb{Q}(v)$-submodule $L(M)$ spanned by $B(M)$, and the images of $B(M)$ in $L(M)/v^{-1}L(M)$, induce a basis at $\infty$ of $M$.
\end{enumerate}

A morphism of based module $f:(M,B(M)) \rightarrow (M',B(M'))$ is a morphism $f:M \rightarrow M'$ of $\mathbf{U}$-modules such that 
\begin{enumerate}
	\item[(1)] For any $b \in B(M)$, $f(b) \in B(M') \cup \{0\}$;  
	\item[(2)] The set $B \cap \ker(f)$ is a basis of $\ker(f)$.
\end{enumerate}

\subsection{Canonical bases of tensor products}
Given a dominant integral weight $\lambda$ of $\mathfrak{g}$, the canonical bases $B(\lambda)$ of the irreducible highest weight $\mathbf{U}$-module $L(\lambda)$ is defined by 
$$ \{ \pi_{\lambda}(b)\neq 0| b \in B(\infty) \},  $$
where $B(\infty)$ is the canonical basis of $\mathbf{f} \cong {_{\mathcal{A}}\mathbf{U}}^{-}$ and $\pi_{\lambda}:\mathbf{U}^{-} \rightarrow  L(\lambda) $  is the canonical projection, sending the unit $1$ to the fixed highest weight vector $v_{\lambda}$. Then the pair $(L(\lambda),B(\lambda))$ is a based module.

Now we assume $(M,B(M))$ is a based module, and let $(M',B(M'))$ be the based module $(L(\lambda),B(\lambda))$. By \cite[27.3.1]{MR1227098}, the set  $B(M)\times B(\lambda)$ has a partial order $<$. Moreover, the quasi $\mathcal{R}$-matrix $\Theta$ acts on $M\otimes M'$ and by \cite[Proposition 2.4]{bao2016canonical}, it preserves the $\mathcal{A}$-submodule $_{\mathcal{A}}M \otimes {_{\mathcal{A}}M'}$. Let $\Psi: M \otimes M' \rightarrow M\otimes M' $ be the involution defined by $\Psi(m\otimes m')= \Theta (\bar{m}\otimes \bar{m'}),$ then we have the following theorem.

\begin{theorem}\cite{lusztig1992canonical}\cite[Theorem 2.7]{bao2016canonical}
	Let $\mathcal{L}(M\otimes M')$ be the $\mathbf{A}$-submodule of $M\otimes M'$  generated by $B \otimes B'$, then  the pair $(M\otimes M', B(M) \diamond B(M'))$ is a based module. More precisely, we have the following statements.
	
	(1) For any $(b,b') \in B(M)\times B(\lambda)$, there exists a unique element $b\diamond b' \in \mathcal{L}(M\otimes M')$  such that $\Psi(b\diamond b')=b\diamond b'$ and $b\diamond b'- b\otimes b' \in v^{-1}\mathcal{L}(M \otimes M')$.
	
	(2) There is a family of  $p_{b,b',b_{2},b'_{2}} \in v^{-1}\mathbf{A}$ such that $$b\diamond b'=\sum_{b_{2}\in B(M),b'_{2}\in B(\lambda)} p_{b,b',b_{2},b'_{2}}b_{2}\otimes b'_{2}.$$ Moreover, $p_{b,b',b,b'}=1 $ and  $p_{b,b',b_{2},b'_{2}} \neq 0$ implies that $(b_{2},b'_{2}) < (b,b')$.
	
	(3) The set $B(M) \diamond B(M')=\{b \diamond b' |(b,b') \in B(M)\times B(\lambda)\}$ forms an $\mathcal{A}$-basis of $_{\mathcal{A}}(M\otimes M')=\mathcal{A}\otimes_{\mathbf{A}} \mathcal{L}(M\otimes M'),$  an $\mathbf{A}$-basis of $\mathcal{L}(M\otimes M')$ and a $\mathbb{Q}(v)$-basis of $M\otimes M'$.
\end{theorem}

Let $\lambda^{\bullet}=(\lambda^{1},\lambda^{2},\cdots,\lambda^{N})$ be a sequence of dominant weights, and denote $L(\lambda^{\bullet})=L(\lambda^{N})\otimes L(\lambda^{N-1})\otimes \cdots\otimes L(\lambda^{1})$ by the tensor product. Define $\mathcal{L}(L(\lambda^{\bullet}))$ be the $\mathbf{A}$-submodule generated by $B(\lambda^{N}) \otimes B(\lambda^{N-1}) \otimes \cdots \otimes B(\lambda^{1}), $ then after using the above theorem inductively, we obtain the following theorem.

\begin{theorem}\cite{lusztig1992canonical}\cite[Theorem 2.9]{bao2016canonical} \label{CBT}
	(1) For any $b_{l} \in B(\lambda^{N})$, there exists a unique element $b_{N}\diamond b_{N-1} \diamond \cdots \diamond  b_{1} \in \mathcal{L}(L(\lambda^{\bullet}))$  such that $\Psi(b_{N}\diamond b_{N-1} \diamond \cdots \diamond b_{1})=b_{N}\diamond b_{N-1} \diamond \cdots \diamond b_{1}$ and $b_{N}\diamond b_{N-1} \diamond \cdots \diamond b_{1}- b_{N}\otimes b_{N-1} \otimes \cdots \otimes b_{1} \in v^{-1}\mathcal{L}(L(\lambda^{\bullet}))$.
	
	(2) There is a family of  $p^{b'_{1},b'_{2},\cdots,b'_{N} }_{b_{1},b_{2},\cdots,b_{N}} \in v^{-1}\mathbf{A}$ such that $$b_{N}\diamond b_{N-1} \diamond \cdots \diamond b_{1}=\sum_{b'_{l}\in B(\lambda^{l})} p^{b'_{1},b'_{2},\cdots,b'_{N} }_{b_{1},b_{2},\cdots,b_{N}}b'_{N}\diamond b'_{N-1} \diamond \cdots \diamond b'_{1}.$$ Moreover,   $p^{b_{1},b_{2},\cdots,b_{N} }_{b_{1},b_{2},\cdots,b_{N}}=1$.
	
	(3) The set  $B(\lambda^{\bullet})=\{b_{N}\diamond b_{N-1} \diamond \cdots \diamond  b_{1}| b_{l} \in B(\lambda^{l}),1\leqslant l \leqslant N\}$ forms an $\mathcal{A}$-basis of $_{\mathcal{A}}L(\lambda^{\bullet})=\mathcal{A}\otimes_{\mathbf{A}} \mathcal{L}(L(\lambda^{\bullet})),$  an $\mathbf{A}$-basis of $\mathcal{L}(L(\lambda^{\bullet}))$ and a $\mathbb{Q}(v)$-basis of $L(\lambda^{\bullet})$.
	
	(4)The natural map $\mathcal{L}(L(\lambda^{\bullet})) \cap \Psi(\mathcal{L}(L(\lambda^{\bullet}))) \rightarrow \mathcal{L}(L(\lambda^{\bullet}))/v^{-1}\mathcal{L}(L(\lambda^{\bullet}))$ is an isomorphism.
\end{theorem}
In particular, the set $B(\lambda^{\bullet})$ is called the canonical basis of the  tensor product $L(\lambda^{\bullet})$ by the authors in \cite{lusztig1992canonical} and \cite{bao2016canonical}.

\subsection{Based modules and their subquotients}

Given an integrable highest weight module $M$ of $\mathbf{U}(\mathfrak{g})$, we denote by $M[\lambda]$ be the sum of irreducible subobjects of $M$ which are isomorphic to $L(\lambda)$, then $M=\bigoplus_{\lambda}M[\lambda]$. We also define $M[\geqslant \lambda]=\bigoplus_{\lambda'\geqslant \lambda} M[\lambda']$ and $M[> \lambda]=\bigoplus_{\lambda'> \lambda} M[\lambda']$.

As mentioned in \cite[Remark 2.10]{bao2016canonical}, even though the main results of \cite[27.1.1-27.1.8 and 27.2.1-27.2.2]{MR1227098} are only proved for finite types, the following statements of based modules remain valid for quantum groups of Kac-Moody types.

\begin{proposition}\cite[27.1.8]{MR1227098}
	Let $(M,B)$ be a based module and $\lambda$ be an integral dominant weight. Then the set
	 $B\cap M[\geqslant \lambda]$ is a basis of $M[\geqslant \lambda]$, and the set $B\cap M[> \lambda]$ is a basis of $M[> \lambda]$.
\end{proposition}

With the notations above, for any $b \in B$, there is  a unique maximal integral dominant weight $\lambda$ such that $b \in M[\geqslant \lambda]$. We denote $B[\lambda]$ by the subset of $B$, which consists of $b \in B$ as above.

\begin{proposition}
	Let $f:(M,B(M)) \rightarrow (M',B(M)')$ be a morphism of based modules, then for any integral dominant weight $\lambda$, we have $f(B(M)[\lambda]) \subseteq B(M')[\lambda] \cup\{0\}.$
\end{proposition}

Recall that the space of coinvariants $M_{\ast}$ of $M$ is defined by $M_{\ast}=M/M[\neq 0]$, where $M[\neq 0] =\bigoplus _{\lambda \neq 0}M[\lambda].$ Then $\bigcup_{\lambda\neq 0} B[\lambda]$ is a basis of $M[\neq 0]$. Under the canonical projection $\pi: M \rightarrow M_{\ast}$, the subset $B[0]$ is mapped bijectively onto a basis $B_{\ast}$ of $M_{\ast}$.  Moreover, the pair $(M_{\ast},B_{\ast})$ becomes a based module with trivial $\mathbf{U}$-action, and $\pi:(M,B) \rightarrow (M_{\ast},B_{\ast})$ is a morphism of based modules.

In general, for any dominant weight $\lambda$, let $\pi_{\geqslant \lambda}:M[\geqslant \lambda] \rightarrow M[\geqslant \lambda]/M[> \lambda]  $ be the projection map, then the nonzero image of $B(M) \cap M[\geqslant \lambda]$ under $\pi_{\geqslant \lambda}$ form a basis of $M[\geqslant \lambda]/M[> \lambda]$, which makes $M[\geqslant \lambda]/M[> \lambda]$ a based module.

\section{Realization of integrable highest weight  modules}
In this section, we recall the sheaf theoretic construction of tensor products in \cite{fang2023tensor}. We fix a positive integer $N$. 

\subsection{The $N$-framed quiver}
For a given quiver $Q=(I,H,\Omega)$, the $N$-framed quiver  $Q^{(N)}=(I^{(N)},H^{(N)},\Omega^{(N)})$ associated to $Q$ is defined as the following.
\begin{enumerate}
	\item[(1)] The set of vetices $I^{(N)}=I \cup I^{1} \cup \cdots I^{N}$ is obtained by adding $N$-copies of $I$ from $I$, here each $I^{l}$ is the $l$-th copy of $I$. We also denote $i^{l}$ by the $l$-th copy of $i \in I$.
	\item[(2)] The set of oriented arrows $\Omega^{(N)}=\Omega \cup \{i \rightarrow i^{l}|i\in I,1\leqslant l \leqslant N \} $ is obtained by adding arrows  $i \rightarrow i^{l}$ to $\Omega$. 
	\item [(3)] The set of disoriented edges $H^{(N)}$ is defined by $H^{(N)}=\Omega^{(N)} \cup \overline{\Omega^{(N)}}$, where  $~\bar{ }~$  is the involution taking the opposite orientation for arrows.
\end{enumerate}

Given a sequence of integral dominant weights $\lambda^{\bullet}=(\lambda^{1},\lambda^{2},\cdots,\lambda^{N}), $ we assume that $\lambda^{l}=\sum_{i \in I}\lambda^{l}_{i} \Lambda_{i}$, where $\Lambda_{i}$ is the $i$-th fundamental weight of $\mathfrak{g}$. Take $I^{l}$-graded space $\mathbf{W}^{l}$ such that $\dim \mathbf{W}^{l}_{i^{l}}= \lambda^{l}_{i} $, we call each $\mathbf{W}^{l}$ a framing of $Q$. For any $I$-graded space $\mathbf{V}$ with dimension vector $\nu$, the moduli space of the $N$-framed quiver with the framing $\mathbf{W}^{\bullet}=\bigoplus_{1\leqslant l\leqslant N}\mathbf{W}^{l}$ is defined by
$$\mathbf{E}_{\mathbf{V},\mathbf{W}^{\bullet},\Omega^{(N)}}=\bigoplus_{h \in \Omega}\mathbf{Hom}(\mathbf{V}_{h'},\mathbf{V}_{h''}) \oplus \bigoplus_{i \in I,1\leqslant l\leqslant N}\mathbf{Hom}(\mathbf{V}_{i},\mathbf{W}^{l}_{i^{l}}). $$
The group $G_{I}=(\mathbb{G}_{m})^{I}$ acts on $\bigoplus_{i \in I,1\leqslant l\leqslant N}\mathbf{Hom}(\mathbf{V}_{i},\mathbf{W}^{l}_{i^{l}})$ by scaling
$$ (t_{i})_{i \in I} \cdot (y^{l}_{i})_{i\in I,1\leqslant l \leqslant N} = (t_{i}y^{l}_{i})_{i\in I,1\leqslant l \leqslant N}, $$
and the algebraic group $G_{\mathbf{V}}=\prod_{i \in I}GL(\mathbf{V}_{i})$ acts naturally on $\mathbf{E}_{\mathbf{V},\mathbf{W}^{\bullet},\Omega^{(N)}}$ by taking composition
$$(g_{i})_{i\in I} \cdot (x_{h},y^{l}_{i})_{h\in \Omega,i\in I,1\leqslant l \leqslant N} = (g_{h''}x_{h}g^{-1}_{h'}, y_{i}^{l}g^{-1}_{i} )_{h\in \Omega,i\in I,1\leqslant l \leqslant N}. $$
Here we need the $G_{I}$-equivariant structure because we want to use monodromic structures to define Fourier-Sato transforms reversing arrows between $i$ and  $i^{l}, 1\leqslant l \leqslant N$. 

We consider the (bounded) $G_{I}\times G_{\mathbf{V}}$-equivariant derived category  of constructible $\mathbb{C}$-sheaves on $\mathbf{E}_{\mathbf{V},\mathbf{W}^{\bullet},\Omega^{(N)}} $  and denote it by $\mathcal{D}_{\nu}(\lambda^{\bullet})=\mathcal{D}^{b}_{G_{I}\times G_{\mathbf{V}}} (\mathbf{E}_{\mathbf{V},\mathbf{W}^{\bullet},\Omega^{(N)}})$. We denote $\mathbf{D}$ by the Verdier duality and $[1]$ by the shift functor. Note that $\mathcal{D}_{\nu}(\lambda^{\bullet})$ has a perverse $t$-structure.

Assume the set of vertices $I$ is numbered by $I=\{i_{1},i_{2},\cdots, i_{n}\}$, take $\boldsymbol{d}^{l}=(\lambda^{l}_{i_{1}}i^{l}_{1},\lambda^{l}_{i_{2}}i^{l}_{2},\cdots,\lambda^{l}_{i_{n}}i^{l}_{n})$ be the ordered dimension vector of $\mathbf{W}^{l}$ for any $1\leqslant l \leqslant N$. For each $l$ we take a sequence $\boldsymbol{\nu}^{l}=(\nu^{l}_{1},\nu^{l}_{2},\cdots,\nu^{l}_{s_{l}}),1\leqslant l \leqslant N$ of dimension vectors in $\mathbb{N}I$ such that every  $\nu^{l}_{r}$ equals to some $ai$ with $a \in \mathbb{N}_{>0}, i \in I$, we denote $|\boldsymbol{\nu}^{l}|=\sum_{r} \nu^{l}_{r} \in \mathbb{N}I$. If the dimension vector of $\mathbf{V}$ is $\nu$ and $\sum_{1\leqslant l\leqslant N}|\boldsymbol{\nu}^{l}|=\nu$, the sequence $(\boldsymbol{\nu}^{1}\boldsymbol{d}^{1}\boldsymbol{\nu}^{2}\boldsymbol{d}^{2}\cdots\boldsymbol{\nu}^{N}\boldsymbol{d}^{N})$ is a flag type of $\mathbf{V}\oplus \mathbf{W}^{\bullet}$. In particular, by \cite[2.2]{MR1088333}, there is a semisimple complex $L_{ \boldsymbol{\nu}^{1}\boldsymbol{d}^{1}\boldsymbol{\nu}^{2}\boldsymbol{d}^{2}\cdots\boldsymbol{\nu}^{N}\boldsymbol{d}^{N}}$ in $\mathcal{D}_{\nu}(\lambda^{\bullet})$.

\begin{definition}
		Let $\mathcal{P}_{\nu}(\lambda^{\bullet})$  be the set of simple perverse sheaves which appear as shifted direct summands in semisimple complexes of the form $L_{ \boldsymbol{\nu}^{1}\boldsymbol{d}^{1}\boldsymbol{\nu}^{2}\boldsymbol{d}^{2}\cdots\boldsymbol{\nu}^{N}\boldsymbol{d}^{N}},$ we define $\mathcal{Q}_{\nu}(\lambda^{\bullet})$ to be the full subcategory of $\mathcal{D}_{\nu}(\lambda^{\bullet})$ which consists of finite direct sums of shifted objects in $\mathcal{P}_{\nu}(\lambda^{\bullet})$. We call $\mathcal{Q}_{\nu}(\lambda^{\bullet})$ the category of Lusztig's sheaves for $N$-framed quiver. Together with the shift functor, Lusztig's sheaves form a graded linear category.
\end{definition}

\subsection{Lusztig's induction and restriction functors}

Given graded spaces $\bV,\bV'$ and $\bV''$ such that $\bV=\bV'\oplus \bV''$, denote their dimension vectors by $\nu,\nu'$ and $\nu''$ respectively. 

Let $F$ be the closed subset of $\mathbf{E}_{\bV,\mathbf{W}^{\bullet},\Omega^{(N)}}$ consisting of $(x,y)$ such that $(x,y)(\bV''\oplus \mathbf{W}^{\bullet})\subseteq \bV''\oplus \mathbf{W}^{\bullet}$, and let $P\subseteq G_{\bV}$ be the stabilizer of $\bV''$ which is a parabolic subgroup, and let $U\subseteq P$ be the unipotent radical. Consider the following diagrams
\begin{equation}\label{indd}
	\mathbf{E}_{\mathbf{V}',\Omega} \times \mathbf{E}_{\mathbf{V}'',\mathbf{W}^{\bullet},\Omega^{(N)}} \xleftarrow{p_{1}} G_{\bV} \times^{U} F \xrightarrow{p_{2}} G_{\bV} \times^{P} F \xrightarrow{p_{3}} \mathbf{E}_{\mathbf{V},\mathbf{W}^{\bullet},\Omega^{(N)}},
\end{equation}
where $p_{1}(g,x,y) =\kappa (x,y), p_{2}(g,x,y)=(g,x,y), p_{3}(g,x,y)=g(\iota(x,y))$, and $\iota:F \rightarrow \mathbf{E}_{\mathbf{V},\mathbf{W}^{\bullet},\Omega^{(N)}}$ is the closed embedding. The induction functor functor is defined by 
\begin{align*}
	\mathbf{Ind}^{\mathbf{V}\oplus \mathbf{W}^{\bullet}}_{\mathbf{V}',\mathbf{V}''\oplus \mathbf{W}^{\bullet}}: 
	&\mathcal{D}^b_{G_{\bV'}\times G_{I}}(\mathbf{E}_{\mathbf{V}',\Omega}) \boxtimes \mD_{\nu''}(\lambda^{\bullet})\rightarrow  \mD_{\nu}(\lambda^{\bullet}),\\
	&(A\boxtimes B) \mapsto (p_{3})_{!}(p_{2})_{\flat}(p_{1})^{\ast}(A\boxtimes B)[d_{1}-d_{2}],
\end{align*}
where $(p_{2})_{\flat}$ is the equivariant decent functor of the principal bundle $p_2$ and $d_1,d_2$ are the dimension of the fibers of $p_1,p_2$ respectively.

Let $\bullet^{1}$ be a subset of $\{1,2,\cdots, N\}$ and $\bullet^{2}$ be its complement, denote $\bigoplus_{l \in \bullet^{1} }\mathbf{W}^{l}$  and  $\bigoplus_{l \in \bullet^{2} }\mathbf{W}^{l}$ by $\mathbf{W}^{\bullet,1}$ and $\mathbf{W}^{\bullet,2}$ respectively. Let $F'$ be  the closed subset of $\mathbf{E}_{\bV,\mathbf{W}^{\bullet},\Omega^{(N)}}$ consisting of $(x,y)$ such that $(x,y)(\bV''\oplus \mathbf{W}^{\bullet,2})\subseteq \bV''\oplus \mathbf{W}^{\bullet,2}$. For any $(x,y)\in F'$, we denote the restriction of $(x,y)$ on $\bV'\oplus \mathbf{W}^{\bullet,1}$ by $(x',y')$, denote the restriction of $(x,y)$ on $\bV''\oplus \mathbf{W}^{\bullet,2}$ by $(x'',y'')$. Consider the following diagram
\begin{equation}\label{resd}
	\mathbf{E}_{\mathbf{V}',\mathbf{W}^{\bullet,1},\Omega^{(N)}} \times \mathbf{E}_{\mathbf{V}'',\mathbf{W}^{\bullet,2},\Omega^{(N)}} \xleftarrow{\kappa } F' \xrightarrow{\iota'} \mathbf{E}_{\mathbf{V},\mathbf{W}^{\bullet},\Omega^{(N)}},
\end{equation}
where $\iota'(x,y)=(x,y), \kappa(x,y)=(x',y',x'',y'')$. The restriction functor $$\mathbf{Res}^{\mathbf{V}\oplus\mathbf{W}^{\bullet}}_{\mathbf{V}'\oplus\mathbf{W}^{\bullet,1},\mathbf{V}''\oplus \mathbf{W}^{\bullet,2}}:\mD_{\nu}(\lambda^{\bullet}) \rightarrow \mathcal{D}^{b}_{G_{\mathbf{V}'}\times G_{I} \times G_{\mathbf{V}''} \times G_{I}}(\mathbf{E}_{\mathbf{V}',\mathbf{W}^{\bullet,1},\Omega^{(N)}}\times \mathbf{E}_{\mathbf{V}'',\mathbf{W}^{\bullet,2},\Omega^{(N)}})$$ is defined by
\begin{align*}
C \mapsto (\kappa)_{!} (\iota')^{\ast}(C)[-\langle\nu',\nu''\rangle_{Q^{(N)}}],
\end{align*}
where $\langle \nu',\nu''\rangle_{Q^{(N)}}=\sum_{i\in I^{(N)}}\nu'_i\nu''_i-\sum_{h\in \Omega^{(N)}}\nu'_{h'}\nu''_{h''}$ is the Euler form of the quiver $Q^{(N)}$.

\subsection{The thick subcategory $\mathcal{N}$ and localization}
For each $i \in I$, we fix an orientation $\Omega_{i}$ such that $i$ is a source in $\Omega_{i}$, then $\mathbf{E}_{\mathbf{V},\mathbf{W}^{\bullet},\Omega_{i}^{(N)}}$ admits a partition
$\mathbf{E}_{\mathbf{V},\mathbf{W}^{\bullet},\Omega_{i}^{(N)}}= \bigcup_{r} \mathbf{E}^{r}_{\mathbf{V},\mathbf{W}^{\bullet},i},  $
where
\begin{equation}\label{thick}
	\begin{split}
			\mathbf{E}^{r}_{\mathbf{V},\mathbf{W}^{\bullet},i}=\{ (x,y)\in \mathbf{E}_{\mathbf{V},\mathbf{W}^{\bullet},\Omega_{i}^{(N)}}|&{{\rm{dim}}}{\rm{ker}}((\bigoplus _{h \in \Omega_{i}, h'=i} x_{h} ) \oplus (\bigoplus_{1\leqslant l \leqslant N}y^{l}_{i} ) ):\\
			& \mathbf{V}_{i} \rightarrow \bigoplus_{h \in \Omega_{i},h'=i}\mathbf{V}_{h''}\oplus\bigoplus_{1\leqslant l \leqslant N} \mathbf{W}^{l}_{i^{l}})=r \}. 
	\end{split}
\end{equation} 

Let $\mathcal{N}_{\nu,i}$ be the full subcategory of $\mathcal{D}^{b}_{G_{I}\times G_{\mathbf{V}}} (\mathbf{E}_{\mathbf{V},\mathbf{W}^{\bullet},\Omega_{i}^{(N)}})$, which consists of objects whose supports are contained in the closed subset $\mathbf{E}^{\geqslant 1}_{\mathbf{V},\mathbf{W}^{\bullet},i}=\bigcup_{r\geqslant 1} \mathbf{E}^{r}_{\mathbf{V},\mathbf{W}^{\bullet},i}, $ then $\mathcal{N}_{\nu,i}$ is a thick subcategory of $\mathcal{D}^{b}_{G_{I}\times G_{\mathbf{V}}} (\mathbf{E}_{\mathbf{V},\mathbf{W}^{\bullet},\Omega_{i}^{(N)}})$.

Following \cite[Section 3.3]{fang2025lusztigsheavescharacteristiccycles}, we say two orientations are equivalent under mutation if they  are related to each other by a sequence of quiver mutations at sink or source. Recall that the Fourier-Sato transform $$\mathbf{Four}_{\Omega,\Omega'}:\mathcal{D}^{b}_{G_{I}\times G_{\mathbf{V}}} (\mathbf{E}_{\mathbf{V},\mathbf{W}^{\bullet},\Omega^{(N)}})  \rightarrow \mathcal{D}^{b}_{G_{I}\times G_{\mathbf{V}}} (\mathbf{E}_{\mathbf{V},\mathbf{W}^{\bullet},\Omega^{',(N)}}) $$ is defined for two equivalent orientations in \cite[Section 3.3 and 4.1]{fang2025lusztigsheavescharacteristiccycles}.  (One can also see details in \cite[Section 2.7]{AHJR} and \cite[Appendix D]{HLSS}.)  By similar arguments in \cite[Proposition 10.4.5]{MR4337423}, the Fourier-Sato transform commutes with Lusztig's induction functor and induces perverse equivalences. In particular, if we take an orientation $\Omega_{i}$ such that $i$ is a source in $\Omega_{i}$ and $\Omega^{(N)}_{i}$ and $\Omega^{(N)}$ are  equivalent under mutations, $\mathbf{Four}_{\Omega_{i},\Omega}  ( \mathcal{N}_{\nu,i})$ is a thick subcategory of $\mD_{\nu}(\lambda^{\bullet} )$. (Such $\Omega_{i}$ always exists by \cite[Lemma 3.4]{fang2025lusztigsheavescharacteristiccycles}.)

\begin{definition}
	We define $\mN_{\nu}$ to be the thick subcategory of $\mD_{\nu}(\lambda^{\bullet} )$  generated by  $\mathbf{Four}_{\Omega_{i},\Omega} ( \mathcal{N}_{\nu,i}),i \in I$, and define $\mD_{\nu}(\lambda^{\bullet})/\mN_{\nu}$ to be the Verdier quotient of $\mD_{\nu}(\lambda^{\bullet})$ with respect to the thick subcategory $\mN_{\nu}$.
	
	The localization $\mL_{\nu}(\lambda^{\bullet})=\mQ_{\nu}(\lambda^{\bullet})/\mN_{\nu}$ is defined to be the full subcategory of $\mD_{\nu}(\lambda^{\bullet})/\mN_{\nu}$, which consists of those objects that are isomorphic to objects of $\mQ_{\nu}(\lambda^{\bullet})$ in $\mD_{\nu}(\lambda^{\bullet})/\mN_{\nu}$.
\end{definition}

Let $\mK(\nu,\lambda^{\bullet})=\mK(\mL_{\nu}(\lambda^{\bullet}))$ be the Grothendieck group of $\mL_{\nu}(\lambda^{\bullet})$, that is, a $\mathbb{Z}[v^{\pm}]$-module spanned by the isomorphism classes $[L]$ of objects $L\in \mL_{\nu}(\lambda^{\bullet})$ subject to the relations
$$[L \oplus L']=[L]+[L'],v[L]=[L[1]].$$
We also denote $\bigoplus_{\nu \in \mathbb{N}I} \mK(\nu,\lambda^{\bullet})$ by $\mK(\lambda^{\bullet})$.

\subsection{The realization of integrable modules}

\begin{definition}
	For any $i\in I,r\in \mathbb{N}$ and $\nu=ri+\nu''\in\mathbb{N}I$, we define the functor $\mathcal{F}^{(r)}_{i}$ by
	\begin{align*}
		\mathcal{F}^{(r)}_{i}&=\mathbf{Ind}^{\mathbf{V}\oplus \mathbf{W}^{\bullet}}_{\mathbf{V}',\mathbf{V}''\oplus \mathbf{W}^{\bullet}}( \mathbb{C}_{\mathbf{E}_{\mathbf{V}',\Omega}}\boxtimes-):\mD_{\nu''}(\lambda^{\bullet}) \rightarrow \mD_{\nu}(\lambda^{\bullet}).
	\end{align*}
	For simplicity, we also denote $\mathcal{F}^{(1)}_{i}$ by $\mathcal{F}_{i}$.
\end{definition}

Let $i\in I$ be a source for the orientation $\Omega_{i}$. For any $r\in \mathbb{N}$ and $\nu=ri+\nu''\in\mathbb{N}I$, consider the following diagram.
\begin{equation}\label{UD}
	\begin{split}
		\xymatrix{
			\mathbf{E}_{\mathbf{V},\mathbf{W}^{\bullet},\Omega_{i}^{(N)}}
			&
			& 	\mathbf{E}_{\mathbf{V}'',\mathbf{W}^{\bullet},\Omega_{i}^{(N)}} \\
			\mathbf{E}^{0}_{\mathbf{V},\mathbf{W}^{\bullet},i} \ar[d]_{\phi_{\mathbf{V},i}} \ar[u]^{j_{\mathbf{V},i}}
			&
			& 	\mathbf{E}^{0}_{\mathbf{V}'',\mathbf{W}^{\bullet},i} \ar[d]^{\phi_{\mathbf{V}'',i}} \ar[u]_{j_{\mathbf{V}'',i}} \\
			\txt{$\dot{\mathbf{E}}_{\mathbf{V},\mathbf{W}^{\bullet},i}$\\ $\times$ \\ $\mathbf{Grass}(\nu_i, \tilde{\nu}_{i})$}
			& \txt{$\dot{\mathbf{E}}_{\mathbf{V},\mathbf{W}^{\bullet},i}$\\ $\times$ \\$\mathbf{Flag}(\nu''_{i},\nu_{i},\tilde{\nu}_{i})$ } \ar[r]^-{q_{2}} \ar[l]_-{q_{1}}
			&  \txt{$\dot{\mathbf{E}}_{\mathbf{V}'',\mathbf{W}^{\bullet},i}$\\  $\times$ \\ $\mathbf{Grass}(\nu''_{i}, \tilde{\nu}_{i})$},
		}
	\end{split}
\end{equation}

where $\mathbf{E}^{0}_{\mathbf{V},\mathbf{W}^{\bullet},i},\mathbf{E}^{0}_{\mathbf{V}'',\mathbf{W}^{\bullet},i}$ are defined before in (\ref{thick}), and $j_{\mathbf{V},i},j_{\mathbf{V}'',i}$ are the open embeddings;
\begin{align*}
\dot{\mathbf{E}}_{\mathbf{V},\mathbf{W}^{\bullet},i} &=\bigoplus_{h \in \Omega_{i}, h'\neq i} \mathbf{Hom}(\mathbf{V}_{h'},\mathbf{V}_{h''}) \oplus \bigoplus_{j\in I,j\not=i,1\leqslant l \leqslant N} \mathbf{Hom}(\mathbf{V}_{j},\mathbf{W}^{l}_{j^{l}}),
\end{align*}
and $\mathbf{Grass}(\nu_i, \tilde{\nu}_{i})$ and $\mathbf{Grass}(\nu''_{i}, \tilde{\nu}_{i})$ are the Grassmannian varieties consisting of $\nu_{i}$-dimensional and $\nu''_i$-dimensional subspaces respectively, of  the space
$$\bigoplus_{h\in \Omega_{i}, h'=i}\mathbf{V}_{h''}\oplus \bigoplus_{1\leqslant l \leqslant N}\mathbf{W}^{l}_{i^{l}},$$
which is of dimension $\tilde{\nu}_{i}=\sum_{h\in \Omega,h'=i}\nu_{h''}+\sum_{1\leqslant l \leqslant N} \lambda^{l}_{i}$, and 
\begin{align*}
	\phi_{\mathbf{V},i}((x,y))=((x_{h},y^{l}_{j})_{h' \neq i,j \neq i}, {\rm{Im}}  (\bigoplus _{h \in \Omega_{i}, h'=i;1\leqslant l \leqslant N} x_{h}\oplus y^{l}_{i})),
\end{align*}
which is a principal $\mathrm{GL}(\bV_{i})$-bundle. The variety  $\dot{\mathbf{E}}_{\mathbf{V}'',\mathbf{W}^{\bullet},i}$ and the morphism $\phi_{\mathbf{V}'',i}$ are defined in a similar way. The variety $\mathbf{Flag}(\nu''_{i},\nu_{i},\tilde{\nu}_{i})$ is the flag variety, and $q_{1}, q_{2}$ are natural projections, which are smooth and proper.

\begin{definition}
	For any $i\in I,r\in \mathbb{N}$ and $\nu=ri+\nu''\in\mathbb{N}I$, we define the functor 
		$\mathcal{E}^{(r)}_{i}:\mD_{\nu}(\lambda^{\bullet}) \rightarrow \mD_{\nu''}(\lambda^{\bullet})$ by  the following composition
\begin{align*}
	 \mathbf{Four}_{\Omega_i^{(N)},\Omega^{(N)}} \circ ((j_{\mathbf{V}'',i})_{!} (\phi_{\mathbf{V},i})^{\ast} (q_{2})_{!}(q_{1})^{\ast} (\phi_{\mathbf{V},i})_{\flat}(j_{\mathbf{V},i})^{\ast}[-r\nu_{i}]) \circ \mathbf{Four}_{\Omega^{(N)},\Omega_{i}^{(N)}}.
\end{align*}
We also denote $\mathcal{E}^{(1)}_{i}$ by $\mathcal{E}_{i}$. Note that $\mathcal{E}^{(r)}_{i}$ does not depend on the choice of the orientation $\Omega_{i}$.
\end{definition}

By \cite[Proposition 3.18, Corollary 3.19]{fang2023lusztig}, the functors $\mathcal{E}^{(r)}_{i}$ and $\mathcal{F}^{(r)}_{i}$ are well-defined on the Verdier quotients
$$\mathcal{E}^{(r)}_{i}:\mD_{\nu}(\lambda^{\bullet})/\mN_{\nu} \rightarrow \mD_{\nu''}(\lambda^{\bullet})/\mN_{\nu''},$$
$$\mathcal{F}^{(r)}_{i}:\mD_{\nu''}(\lambda^{\bullet})/\mN_{\nu''} \rightarrow \mD_{\nu}(\lambda^{\bullet})/\mN_{\nu},$$
and induces functors  of localizations
$$\mathcal{E}^{(r)}_{i}:\mL_{\nu}(\lambda^{\bullet}) \rightarrow \mL_{\nu''}(\lambda^{\bullet}),$$
$$\mathcal{F}^{(r)}_{i}:\mL_{\nu''}(\lambda^{\bullet}) \rightarrow \mL_{\nu}(\lambda^{\bullet}).$$

\begin{definition}
	We define the functors $\mathcal{K}_i,\mathcal{K}_{-i}$ to be the shift functors
	\begin{align*}
		&\mathcal{K}_{i}={\mathrm{Id}}[-2\nu_{i}+\sum_{h \in \Omega_{i}, h'=i} \nu_{h''} +\sum_{1\leqslant l \leqslant N}\lambda^{l}_{i}]:\mL_{\nu}(\lambda^{\bullet}) \rightarrow \mL_{\nu}(\lambda^{\bullet}),\\
		&\mathcal{K}_{-i}={\mathrm{Id}}[2\nu_{i}-\sum_{h \in \Omega_{i}, h'=i} \nu_{h''} -\sum_{1\leqslant l \leqslant N}\lambda^{l}_{i}]:\mL_{\nu}(\lambda^{\bullet}) \rightarrow \mL_{\nu}(\lambda^{\bullet}).
	\end{align*}
\end{definition}

\begin{proposition}\label{c1}
	As endofunctors of $\coprod_{\nu \in \mathbb{N}I} \mL_{\nu}(\lambda^{\bullet}) $, the functors $\mathcal{E}_{i}$, $\mathcal{F}_{i}$ and $\mathcal{K}_{i},i\in I$ satisfy the following relations
	\begin{equation*} 
		\mathcal{K}_{i}\mathcal{K}_{j}=\mathcal{K}_{j}\mathcal{K}_{i},
	\end{equation*}
	\begin{equation*}
		\mathcal{E}_{i}\mathcal{K}_{j}=\mathcal{K}_{j}\mathcal{E}_{i}[-a_{j,i}],
	\end{equation*}
	\begin{equation*}
		\mathcal{F}_{i}\mathcal{K}_{j}=\mathcal{K}_{j}\mathcal{F}_{i}[a_{i,j}],
	\end{equation*}
	\begin{equation*}
		\mathcal{E}^{(r)}_{i}\mathcal{F}^{(s)}_{j}=\mathcal{F}^{(s)}_{j}\mathcal{E}^{(r)}_{i}\ \textrm{for}\ i \neq j.
	\end{equation*}
	As endofunctors of $ \mL_{\nu}(\lambda^{\bullet})$, we have
	\begin{equation*}
		\mathcal{E}_{i}\mathcal{F}_{i} \oplus \bigoplus\limits_{0\leqslant m \leqslant M-1} Id[M-1-2m] \cong \mathcal{F}_{i}\mathcal{E}_{i} \oplus \bigoplus\limits_{0\leqslant m \leqslant -M-1} Id[-2m-M-1],
	\end{equation*}
	where $M= 2\nu_{i}-\tilde{\nu}_{i}$. Recall that $\tilde{\nu}_{i}=\sum_{h\in H,h'=i}\nu_{h''}+\sum_{1\leqslant l \leqslant N} \lambda^{l}_{i}$.
\end{proposition}
\begin{proposition}\label{c2}
	These functors  $\mathcal{E}^{(r)}_{i}$ and $\mathcal{F}^{(r)}_{i}$ satisfy the Serre relations and the relation of divided powers.
	\begin{equation*}
		\bigoplus\limits_{0\leqslant m \leqslant 1- a_{i,j},m~odd}\mathcal{E}^{(m)}_{i}\mathcal{E}_{j}\mathcal{E}^{(1-a_{i,j}-m)}_{i} \cong 	\bigoplus\limits_{0\leqslant m \leqslant 1- a_{i,j},m~even}\mathcal{E}^{(m)}_{i}\mathcal{E}_{j}\mathcal{E}^{(1-a_{i,j}-m)}_{i},
	\end{equation*}
	\begin{equation*}
		\bigoplus\limits_{0\leqslant m \leqslant 1-a_{i,j},m~ odd}\mathcal{F}^{(m)}_{i}\mathcal{F}_{j}\mathcal{F}^{(1-a_{i,j}-m)}_{i}\cong 	\bigoplus\limits_{0\leqslant m \leqslant 1-a_{i,j},m~ even}\mathcal{F}^{(m)}_{i}\mathcal{F}_{j}\mathcal{F}^{(1-a_{i,j}-m)}_{i},
	\end{equation*}
	\begin{equation*}
		\bigoplus \limits_{0 \leqslant m < r } \mathcal{E}^{(r)}_{i}[r-1-2m] \cong \mathcal{E}^{(r-1)}_{i}\mathcal{E}_{i}\cong \mathcal{E}_{i}\mathcal{E}^{(r-1)}_{i}, \textrm{for}\ r \geqslant 1,
	\end{equation*}
	\begin{equation*}
		\bigoplus \limits_{0 \leqslant m < r } \mathcal{F}^{(r)}_{i}[r-1-2m] \cong \mathcal{F}^{(r-1)}_{i}\mathcal{F}_{i}\cong \mathcal{F}_{i} \mathcal{F}^{(r-1)}_{i}\ \textrm{for}\ r \geqslant 1.
	\end{equation*}
	Here $\mathcal{F}^{(0)}_{i}=\mathcal{E}^{(0)}_{i}$ is defined to be the identity functor.
\end{proposition}

\begin{theorem}[{\cite[Theorem 3.28]{fang2023lusztig}}]\label{high}
	When $N=1$, we denote $\mK(\lambda^{\bullet})$ by $\mK(\lambda)$, then $\mK(\lambda)$ becomes an integrable highest weight module together with the functors $\mathcal{E}^{(r)}_i,\mathcal{F}^{(r)}_i,\mathcal{K}^{\pm}_i$ for $i\in I,r\in \mathbb{N}_{>0}$. There is a unique canonical isomorphism $$\chi^{\lambda}: \mK(\lambda) \rightarrow  {_{\mathcal{A}}L}_{v}(\lambda),$$ which sends the image of constant sheaf $\mathbb{C}_{\mathbf{E}_{0,\mathbf{W},\Omega}^{(1)}}$ to the fixed highest weight vector. Moreover, the images of nonzero simple perverse sheaves in $\mL_{\nu}(\lambda),\nu\in\mathbb{N}I$ under $\chi^{\lambda}$ form an $\mathcal{A}$-basis of ${_{\mathcal{A}}L}_{v}(\lambda)$, which coincides with the canonical basis. 
\end{theorem}

For general $N$, we take $\bullet^{1}=\{1\}$, $\bullet^{2}=\{2,3\cdots,N\}$ and consider  the following composition 
\begin{align*}
	&\mathcal{D}^b_{G_\bV\times G_{I}}(\mathbf{E}_{\bV,\mathbf{W}^{\bullet},\Omega^{(N)}})\xrightarrow{\mathbf{D}}\mathcal{D}^b_{G_\bV\times G_{I}}(\mathbf{E}_{\bV,\mathbf{W}^{\bullet},\Omega^{(N)}})\xrightarrow{\mathbf{Res}^{\mathbf{V}\oplus\mathbf{W}^{\bullet}}_{\mathbf{V}'\oplus\mathbf{W}^{\bullet,1},\mathbf{V}''\oplus \mathbf{W}^{\bullet,2}}}\\
	&\mathcal{D}^b_{G_{\bV'}\times G_{I}\times G_{\bV''}\times G_{I}}(
	\mathbf{E}_{\mathbf{V}',\mathbf{W}^{\bullet,1},\Omega^{(1)}}\times \mathbf{E}_{\mathbf{V}'', \mathbf{W}^{\bullet,2},\Omega^{(N-1)}})
	\xrightarrow{\mathbf{D} \boxtimes \mathbf{D}}\\
	&\mathcal{D}^b_{G_{\bV'}\times G_{I}\times G_{\bV''}\times G_{I}}(
	\mathbf{E}_{\mathbf{V}',\mathbf{W}^{\bullet,1},\Omega^{(1)}}\times \mathbf{E}_{\mathbf{V}'', \mathbf{W}^{\bullet,2},\Omega^{(N-1)}})
	\xrightarrow{\mathrm{sw}}\\
	&\mathcal{D}^b_{G_{\bV''}\times G_{I}\times G_{\bV'}\times G_{I}}(
	\mathbf{E}_{\mathbf{V}'',\mathbf{W}^{\bullet,2},\Omega^{(N-1)}}\times \mathbf{E}_{\mathbf{V}', \mathbf{W}^{\bullet,1},\Omega^{(1)}}).
\end{align*}
Take direct sum $\bigoplus_{\bV',\bV''}$ of the composition above, we obtain a well-defined $\mathcal{A}$-linear map $\Delta_{N}$ on the Grothendieck groups
$$\Delta_{N}: \mK(\lambda^{\bullet})\rightarrow \mK(\lambda^{\bullet^{2}}) \otimes_{\mathcal{A}}  \mK(\lambda^{1}). $$
The map $\Delta_{N}: \mK(\lambda^{\bullet})\rightarrow \mK(\lambda^{\bullet^{2}}) \otimes_{\mathcal{A}}  \mK(\lambda^{1})$ is an $\mathcal{A}$-linear isomorphism and allows us reduce $N$-framed quivers to $(N-1)$-framed quivers.

\begin{theorem}{\cite[Theorem 4.15]{fang2023tensor}} \label{tensor}\
	
  (1)With the actions induced from the functors $\mathcal{E}^{(r)}_i,\mathcal{F}^{(r)}_i,\mathcal{K}^{\pm}_i$ for $i\in I,r\in \mathbb{N}_{>0}$, the Grothendieck group $\mK(\lambda^{\bullet})$ becomes an integrable highest weight module. 
  
  (2)The morphism $\Delta_{N}: \mK(\lambda^{\bullet})\rightarrow \mK(\lambda^{\bullet^{2}}) \otimes_{\mathcal{A}}  \mK(\lambda^{1})$ is an isomorphism of $_{\mathcal{A}}\mathbf{U}$-modules.  In particular, the Grothendieck group $\mK(\lambda^{\bullet})$ is canonically isomorphic to the tensor product ${_{\mathcal{A}}L}(\lambda^{\bullet})$ via the composition $$\chi^{\lambda^{\bullet}}= (\chi^{\lambda^{N}}\otimes \chi^{\lambda^{N-1}} \otimes \cdots \otimes \chi^{\lambda^{1}} )\circ (id^{\otimes (N-2)} \otimes \Delta_{2})\circ (id^{\otimes (N-3)} \otimes\Delta_{3})\circ \cdots  \circ \Delta_{N}. $$
  
  (3)The images of nonzero simple perverse sheaves in $\mL_{\nu}(\lambda^{\bullet}),\nu\in\mathbb{N}I$ under $\chi^{\lambda^{\bullet}}$ form an $\mathcal{A}$-basis of the tensor product, which coincides with the canonical basis in Theorem \ref{CBT}. 
\end{theorem}

In \cite{fang2023tensor}, we only compare the basis provided by simple perverse sheaves and the canonical basis defined in Theorem \ref{CBT} for $N=2$ case, but the proof indeed works for general $N$.

\begin{remark}
	We can consider the vector bundle $$\mathbf{E}_{\bV,\mathbf{W}^{\bullet},\Omega^{(N)}} \rightarrow \mathbf{E}_{\bV,\mathbf{W}^{\bullet},\Omega^{(N)} \cap \Omega^{(N),'}} $$ for any orientations $\Omega^{(N)} $ and $\Omega^{(N),'}$ ( which may not be equivalent under mutations). All Lusztig's semisimple sheaves are monodromic with respect to this vector bundle, and the Lusztig's sheaves in thick subcategories $\mN_{\nu}$ are preserved under the Fourier-Sato transform. It implies that our construction is independent on the choice of the orientation, but not only independent on the choice of mutation equivalent class of the orientation.
\end{remark} 

\begin{remark}
	One may ask if our construction is a $\mathfrak{sl}_{2}$-categorification in the sense of \cite{CR} and \cite{R2}, or the tensor product $2$-categorification in the sense of \cite{BLW} and \cite{WB15}. Since our functor $\mathcal{F}_{i}$ is defined by Lusztig's induction functor, the composition $\mathcal{F}^{r}_{i}$ is isomorphic to $\mathbf{Ind}(L_{\boldsymbol{i}\cdots \boldsymbol{i}} \boxtimes -)$. By main results of \cite{VVC}, the graded endomorphism ring ${\textrm{Ext}}_{G_{\bV}}^{\ast}(L_{\boldsymbol{i}\cdots \boldsymbol{i}},L_{\boldsymbol{i}\cdots \boldsymbol{i}}  )$ is isomorphic to the KLR algebra of type $A_{1}$, hence there exist generators $X$ and $T$ in  ${\textrm{Ext}}_{G_{\bV}}^{\ast}(L_{\boldsymbol{i}\cdots \boldsymbol{i}},L_{\boldsymbol{i}\cdots \boldsymbol{i}}  )$ induce natural transforms of $\mathcal{F}^{r}_{i}$ satisfying the affine nil-Hecke relations. Our $\mathcal{E}_{i}$ is left and right adjoint to $\mathcal{F}_{i}$ (up to different shifts) as functors between those $\mD_{\nu}(\lambda^{\bullet})/\mN_{\nu},\nu \in \mathbb{N}I$, so there also exist natural transforms of $\mathcal{E}^{r}_{i}$ satisfying the affine nil-Hecke relations. In particular, it implies that our construction is a categorical geometric $\mathfrak{sl}_{2}$-action in the sense of \cite{CKL13}, which is a geometric analogy of the $2$-categorification in \cite{CR}.
\end{remark}

\section{Filtration of based modules arising from tensor products}

We assume $M=L(\lambda^{\bullet})$ is the tensor product of $L(\lambda^{1}),L(\lambda^{2}),\cdots, L(\lambda^{N})$. In this chapter, we construct subcategories $\mL^{\geqslant \mu}(\lambda^{\bullet})$ and $\mL^{>\mu}(\lambda^{\bullet})$ of $\mL(\lambda^{\bullet})$ for each dominant weight $\mu$ of $\mathfrak{g}$, which realize the sub-based modules $M[\geqslant \mu]$ and $M[> \mu]$, and we also generalize our construction to the restricted module $\mathbf{res}^{\mathfrak{g}}_{\mathfrak{l}}M$ of a Levi subalgebra $\mathfrak{l}$.

\subsection{The thick subcategory $\mD^{\geqslant \mu}$ and $\mD^{>\mu}$}\label{section thick}

Recall that the cotangent bundle of the moduli space $\mathbf{E}_{\bV,\mathbf{W}^{\bullet},\Omega^{(N)}}$ can be realized by the moduli space of the representations of doubled quivers
\begin{equation}
	\begin{split}
		 T^{\ast}\mathbf{E}_{\bV,\mathbf{W}^{\bullet},\Omega^{(N)}}=\mathbf{E}_{\bV,\mathbf{W}^{\bullet},H^{(N)}}&=\bigoplus_{h \in \Omega}\mathbf{Hom}(\mathbf{V}_{h'},\mathbf{V}_{h''}) \oplus \bigoplus_{h \in \bar{\Omega}}\mathbf{Hom}(\mathbf{V}_{h'},\mathbf{V}_{h''}) \oplus \\
		 &\bigoplus_{i \in I,1\leqslant l\leqslant N}\mathbf{Hom}(\mathbf{V}_{i},\mathbf{W}^{l}_{i^{l}})  \oplus \bigoplus_{i \in I,1\leqslant l\leqslant N}\mathbf{Hom}(\mathbf{W}^{l}_{i^{l}},\mathbf{V}_{i}).  
	\end{split}
\end{equation}
We denote an element in $\mathbf{E}_{\bV,\mathbf{W}^{\bullet},H^{(N)}}$ by $(x,\bar{x},y,z)$, where $$x=(x_{h})_{h \in \Omega}\in \bigoplus_{h \in \Omega}\mathbf{Hom}(\mathbf{V}_{h'},\mathbf{V}_{h''}),$$ $$\bar{x}=(x_{\bar{h}})_{h\in \Omega} \in \bigoplus_{h \in \bar{\Omega}}\mathbf{Hom}(\mathbf{V}_{h'},\mathbf{V}_{h''}),$$ $$y=(y^{l}_{i})_{i,l} \in \bigoplus_{i \in I,1\leqslant l\leqslant N}\mathbf{Hom}(\mathbf{V}_{i},\mathbf{W}^{l}_{i^{l}}),$$$$z=(z^{l}_{i})_{i,l} \in \bigoplus_{i \in I,1\leqslant l\leqslant N}\mathbf{Hom}(\mathbf{W}^{l}_{i^{l}},\mathbf{V}_{i}).$$

For a fixed dominant weight $\mu \leqslant \sum_{1\leqslant l \leqslant N} \lambda^{l}$, there exists at most one dimension vector $\nu=\nu(\mu)$ such that $wt(\nu)=\sum_{1\leqslant l \leqslant N} \lambda^{l} -\nu =\mu$.

\begin{definition}
Given a fixed dominant weight $\mu \leqslant \sum_{1\leqslant l \leqslant N} \lambda^{l}$ such that $wt^{-1}(\mu)\in \mathbb{N}I$ exists, we say an element $(x,\bar{x},y,z)$ in $\mathbf{E}_{\bV,\mathbf{W}^{\bullet},H^{(N)}}$ is dominated by the dominant weight $\geqslant \mu$, if there exists a subspace $\bV''$  of $\bV$ such that  $\bV'' \oplus \mathbf{W}^{\bullet}$ is  $(x,\bar{x},y,z)$-stable, and its dimension vector $\nu''$ satisfies $\nu'' \leqslant wt^{-1}(\mu)$. 
	
	Define $\mathbf{E}^{\geqslant \mu}_{\bV,\mathbf{W}^{\bullet},H^{(N)}}$ to be the subset of $\mathbf{E}_{\bV,\mathbf{W}^{\bullet},H^{(N)}}$, which consists of those points  dominated by the dominant weight $\geqslant \mu$, and
	we also denote $\mathbf{E}^{> \mu}_{\bV,\mathbf{W}^{\bullet},H^{(N)}}=\bigcup_{\mu'>\mu} \mathbf{E}^{\geqslant \mu'}_{\bV,\mathbf{W}^{\bullet},H^{(N)}}.$
	
	We say a point $(x,\bar{x},y,z)$ in $\mathbf{E}_{\bV,\mathbf{W}^{\bullet},H^{(N)}}$ is strictly dominated by  $\mu$ if it belongs to the subset  $\mathbf{E}^{\geqslant \mu}_{\bV,\mathbf{W}^{\bullet},H^{(N)}} \backslash \mathbf{E}^{> \mu}_{\bV,\mathbf{W}^{\bullet},H^{(N)}}$
\end{definition}

\begin{lemma}\label{close}
	With the notations above, the subset $\mathbf{E}^{\geqslant \mu}_{\bV,\mathbf{W}^{\bullet},H^{(N)}}$ and $\mathbf{E}^{> \mu}_{\bV,\mathbf{W}^{\bullet},H^{(N)}}$ are close subsets of $\mathbf{E}_{\bV,\mathbf{W}^{\bullet},H^{(N)}}$.
\end{lemma}

\begin{proof}
	Since $\mathbf{E}^{> \mu}_{\bV,\mathbf{W}^{\bullet},H^{(N)}}$ is a finite union of $\mathbf{E}^{\geqslant \mu'}_{\bV,\mathbf{W}^{\bullet},H^{(N)}}$, it suffices to prove the statement for $\mathbf{E}^{\geqslant \mu}_{\bV,\mathbf{W}^{\bullet},H^{(N)}}$. For a fixed dimension vector $\nu'' \in \mathbb{N}I$, let $\mathbf{E}''_{\nu''}$  be the variety which consists of $( (x,\bar{x},y,z), \mathbf{V}''\oplus \mathbf{W}^{\bullet} )$ such that $(x,\bar{x},y,z) \in \mathbf{E}_{\bV,\mathbf{W}^{\bullet},H^{(N)}}$, $\mathbf{V}''$ is a subsapce of $\mathbf{V}$ with dimension vector $\nu''$, and $\mathbf{V}''\oplus \mathbf{W}^{\bullet} $ is  $(x,\bar{x},y,z)$-stable.  Notice that the natural projection $$\pi_{\nu''}:\mathbf{E}''_{\nu''} \rightarrow  \mathbf{E}_{\bV,\mathbf{W}^{\bullet},H^{(N)}}$$ is proper, hence the image of $\pi_{\nu''}$ is a close subset. 
	Notice that $\mathbf{E}^{\geqslant \mu}_{\bV,\mathbf{W}^{\bullet},H^{(N)}}= \bigcup_{\nu'' \leqslant wt^{-1}(\mu)} \textrm{Im}(\pi_{\nu''} ) $ is a finite union of close subsets, we finish the proof.
	
\end{proof}

\begin{definition}
	Given a fixed dominant weight $\mu \leqslant \sum_{1\leqslant l \leqslant N} \lambda^{l}$ such that $wt^{-1}(\mu)$ exists,
	let $\mD_{\nu}^{\geqslant \mu}(\lambda^{\bullet})$ be the full subcategory of $\mD_{\nu}(\lambda^{\bullet})$, which consists of objects whose singular supports are contained in $\mathbf{E}^{\geqslant \mu}_{\bV,\mathbf{W}^{\bullet},H^{(N)}}$. Similarly, let $\mD_{\nu}^{> \mu}(\lambda^{\bullet})$ be the full subcategory of $\mD_{\nu}(\lambda^{\bullet})$, which consists of objects whose singular supports are contained in $\mathbf{E}^{> \mu}_{\bV,\mathbf{W}^{\bullet},H^{(N)}}$. 
	We denote $\mP^{\geqslant \mu}_{\nu}(\lambda^{\bullet})$ and  $\mP^{>\mu}_{\nu}(\lambda^{\bullet})$ by the subsets of $\mP_{\nu}(\lambda^{\bullet})$, which consist of simple perverse sheaves in $\mD_{\nu}^{\geqslant \mu}(\lambda^{\bullet})$ and $\mD_{\nu}^{> \mu}(\lambda^{\bullet})$ respectively.
\end{definition}

By definition and the property of singular supports, we can see $\mD_{\nu}^{\geqslant \mu}(\lambda^{\bullet})$ and $\mD_{\nu}^{> \mu}(\lambda^{\bullet})$ are thick subcategories.

\subsection{The localizations and their functors}

\begin{definition}
		Given a fixed dominant weight $\mu \leqslant \sum_{1\leqslant l \leqslant N} \lambda^{l}$ such that $wt^{-1}(\mu)$ exists, define $\mD_{\nu}^{\geqslant \mu}(\lambda^{\bullet})/\mN_{\nu}$ to be the full subcategory of $\mD_{\nu}(\lambda^{\bullet})/\mN_{\nu}$, which consists of objects isomorphic to objects of $\mD_{\nu}^{\geqslant \mu}(\lambda^{\bullet})$. Similarly, define $\mD_{\nu}^{> \mu}(\lambda^{\bullet})/\mN_{\nu}$ to be the full subcategory of $\mD_{\nu}(\lambda^{\bullet})/\mN_{\nu}$, which consists of objects isomorphic to objects of $\mD_{\nu}^{> \mu}(\lambda^{\bullet})$.
\end{definition}

\begin{proposition}\label{keyprop}
	For any $i\in I,r\in \mathbb{N}_{>0}$ and $\nu=ri+\nu''\in\mathbb{N}I$, the functors 
	$$\mathcal{E}^{(r)}_{i}:\mD_{\nu}^{\geqslant \mu}(\lambda^{\bullet})/\mN_{\nu} \rightarrow \mD_{\nu''}^{\geqslant \mu}(\lambda^{\bullet})/\mN_{\nu''},$$
	$$\mathcal{E}^{(r)}_{i}:\mD_{\nu}^{> \mu}(\lambda^{\bullet})/\mN_{\nu} \rightarrow \mD_{\nu''}^{> \mu}(\lambda^{\bullet})/\mN_{\nu''}$$
	are well-defined.
\end{proposition}

In order to prove the proposition above, we need to study the cotangent correspondence of the morphisms in the diagram (\ref{UD}).

Since the Fourier transforms don't change the singular supports by \cite[Theorem D.3]{HLSS}, we  may assume $i$ is a source in $\Omega$, let $\dot{\bV}=\bigoplus_{j\neq i \in I}\mathbf{V}_{j}$ for any $I$-graded space $\bV$, and $$\tilde{\mathbf{V}}_{i}= \bigoplus_{h\in \Omega, h'=i}\mathbf{V}_{h''}\oplus \bigoplus_{1\leqslant l \leqslant N}\mathbf{W}^{l}_{i^{l}}$$ for any $I$-graded space $\bV$ or $I\backslash \{i\}$-graded space $\dot{\bV}$. We also let $\dot{x}=(x_{h})_{h'\neq i,h \in \Omega}$ for any $x=(x_{h})_{h\in \Omega}$, and  $\dot{y}=(y^{l}_{j})_{j \neq i \in I, q\leqslant l \leqslant N}$ for any $y=(y^{l}_{i})$.

Denote $\mathbf{E}^{0}_{\mathbf{V},\mathbf{W}^{\bullet},i}$ by $X$ and denote $\dot{\mathbf{E}}_{\mathbf{V},\mathbf{W}^{\bullet},i} \times \mathbf{Grass}(\nu_i, \tilde{\nu}_{i})$ by $Y$. The cotangent space $T^{\ast}X$ is the subspace of  $\mathbf{E}_{\bV,\mathbf{W}^{\bullet},H^{(N)}}$, which consists of those $(x,\bar{x},y,z)$ such that $${{\rm{dim}}}{\rm{ker}}((\bigoplus _{h \in \Omega, h'=i} x_{h}\oplus \bigoplus_{1\leqslant l \leqslant N} y^{l}_{i}): \mathbf{V}_{i} \rightarrow \bigoplus_{h\in \Omega,h'=i}\mathbf{V}_{h''}\oplus\bigoplus_{1\leqslant l \leqslant N} \mathbf{W}_{i^{l}})=0 .$$

The cotangent space of $T^{\ast}Y$ is the variety, which consists of six-tuples $$(\dot{x},\dot{\bar{x}},\dot{y},\dot{z},\mathbf{S},\rho)$$ such that $(\dot{x},\dot{\bar{x}},\dot{y},\dot{z}) \in T^{\ast}\dot{\mathbf{E}}_{\mathbf{V},\mathbf{W}^{\bullet},i} $, $\mathbf{S}$ is a point in  $ \mathbf{Grass}(\nu_i, \tilde{\nu}_{i})$ and $\rho: \tilde{\mathbf{V}}_{i}/\mathbf{S} \rightarrow \mathbf{S}$ is a linear map.

Let $T^{\ast}X^{\geqslant \mu}= \mathbf{E}^{\geqslant \mu}_{\bV,\mathbf{W}^{\bullet},H^{(N)}} \cap T^{\ast}X$. Let  $T^{\ast}Y^{\geqslant \mu}$ be the subset of $T^{\ast}Y$ such that 
$(\dot{x},\dot{\bar{x}},\dot{y},\dot{z},\mathbf{S},\rho) \in T^{\ast}Y^{\geqslant \mu}$  if and only if there exists a pair of subspace $\dot{\mathbf{T}} \subseteq \dot{\bV}$ and $\mathbf{T}_{i} \subseteq \mathbf{S}$ such that the following three conditions hold
\begin{enumerate}
	\item The dimension vector of $\mathbf{T}_{i} \oplus \dot{\mathbf{T}} $ is less or equal than $wt^{-1}(\mu)$.
	\item The space $\dot{\mathbf{T}}\oplus \dot{\mathbf{W}}^{\bullet}$ is $(\dot{x},\dot{\bar{x}},\dot{y},\dot{z})$-stable.
	\item The composition $\tilde{\mathbf{V}}_{i} \xrightarrow{pr} \tilde{\mathbf{V}}_{i}/\mathbf{S} \xrightarrow{ \rho} \mathbf{S}$ sends $\tilde{\mathbf{T}}_{i}$ to $\mathbf{T}_{i}$.
\end{enumerate}

\begin{lemma}\label{key lemma}
	Assume $i$ is a source, $L$ is an equivariant complex on $Y$ and $K= \phi^{\ast}_{\mathbf{V},i}L$, then $SS(L) \subseteq T^{\ast}Y^{\geqslant \mu}$ if and only if $SS(K) \subseteq T^{\ast}X^{\geqslant \mu}.$
\end{lemma}
\begin{proof}
	For any $(x,y) \in X$, we denote the map $(\bigoplus _{h \in \Omega, h'=i} x_{h}\oplus \bigoplus_{1\leqslant l \leqslant N} y^{l}_{i})$ by $f^{i}_{x,y}$, then $T^{\ast}Y \times_{Y} X$ consists of six-tuples  $(x,\dot{\bar{x}},y,\dot{z},{\rm{Im}} f^{i}_{x,y} ,\rho)$ such that $(x,y) \in X$ and $\rho: \tilde{\mathbf{V}}_{i}/{\rm{Im}} f^{i}_{x,y} \rightarrow {\rm{Im}} f^{i}_{x,y}  $ is a linear map.
	Consider the cotangent correspondence
	$$ T^{\ast}Y \xleftarrow{pr} T^{\ast}Y\times_{Y} X \xrightarrow{d\phi'} T^{\ast}X, $$
	where $pr$ maps $(x,\dot{\bar{x}},y,\dot{z},{\rm{Im}} f^{i}_{x,y} ,\rho)$ to  $(\dot{x},\dot{\bar{x}},\dot{y},\dot{z},{\rm{Im}} f^{i}_{x,y},\rho)$, and $d\phi'$ maps  $(x,\dot{\bar{x}},y,\dot{z},{\rm{Im}} f^{i}_{x,y} ,\rho)$ to  $(x,\bar{x}^{\rho},y,z^{\rho})$. Here $\bar{x}^{\rho}$ and $z^\rho$ are the unique linear map such that $\dot{\bar{x}}=\dot{\bar{x}}^{\rho}$, $\dot{z}=\dot{z}^{\rho}$, and $\sum_{1\leqslant l \leqslant N}z^{\rho,l}_{i}+\sum_{h'=i,h \in \Omega} \bar{x}_{\bar{h}}^{\rho} :\tilde{\mathbf{V}}_{i} \rightarrow \mathbf{V}_{i} $ equals to the composition $\tilde{\mathbf{V}}_{i} \xrightarrow{pr} \tilde{\mathbf{V}}_{i}/ {\rm{Im}} f^{i}_{x,y} \xrightarrow{\rho} {\rm{Im}} f^{i}_{x,y}\xrightarrow{(f^{i}_{x,y})^{-1}} \mathbf{V}_{i}. $
	
	By \cite[Proposition 5.4.5]{MR1074006}, the singular support of $K$ is exactly $d\phi'pr^{-1}(SS(L))$. Given a point $(\dot{x},\dot{\bar{x}},\dot{y},\dot{z},\mathbf{S},\rho)$  in $T^{\ast}Y^{\geqslant \mu}$, we assume that $\dot{\mathbf{T}}$ and $\mathbf{T}_{i}$ satisfy the conditions defining $T^{\ast}Y^{\geqslant \mu}$. Then for any point $(x,\bar{x},y,z)$ in $d\phi'pr^{-1}((\dot{x},\dot{\bar{x}},\dot{y},\dot{z},\mathbf{S},\rho)) $, the subspace $\dot{\mathbf{T}} \oplus (f^{i}_{x,y})^{-1} ( \mathbf{T}_{i} ) $ is a subspace of dimension vector $\leqslant wt^{-1}(\mu)$. By definition of $pr$ and $d\phi'$, $\dot{\mathbf{T}} \oplus (f^{i}_{x,y})^{-1} ( \mathbf{T}_{i} )\oplus \mathbf{W}^{\bullet} $ is also $(x,\bar{x},y,z)$-stable. Hence $(x,\bar{x},y,z)$ belongs to $T^{\ast}X^{\geqslant \mu}$. Conversely, if a point $(x,\bar{x},y,z)$ in $d\phi'pr^{-1}((\dot{x},\dot{\bar{x}},\dot{y},\dot{z},\mathbf{S},\rho)) $ admits a stable subspace $\mathbf{T}\oplus \mathbf{W}^{\bullet}$ such that $\mathbf{T}$ is of dimension vector $\leqslant wt^{-1}(\mu)$. Then $\dot{\mathbf{T}}$ and $f^{i}_{x,y}(\mathbf{T}_{i})$ will satisfy the conditions which make $(\dot{x},\dot{\bar{x}},\dot{y},\dot{z},\mathbf{S},\rho)$ in $T^{\ast}Y^{\geqslant \mu}$. We finish the proof.
\end{proof}

Now we go back to the proof of the Proposition \ref{keyprop}.
\begin{proof}[Proof of \ref{keyprop}]
	Denote $\dot{\mathbf{E}}_{\mathbf{V},\mathbf{W}^{\bullet},i} \times \mathbf{Grass}(\nu_i, \tilde{\nu}_{i}) $ and $\dot{\mathbf{E}}_{\mathbf{V},\mathbf{W}^{\bullet},i} \times \mathbf{Grass}(\nu''_i, \tilde{\nu}_{i})$ by $Y_{1}$ and $Y_{2}$ respectively, and also denote $\dot{\mathbf{E}}_{\mathbf{V},\mathbf{W}^{\bullet},i} \times \mathbf{Flag}(\nu''_{i},\nu_{i},\tilde{\nu}_{i})$ by $Z$. By Lemma \ref{key lemma}, it suffices to show that if $L$ is an equivariant complex on $Y_{1}$ with $SS(L) \subseteq T^{\ast}Y_{1}^{\geqslant \mu}$, then $(q_{2})_{!}q^{\ast}_{1}L$ has singular support contained in  $T^{\ast}Y_{2}^{\geqslant \mu}$.
	
	Notice that $T^{\ast} Y_{1}$ is naturally isomorphic to the variety $\bar{Y}_{1}$, which consists of  $(\dot{x},\dot{\bar{x}},\dot{y},\dot{z},\mathbf{S},\psi)$, such that $(\dot{x},\dot{\bar{x}},\dot{y},\dot{z}) \in T^{\ast}\dot{\mathbf{E}}_{\mathbf{V},\mathbf{W}^{\bullet},i} $, $\mathbf{S}$ is a point in  $ \mathbf{Grass}(\nu_i, \tilde{\nu}_{i})$ and $\psi: \tilde{\mathbf{V}}_{i} \rightarrow \tilde{\mathbf{V}}_{i}$ is a linear map satisfying $\psi( \tilde{\mathbf{V}}_{i} )  \subseteq \mathbf{S} , \psi(\mathbf{S}) =0$. The natural isomorphism is given by sending $\rho$ to the composition $\tilde{\mathbf{V}}_{i} \xrightarrow{pr} \tilde{\mathbf{V}}_{i}/\mathbf{S} \xrightarrow{\rho} \mathbf{S} \xrightarrow{incl} \tilde{\mathbf{V}}_{i}$. Under the natural isomorphism above, the subset $T^{\ast} Y_{1}^{\geqslant \mu}$ is bijective to the  subset $\bar{Y}_{1}^{\geqslant \mu}$ of $\bar{Y}_{1}$, such that $(\dot{x},\dot{\bar{x}},\dot{y},\dot{z},\mathbf{S},\psi) \in \bar{Y}_{1}^{\geqslant \mu}$ if and only if there exists  a pair of subspace $\dot{\mathbf{T}} \subseteq \dot{\bV}$ and $\mathbf{T}_{i} \subseteq \mathbf{S}$ satisfying the following three conditions
	\begin{enumerate}
		\item The dimension vector of $\mathbf{T}_{i} \oplus \dot{\mathbf{T}} $ is less or equal than $wt^{-1}(\mu)$.
		\item The space $\dot{\mathbf{T}}\oplus \dot{\mathbf{W}}^{\bullet}$ is $(\dot{x},\dot{\bar{x}},\dot{y},\dot{z})$-stable.
		\item The morphism $\psi$ sends $\tilde{\mathbf{T}}_{i}$ to $\mathbf{T}_{i}$.
	\end{enumerate}
	
	Similarly, $T^{\ast} Z$ is isomorphic to the variety $\bar{Z}$, which consists of  six-tuples $(\dot{x},\dot{\bar{x}},\dot{y},\dot{z},\mathbf{S}_{1},\mathbf{S}_{2},\psi)$, such that $(\dot{x},\dot{\bar{x}},\dot{y},\dot{z}) \in T^{\ast}\dot{\mathbf{E}}_{\mathbf{V},\mathbf{W}^{\bullet},i} $, $\mathbf{S}_{1} \subseteq \mathbf{S}_{2}$ is a point in  $\mathbf{Flag}(\nu''_{i},\nu_{i},\tilde{\nu}_{i})$ and $\psi: \tilde{\mathbf{V}}_{i} \rightarrow \tilde{\mathbf{V}}_{i}$ is a linear map such that $\psi( \tilde{\mathbf{V}}_{i} )  \subseteq \mathbf{S}_{2} ,\psi(\mathbf{S}_{2}) \subseteq \mathbf{S}_{1}, \psi(\mathbf{S}_{1}) =0$.
	
	The cotangent correspondence of $q_{1}$ can be written as 
	$$ \bar{Y}_{1} \xleftarrow{p_{1}} \bar{Y}_{1} \times_{Y_{1}} Z \xrightarrow{dq'_{1}} \bar{Z}, $$
	where $p_{1}$ is the map forgetting $\mathbf{S}_{1}$ and $dq'_{1}$ is the identity.
	
	Similarly, we can define $\bar{Y}_{2}$ and draw the cotangent correspondence  of $q_{2}$,
	$$ \bar{Y}_{2} \xleftarrow{p_{2}} \bar{Y}_{2} \times_{Y_{2}} Z \xrightarrow{dq'_{2}} \bar{Z}, $$
	where $p_{2}$ is the map forgetting $\mathbf{S}_{2}$ and $dq'_{2}$ is the identity.
	
	We assume that $(\dot{x},\dot{\bar{x}},\dot{y},\dot{z},\mathbf{S}_{2},\psi)$ is a point in $\bar{Y}_{1}^{\geqslant \mu}$ and $\dot{\mathbf{T}} \subseteq \dot{\bV}$ and $\mathbf{T}_{i} \subseteq \mathbf{S}_{2}$ is a pair of subspaces such that the following conditions hold
	\begin{enumerate}
		\item The dimension vector of $\mathbf{T}_{i} \oplus \dot{\mathbf{T}} $ is less or equal than $wt^{-1}(\mu)$.
		\item The space $\dot{\mathbf{T}}\oplus \dot{\mathbf{W}}^{\bullet}$ is $(\dot{x},\dot{\bar{x}},\dot{y},\dot{z})$-stable.
		\item The morphism $\psi$ sends $\tilde{\mathbf{T}}_{i}$ to $\mathbf{T}_{i}$.
	\end{enumerate}
	
	Then for any point $(\dot{x}',\dot{\bar{x}}',\dot{y}',\dot{z}',\mathbf{S}_{1},\psi')$ in $p_{2}p^{-1}_{1}((\dot{x},\dot{\bar{x}},\dot{y},\dot{z},\mathbf{S}_{2},\psi))$, the subspaces $\dot{\mathbf{T}} \subseteq \dot{\bV}$ and $\mathbf{T}'_{i}=\mathbf{T}_{i} \cap \mathbf{S}_{1} \subseteq \mathbf{S}_{1}$ will satisfy the following condition 
	\begin{enumerate}
		\item The dimension vector of $\mathbf{T}'_{i} \oplus \dot{\mathbf{T}} $ is less or equal than $wt^{-1}(\mu)$.
		\item The space $\dot{\mathbf{T}}\oplus \dot{\mathbf{W}}^{\bullet}$ is $(\dot{x}',\dot{\bar{x}}',\dot{y}',\dot{z}')$-stable.
		\item The morphism $\psi'$ sends $\tilde{\mathbf{T}}_{i}$ to $\mathbf{T}'_{i}$.
	\end{enumerate}
	By \cite[Proposition 5.4.4,5.4.5]{MR1074006}, we finish the proof for $\geqslant \mu$. Notice that $ T^{\ast}Y^{> \mu}$ is a union of those $T^{\ast}Y^{\geqslant \mu'}, \mu' > \mu$, the proof of the second statement follows from the proof of the first one.
\end{proof}

\begin{proposition}\label{keyprop2}
	For any $i\in I,r\in \mathbb{N}_{>0}$ and $\nu=ri+\nu''\in\mathbb{N}I$, the functors 
	$$\mathcal{F}^{(r)}_{i}:\mD_{\nu''}^{\geqslant \mu}(\lambda^{\bullet})/\mN_{\nu''} \rightarrow \mD_{\nu}^{\geqslant \mu}(\lambda^{\bullet})/\mN_{\nu},$$
	$$\mathcal{F}^{(r)}_{i}:\mD_{\nu''}^{> \mu}(\lambda^{\bullet})/\mN_{\nu''} \rightarrow \mD_{\nu}^{> \mu}(\lambda^{\bullet})/\mN_{\nu}$$
	are well-defined.
\end{proposition}

\begin{proof}
	Since the induction functor and the singular supports are independent of the orientation, we can assume $i$ is a source. Let $j_{\mathbf{V},i}: \mathbf{E}^{0}_{\mathbf{V},\mathbf{W}^{\bullet},i}  \rightarrow  \mathbf{E}_{\mathbf{V},\mathbf{W}^{\bullet},\Omega_{i}^{(N)}} 
	$  be the open embedding, since the localization factors  through $\mN_{\nu,i}$, there is an isomorphism of functors $$(j_{\mathbf{V},i})_{!}j_{\mathbf{V},i}^{\ast} \cong Id :  \mD_{\nu}(\lambda^{\bullet})/\mN_{\nu} \rightarrow \mD_{\nu}(\lambda^{\bullet})/\mN_{\nu}. $$
	In particular, by the following commutative diagram, 
	\[
	\xymatrix{
		\mathbf{E}_{\mathbf{V}'',\mathbf{W}^{\bullet},\Omega^{(N)}} & \mathbf{E}^{'}_{\mathbf{V},\mathbf{W}^{\bullet},\Omega^{(N)}}\ar[l]_{p_{1}} \ar[r]^{p_{2}}& \mathbf{E}^{''}_{\mathbf{V},\mathbf{W}^{\bullet},\Omega^{(N)}} \ar[r]^{p_{3}}& \mathbf{E}_{\mathbf{V},\mathbf{W}^{\bullet},\Omega^{(N)}} \\
		\mathbf{E}^{0}_{\mathbf{V}'',\mathbf{W}^{\bullet},i} \ar[u] \ar[d] & \mathbf{E}^{',0}_{\mathbf{V},\mathbf{W}^{\bullet},i}  \ar[l]_{\tilde{p}_{1}}  \ar[r]^{\tilde{p}_{2}} \ar[u] \ar[d] & \mathbf{E}^{'',0}_{\mathbf{V},\mathbf{W}^{\bullet},i}  \ar[r]^{\tilde{p}_{3}} \ar[u] \ar[d] & \mathbf{E}^{0}_{\mathbf{V},\mathbf{W}^{\bullet},i} \ar[u] \ar[d]\\
		\txt{$\dot{\mathbf{E}}_{\mathbf{V}'',\mathbf{W}^{\bullet},i}$\\ $\times$\\ $\mathbf{Gr}(\nu''_{i}, \tilde{\nu}_{i})$}  & \txt{$\dot{\mathbf{E}}_{\mathbf{V},\mathbf{W}^{\bullet},i}$\\ $\times$\\ $\mathbf{Fl}(\nu''_{i},\nu_{i},\tilde{\nu}_{i})$} \ar[l]_{q_{2}} \ar@2{-}[r]  & \txt{$\dot{\mathbf{E}}_{\mathbf{V},\mathbf{W}^{\bullet},i}$\\  $\times$\\  $\mathbf{Fl}(\nu''_{i},\nu_{i},\tilde{\nu}_{i})$}  \ar[r]^{q_{1}}  & \txt{$\dot{\mathbf{E}}_{\mathbf{V},\mathbf{W}^{\bullet},i}$\\ $\times$\\$\mathbf{Gr}(\nu_{i}, \tilde{\nu}_{i}$)} 
	}
	\]
	we have $\mathcal{F}^{(r)}_{i} \cong (j_{\mathbf{V},i})_{!} (\phi_{\mathbf{V},i})^{\ast} (q_{1})_{!}(q_{2})^{\ast} (\phi_{\mathbf{V}'',i})_{\flat}(j_{\mathbf{V}'',i})^{\ast}[M(r,\nu)]$ for some integer $M(r,\nu)$.
	
	Denote $\dot{\mathbf{E}}_{\mathbf{V},\mathbf{W}^{\bullet},i} \times \mathbf{Grass}(\nu_i, \tilde{\nu}_{i}) $ and $\dot{\mathbf{E}}_{\mathbf{V},\mathbf{W}^{\bullet},i} \times \mathbf{Grass}(\nu''_i, \tilde{\nu}_{i})$ by $Y_{1}$ and $Y_{2}$. By Lemma \ref{key lemma}, it suffices to show that if $L$ is an equivariant complex on $Y_{2}$ with $SS(L) \subseteq T^{\ast}Y_{2}^{\geqslant \mu}$, then $(q_{1})_{!}q^{\ast}_{2}L$ has singular support contained in  $T^{\ast}Y_{1}^{\geqslant \mu}$.
	With the notations in the proof of Proposition 4.5, we can draw the cotangent correspondences 
		$$ \bar{Y}_{1} \xleftarrow{p_{1}} \bar{Y}_{1} \times_{Y_{1}} Z \xrightarrow{dq'_{1}} \bar{Z}, $$
	$$ \bar{Y}_{2} \xleftarrow{p_{2}} \bar{Y}_{2} \times_{Y_{2}} Z \xrightarrow{dq'_{2}} \bar{Z}, $$
	where $p_{1},p_{2}$ are the forgetting maps and $dq'_{1},dq'_{2}$ are the identity maps.
	
	We assume that $(\dot{x},\dot{\bar{x}},\dot{y},\dot{z},\mathbf{S}_{1},\psi)$ is a point in $\bar{Y}_{2}^{\geqslant \mu}$ and $\dot{\mathbf{T}} \subseteq \dot{\bV}$ and $\mathbf{T}_{i} \subseteq \mathbf{S}_{1}$ is a pair of subspaces such that the following conditions hold
	\begin{enumerate}
		\item The dimension vector of $\mathbf{T}_{i} \oplus \dot{\mathbf{T}} $ is less or equal than $wt^{-1}(\mu)$.
		\item The space $\dot{\mathbf{T}}\oplus \dot{\mathbf{W}}^{\bullet}$ is $(\dot{x},\dot{\bar{x}},\dot{y},\dot{z})$-stable.
		\item The morphism $\psi$ sends $\tilde{\mathbf{T}}_{i}$ to $\mathbf{T}_{i}$.
	\end{enumerate}
	
	Then for any point $(\dot{x}',\dot{\bar{x}}',\dot{y}',\dot{z}',\mathbf{S}_{2},\psi')$ in $p_{1}p^{-1}_{2}((\dot{x},\dot{\bar{x}},\dot{y},\dot{z},\mathbf{S}_{1},\psi))$, the subspaces $\dot{\mathbf{T}} \subseteq \dot{\bV}$ and $\mathbf{T}_{i} \subseteq \mathbf{S}_{1} \subseteq \mathbf{S}_{2}$ will satisfy the following condition 
	\begin{enumerate}
		\item The dimension vector of $\mathbf{T}_{i} \oplus \dot{\mathbf{T}} $ is less or equal than $wt^{-1}(\mu)$.
		\item The space $\dot{\mathbf{T}}\oplus \dot{\mathbf{W}}^{\bullet}$ is $(\dot{x}',\dot{\bar{x}}',\dot{y}',\dot{z}')$-stable.
		\item The morphism $\psi'$ sends $\tilde{\mathbf{T}}_{i}$ to $\mathbf{T}_{i}$.
	\end{enumerate}
	By \cite[Proposition 5.4.4,5.4.5]{MR1074006}, we finish the proof.
\end{proof}

\begin{definition}
	Let $\mL^{\geqslant \mu}_{\nu}(\lambda^{\bullet})$ be the full subcategory of $\mD_{\nu}^{\geqslant \mu}(\lambda^{\bullet})/\mN_{\nu}$, which consists of objects isomorphic to objects of $\mQ_{\nu}(\lambda^{\bullet})$, and let $\mL^{> \mu}_{\nu}(\lambda^{\bullet})$ be the full subcategory of $\mD_{\nu}^{> \mu}(\lambda^{\bullet})/\mN_{\nu}$, which consists of objects isomorphic to objects of $\mQ_{\nu}(\lambda^{\bullet})$. 
\end{definition}

\begin{corollary}
	For any $i\in I,r\in \mathbb{N}_{>0}$ and $\nu=ri+\nu''\in\mathbb{N}I$, the functors 
	$$\mathcal{E}^{(r)}_{i}:\mL_{\nu}^{\geqslant \mu}(\lambda^{\bullet}) \rightarrow \mL_{\nu''}^{\geqslant \mu}(\lambda^{\bullet}),$$
	$$\mathcal{F}^{(r)}_{i}:\mL_{\nu''}^{\geqslant \mu}(\lambda^{\bullet}) \rightarrow \mL_{\nu}^{\geqslant \mu}(\lambda^{\bullet}),$$
	$$\mathcal{E}^{(r)}_{i}:\mL_{\nu}^{> \mu}(\lambda^{\bullet}) \rightarrow \mL_{\nu''}^{> \mu}(\lambda^{\bullet}),$$
	$$\mathcal{F}^{(r)}_{i}:\mL_{\nu''}^{> \mu}(\lambda^{\bullet}) \rightarrow \mL_{\nu}^{> \mu}(\lambda^{\bullet})$$
	are well-defined.
\end{corollary}

\begin{proof}
	Since $\mathcal{F}^{(r)}_{i}$ is defined by induction fuctor, it preserves objects of those $\mQ_{\nu}(\lambda^{\bullet})$. The proof of $\mathcal{E}^{(r)}_{i}$ is the same as \cite[Proposition 3.13]{fang2023tensor}, by using commutative relations of $\mathcal{E}^{(r)}_{i}$ and $\mathcal{F}^{(s)}_{j}$ and induction on the lengths of those $\boldsymbol{\nu}^{l}$. (Here we also need to regard $i^{l},i \in I, 1\leqslant l \leqslant N$ as unframed vertices and consider the functor $\mathcal{F}^{(d^{l}_{i^{l}})}_{i^{l}}$, which always commute with $\mathcal{E}_{i}$.)
\end{proof}

\begin{definition}
	Let $\mN_{\nu}^{>\mu}$ be the thick subcategory of $\mD_{\nu}(\lambda^{\bullet})$ generated by $\mN_{\nu}$ and $\mD^{>\mu}_{\nu}(\lambda^{\bullet})$, define $\mD_{\nu}(\lambda^{\bullet})/\mN_{\nu}^{>\mu}$ be the Verdier quotient of $\mD_{\nu}(\lambda^{\bullet})$ with respect to the thick subcategory, and define $\mL^{\mu}_{\nu}$ be the full subcategory of $\mD_{\nu}(\lambda^{\bullet})/\mN_{\nu}^{>\mu}$, which consists of objects isomorphic to objects of $\mL^{\geqslant \mu}_{\nu}(\lambda^{\bullet})$.
\end{definition}

Since $\mD_{\nu}(\lambda^{\bullet})/\mN_{\nu}^{>\mu}$ is isomorphic to the Verdier quotient $$(\mD_{\nu}(\lambda^{\bullet})/\mN_{\nu})/(\mD_{\nu}^{> \mu}(\lambda^{\bullet})/\mN_{\nu}),$$ the following proposition follows from Proposition 4.5 and 4.7.
	\begin{proposition}\label{keyprop3}
		For any $i\in I,r\in \mathbb{N}_{>0}$ and $\nu=ri+\nu''\in\mathbb{N}I$, the functors 
		$$\mathcal{E}^{(r)}_{i}:\mD_{\nu}(\lambda^{\bullet})/\mN_{\nu}^{> \mu} \rightarrow \mD_{\nu''}(\lambda^{\bullet})/\mN_{\nu''}^{>\mu},$$
		$$\mathcal{F}^{(r)}_{i}:\mD_{\nu''}(\lambda^{\bullet})/\mN_{\nu''}^{>\mu} \rightarrow \mD_{\nu}(\lambda^{\bullet})/\mN_{\nu}^{>\mu},$$
		are well-defined.
	\end{proposition}
	As a corollary, these functors also restrict to functors of Lustzig's sheaves.
	\begin{corollary}
		For any $i\in I,r\in \mathbb{N}_{>0}$ and $\nu=ri+\nu''\in\mathbb{N}I$, the functors 
		$$\mathcal{E}^{(r)}_{i}:\mL^{\mu}_{\nu}(\lambda^{\bullet}) \rightarrow \mL^{\mu}_{\nu''}(\lambda^{\bullet}),$$
		$$\mathcal{F}^{(r)}_{i}:\mL^{\mu}_{\nu''}(\lambda^{\bullet}) \rightarrow \mL^{\mu}_{\nu}(\lambda^{\bullet})$$
		are well-defined.
	\end{corollary}

\subsection{The structure of the Grothendieck groups}
In this subsection, we will determine the module structure of the Grothendieck groups of $\mL^{\mu}_{\nu}(\lambda^{\bullet}),\mL^{\mu}_{>\nu}(\lambda^{\bullet})$ and $\mL^{\mu}_{\geqslant \nu}(\lambda^{\bullet})$. The main tool is the crystal structure of irreducible components of Lusztig's nilpotent variety defined in \cite{MR1088333} and \cite{MR1758244}. 

\subsubsection{The variety $\Lambda_{\bV}$ and its generalization}
For any quiver  $Q=(I,H,\Omega)$ and a dimension vector $\nu\in \mathbb{N}I$, we fix a $I$-graded $\mathbb{C}$-vector space $\bV$ such that its dimension vector is $\nu$. We define the moduli space of the double quiver
$$\bfEVH= \bigoplus_{h\in H} \Hom (\bV_{h'},\bV_{h''})$$
such that the connected algebraic group 
$G_{\bV}$ acts on $\bfEVH$ by conjugation. We fix a function $\varepsilon:H\rightarrow\mathbb{C}^*$ such that $\varepsilon(h)+\varepsilon(\overline{h})=0$ for any $h\in H$. There is a non-degenerate $G_{\bV}$-invariant symplectic form on $\bfEVH$ given by 
$$\langle x,x' \rangle=\sum_{h\in H}\varepsilon(h)\mathrm{tr}(x_hx'_{\overline{h}}:V_{h''}\rightarrow V_{h''})$$
such that $\bfEVH$ can be identified with the cotangent bundle of $\bfEVO$, see \cite[Section 12.8]{MR1088333}. Hence we also use $(x,\bar{x})$ to denote a point in $\bfEVH$, where $x\in \bfEVO$ and $\bar{x} \in \mathbf{E}_{\bV,\bar{\Omega}}$. The moment map attached to the $G_{\bV}$-action on the symplectic vector space $\bfEVH$ is 
\begin{align*}
	\mu_{\bV}:\bfEVH&\rightarrow \bigoplus_{i\in I}\mathrm{End}(\bV_i)\\
	(x,\bar{x})&\mapsto (\sum_{h\in H, h''=i}\varepsilon(h)x_hx_{\overline{h}})_{i\in I}.
\end{align*}
Lusztig's nilpotent variety is defined to be
$$\Lambda_{\bV}=\{(x,\bar{x})\in \mu_{\bV}^{-1}(0)\mid (x,\bar{x})\ \textrm{is nilpotent}\},$$
where $(x,\bar{x})\in \bfEVH$ is said to be nilpotent, if there exists an $N\geqslant 2$ such that for any sequence $h_1,...,h_N\in H$ satisfying $h_1'=h_2'',...,h_{N-1}'=h_N''$, the composition $x_{h_1}...x_{h_N}:\bV_{h'_N}\rightarrow \bV_{h_1''}$ is zero. By \cite[Theorem 12.9]{MR1088333}, the nilpotent variety $\Lambda_{\bV}$ is a Lagrangian subvariety of $\bfEVH$.

Now we replace $Q$ by $Q^{(N)}$, and denote the corresponding moment map by $$\mu_{\bV,\mathbf{W}^{\bullet}}: \mathbf{E}_{\bV,\mathbf{W}^{\bullet},H^{(N)}} \longrightarrow \bigoplus_{i\in I,1\leqslant l \leqslant N}\mathrm{End}(\bV_i) \oplus  \mathrm{End}(\mathbf{W}^{l}_{i^{l}})$$
$$ \mu_{\bV,\mathbf{W}^{\bullet}}:(x,\bar{x},y,z ) \mapsto ( (\sum_{h\in H, h''=i}\varepsilon(h)x_hx_{\overline{h}}+\sum_{1\leqslant l \leqslant N}z^{l}_{i}y^{l}_{i} )_{i\in I},(y^{l}_{i}z^{l}_{i})_{i\in I,1\leqslant l \leqslant N}) .$$
Similarly, we denote the nilpotent variety of $Q^{(N)}$ by $$\Lambda_{\bV,\mathbf{W}^{\bullet}}= \{(x,\bar{x},y,z )\in \mu_{\bV,\mathbf{W}^{\bullet}}^{-1}(0)\mid (x,\bar{x},y,z)\ \textrm{is nilpotent}\}. $$
By \cite[Theorem 12.9]{MR1088333}, the nilpotent variety $\Lambda_{\bV,\mathbf{W}^{\bullet}}$ is also a Lagrangian subvariety.

For any $i\in I, \nu\in \mathbb{N}I$ and $0\leqslant t\leqslant \nu_i$, we define $\Lambda_{\bV,i,t}\subset \Lambda_{\bV}$  to be the locally closed subset consisting of $(x,\bar{x})$  such that 
$$\mathrm{comdim}_{\bV_i}(\sum_{h\in H,h''=i}\mathrm{Im}(x_h:\bV_{h'}\rightarrow \bV_i))=t,$$
and define $\Lambda_{\bV,\mathbf{W}^{\bullet},i,t}\subset \Lambda_{\bV,\mathbf{W}^{\bullet}}$ to be the locally closed subset consisting of  $(x,\bar{x},y,z)$ such that
$$\mathrm{comdim}_{\bV_i}(\sum_{h\in H,h''=i}\mathrm{Im}(x_h:\bV_{h'}\rightarrow \bV_i)+\sum_{1\leqslant l \leqslant N} \mathrm{Im}(z^{l}_i:\mathbf{W}^{l}_{i^{l}}\rightarrow \bV_i))=t,$$
then all these $\Lambda_{\bV,i,t}$ or $\Lambda_{\bV,\mathbf{W}^{\bullet},i,t},0\leqslant t\leqslant \nu_i$ form a stratification of $\Lambda_{\bV}$ or $\Lambda_{\bV,\mathbf{W}^{\bullet}}$ respectively. 

Dually, we define $\Lambda^t_{\bV,i}\subset \Lambda_{\bV}$ to be the locally closed subset consisting of $(x,\bar{x})$ such that
$$\mathrm{dim}\bigcap_{h\in H,h'=i}\mathrm{Ker}(x_h:\bV_i\rightarrow \bV_{h''})=t,$$
and define $\Lambda^t_{\bV,\mathbf{W}^{\bullet},i}\subset \Lambda_{\bV,\mathbf{W}^{\bullet}}$ to be the locally closed subset consisting of $(x,\bar{x},y,z)$ such that
$$\mathrm{dim}\bigcap_{h\in H,h'=i}\mathrm{Ker}(x_h:\bV_i\rightarrow \bV_{h''}) \cap \bigcap_{1\leqslant l \leqslant N}\mathrm{Ker}(y^{l}_{i}:\bV_i\rightarrow \mathbf{W}^{l}_{i^{l}})=t,$$
then all these $\Lambda^t_{\bV,i}$ or $\Lambda^t_{\bV,\mathbf{W}^{\bullet},i},0\leqslant t\leqslant \nu_i$ also form a stratification of $\Lambda_{\bV}$ or or $\Lambda_{\bV,\mathbf{W}^{\bullet}}$ respectively. 

For any irreducible component $Z\subset \Lambda_{\bV}$ (or $Z\subset \Lambda_{\bV,\mathbf{W}^{\bullet}} $), there are unique integers $\epsilon_{i}(Z),\epsilon_{i}^\ast(Z)$ such that $\Lambda_{\bV,i,\epsilon_{i}(Z)} \cap Z$ and $ \Lambda^{\epsilon^\ast_{i}(Z)}_{\bV,i} \cap Z$ ( or $\Lambda_{\bV,\mathbf{W}^{\bullet},i,\epsilon_{i}(Z)} \cap Z$ and $ \Lambda^{\epsilon^\ast_{i}(Z)}_{\bV,\mathbf{W}^{\bullet},i} \cap Z$ respectively) are open dense in $Z$.
\begin{lemma}[{\cite[Lemma 12.5]{MR1088333}}]
	For any $0\leqslant t\leqslant \nu_i$, let $\nu',\nu''=\nu-ti$, then there are bijections
	\begin{align*}
		\rho_{i,t}: \{Z \in {\rm{Irr}} \Lambda_{\bV}\mid\epsilon_{i}(Z)=t \} \rightarrow   \{Z' \in {\rm{Irr}} \Lambda_{\bV''}\mid\epsilon_{i}(Z')=0 \},\\
		\rho^{\ast}_{i,t}: \{Z \in {\rm{Irr}} \Lambda_{\bV}\mid\epsilon^{\ast}_{i}(Z)=t \} \rightarrow   \{Z' \in {\rm{Irr}} \Lambda_{\bV'}\mid\epsilon^{\ast}_{i}(Z')=0 \},
	\end{align*}
	and bijections
	\begin{align*}
		\rho_{i,t}: \{Z \in {\rm{Irr}} \Lambda_{\bV,\mathbf{W}^{\bullet}}\mid\epsilon_{i}(Z)=t \} \rightarrow   \{Z' \in {\rm{Irr}} \Lambda_{\bV'',\mathbf{W}^{\bullet}}\mid\epsilon_{i}(Z')=0 \},\\
		\rho^{\ast}_{i,t}: \{Z \in {\rm{Irr}} \Lambda_{\bV,\mathbf{W}^{\bullet}}\mid\epsilon^{\ast}_{i}(Z)=t \} \rightarrow   \{Z' \in {\rm{Irr}} \Lambda_{\bV',\mathbf{W}^{\bullet}}\mid\epsilon^{\ast}_{i}(Z')=0 \}.
	\end{align*}
\end{lemma}

It is well-known that the bijections $\rho_{i,t}$ induce a crystal structure on the set $\bigcup_{\nu} {\rm{Irr}} \Lambda_{\bV}$ of irreducible components  of $\Lambda_{\bV}$, which is isomorphic to the crystal structure of $B(\infty)$ by \cite{MR1458969}. Similar geometric crystal structures of $B(\lambda)$  are also constructed in  \cite{Saito2002} and \cite{SA}.

\subsubsection{The key lemma of Lusztig's sheaves}
Lusztig's simple perverse sheaves in $\mathcal{P}_{\nu}$ share similar properties of irreducible components of nilpotent varieties. (For the definition of $\mathcal{P}_{\nu}$, one can see details in \cite{MR1088333} or \cite{MR1227098}.)

Let $i$ be a source for the orientation $\Omega$. For any $\nu\in \mathbb{N}I, 0\leqslant t\leqslant \nu_i$, we define $\bfEVO^{i,t}\subset \bfEVO$ to be the locally closed subset consisting of $x$ such that
$${\mathrm{dim}}\,(\bigcap_{h\in \Omega,h'=i} \mathrm{Ker}(x_h:\bV_i\rightarrow \bV_{h''}))=t,$$
then all these $\bfEVO^{i,t},0\leqslant t\leqslant \nu_i$ form a  stratification of $\bfEVO$. For any equivariant simple perverse sheaf $L$, there exists a unique integer $t_i^{\ast}(L)$ such that $\bfEVO^{i,t_i^{\ast}(L)}\cap \mathrm{supp}(L)$ is open dense in $\mathrm{supp}(L)$. The following lemma is  dual to Lusztig's key lemma in \cite[Lemma 6.4]{MR1088333}.

\begin{lemma} \label{rkey}
	For any $0 \leqslant t \leqslant \nu_{i}$, let $\nu''=ti$, then $\mathbf{E}_{\mathbf{V}'',\Omega}=\{0\}$ is a point. For any $\nu'\in \mathbb{N}I$ such that $\nu=\nu'+\nu''$, we identify $\mathbf{E}_{\mathbf{V}',\Omega}\times \mathbf{E}_{\mathbf{V}'',\Omega}$ with $\mathbf{E}_{\mathbf{V}',\Omega}$. \\ 
	{{\rm{(a)}}} For any $L\in \mP_{\nu}$ with $t_{i}^{\ast}(L)=t$, we have 
	$$\mathbf{Res}^{\mathbf{V}}_{\mathbf{V'},\mathbf{V}''}(L)= K \oplus \bigoplus_{t_{i}^{\ast}(K')>0} K'[d']^{\oplus n(K',d')},$$ 
	where $K, K'\in \mP_{\nu'},d'\in \mathbb{Z}, n(K',d')\in \mathbb{N}$ and $K$ is the unique direct summand satisfying $t_{i}^{\ast}(K)=0$.\\ 
	{{\rm{(b)}}} For any $K \in \mP_{\nu'}$ with $t_{i}^{\ast}(K)=0$, we have
	$$\mathbf{Ind}^{\mathbf{V}}_{\mathbf{V}',\mathbf{V}''}( K \boxtimes \mathbb{C}_{\mathbf{E}_{\mathbf{V}'',\Omega}} ) = L \oplus \bigoplus_{t_{i}^{\ast}(L')>t} L'[d']^{\oplus n(L',d')}, $$ 
	where $L, L'\in \mP_{\nu},d'\in \mathbb{Z}, n(L',d')\in \mathbb{N}$ and $L$ is the unique direct summand satisfying $t_{i}^{\ast}(L)=t$.\\
	{{\rm{(b)}}} The maps $L\mapsto K$ and $K\mapsto L$ define a bijection   $\pi^{\ast}_{i,t}:\{ L \in \mP_{\nu}| t^{\ast}_{i}(L)=t \} \rightarrow \{ K \in \mP_{\nu'}| t^{\ast}_{i}(K)=0 \} $.
\end{lemma}

Note that $i$ is a sink for the orientation $\bar{\Omega}$. For any $\nu\in \mathbb{N}I, 0\leqslant t\leqslant \nu_i$, we define ${^{t}\mathbf{E}_{\bV,\bar{\Omega}}^{i}}\subset \mathbf{E}_{\bV,\bar{\Omega}}$ to be the locally closed subset consisting of $x$ such that
$$\mathrm{codim}_{\bV_i}(\sum_{h\in \bar{\Omega}, h''=i}\mathrm{Im} (x_h:\bV_{h'}\rightarrow \bV_i))=t,$$
then all these ${^{t}\mathbf{E}_{\bV,\bar{\Omega}}^{i}}, 0\leqslant t\leqslant \nu_i$ form a  stratification of $\mathbf{E}_{\bV,\bar{\Omega}}$. For any simple equivariant perverse sheaf $L$, there exists a unique integer $t_i(L)$ such that ${^{t_i(L)}\mathbf{E}_{\bV,\bar{\Omega}}^{i}}\cap \mathrm{supp}(L)$ is open dense in $\mathrm{supp}(L)$. 

\begin{lemma}[{\cite[Lemma 6.4]{MR1088333}}]\label{lkey}
	For any $0 \leqslant t \leqslant \nu_{i}$, let $\nu'=ti$, then $\mathbf{E}_{\mathbf{V}',\bar{\Omega}}=\{0\}$ is a point. For any $\nu''\in \mathbb{N}I$ such that $\nu=\nu'+\nu''$, we identify $\mathbf{E}_{\mathbf{V}',\bar{\Omega}}\times \mathbf{E}_{\mathbf{V}'',\bar{\Omega}}$ with $\mathbf{E}_{\mathbf{V}'',\bar{\Omega}}$. \\
	{{\rm{(a)}}} For any $L\in \mP_{\nu}$ with $t_{i}(L)=t$, we have 
	$$\mathbf{Res}^{\mathbf{V}}_{\mathbf{V'},\mathbf{V}''}(L)= K \oplus \bigoplus_{t_{i}(K')>0} K'[d']^{\oplus n(K',d')},$$ 
	where $K, K'\in \mP_{\nu''},d'\in \mathbb{Z}, n(K',d')\in \mathbb{N}$ and $K$ is the unique direct summand satisfying $t_{i}(K)=0$.\\ 
	{{\rm{(b)}}} For any $K \in \mP_{\nu''}$ with $t_{i}(K)=0$, we have
	$$\mathbf{Ind}^{\mathbf{V}}_{\mathbf{V}',\mathbf{V}''}(\mathbb{C}_{\mathbf{E}_{\mathbf{V}',\bar{\Omega}}} \boxtimes K) = L \oplus \bigoplus_{t_{i}(L')>t} L'[d']^{\oplus n(L',d')}, $$ 
	where $L, L'\in \mP_{\nu},d'\in \mathbb{Z}, n(L',d')\in \mathbb{N}$ and $L$ is the unique direct summand satisfying $t_{i}(L)=t$.\\
	{{\rm{(c)}}} The maps $L\mapsto K$ and $K\mapsto L$ define a bijection  $\pi_{i,t}:\{ L \in \mP_{\nu}| t_{i}(L)=t \} \rightarrow \{ K \in \mP_{\nu''}| t_{i}(K)=0 \} $.
\end{lemma}

For general orientation $\Omega$, we choose an orientation $\Omega^{i}$ such that $i$ is a sink and define $t_{i}(L)=t_{i}(\mathbf{Four}_{\Omega,\Omega^{i}}(L))$, and also choose an orientation $\Omega_{i}$ such that $i$ is a source and define $t^{\ast}_{i}(L)=t^{\ast}_{i}(\mathbf{Four}_{\Omega,\Omega_{i}}(L))$.
Since the induction and restriction functors commute with Fourier-Sato transforms, these bijections $\pi_{i,t}$ induce a crystal structure on $\bigcup_{\nu \in \mathbb{N}I}\mathcal{P}_{\nu}$, which is isomorphic to $B(\infty)$ by \cite[Corollary 17.3.1]{MR1227098}.  Combine the main result in \cite{MR1458969} and \cite[Chapter 17]{MR1227098}, we have the following proposition.
\begin{proposition}\label{Cry}
	There is a bijection $\Psi=\Psi_{Q}: \bigcup_{\nu \in \mathbb{N}I}\mathcal{P}_{\nu} \rightarrow \bigcup_{\nu \in \mathbb{N}I} {\rm{Irr}} \Lambda_{\bV}$ such that
	\begin{enumerate}
		\item For any simple perverse sheaf $L$ and $i \in I$, $t_{i}(L)=\epsilon_{i}(\Psi(L))$ and $t^{\ast}_{i}(L)=\epsilon^{\ast}_{i}(\Psi(L))$.
		\item  For any simple perverse sheaf $L$, we have $$\Psi(L) \subseteq SS(L) \subseteq  \bigcup_{\epsilon_{i}(Z') \geqslant t_{i}(L), \epsilon^{\ast}_{i}(Z') \geqslant t^{\ast}_{i}(L) }Z'.$$
		\item For any simple perverse sheaves $L$ and $K$, $\pi_{i,t}(L)=K$ if and only if $\rho_{i,t}(\Psi(L) )=\Psi(K). $
	\end{enumerate}
\end{proposition}

Similarly, assume $i$ is a source in $\Omega^{(N)}$, we define ${\mathbf{E}_{\mathbf{V},\mathbf{W}^{\bullet},\Omega^{(N)}}^{i,t}}\subset \mathbf{E}_{\mathbf{V},\mathbf{W}^{\bullet},\Omega^{(N)}}$ to be the locally closed subset consisting of $(x,y)$ such that
$${\mathrm{dim}}\,(\bigcap_{h\in \Omega,h'=i} \mathrm{Ker}(x_h:\bV_i\rightarrow \bV_{h''}) \cap \bigcap_{1\leqslant l \leqslant N} \mathrm{Ker}(y^{l}_{i}:\bV_i\rightarrow \mathbf{W}^{l}_{i^{l}}))=t,$$
then for any simple equivariant perverse sheaf $L$, there exists a unique integer $t_i^{\ast}(L)$ such that $\mathbf{E}_{\mathbf{V},\mathbf{W}^{\bullet},\Omega^{(N)}}^{i,t_i^{\ast}(L)}\cap \mathrm{supp}(L)$ is open dense in $\mathrm{supp}(L)$. Then following lemma is a generalization of Lemma \ref{rkey} on $N$-framed quivers.

\begin{lemma} \label{rkeyt}
	For any $0 \leqslant t \leqslant \nu_{i}$, let $\nu''=ti$, then $\mathbf{E}_{\mathbf{V}'',0,\Omega^{(N)}}=\{0\}$ is a point. For any $\nu'\in \mathbb{N}I$ such that $\nu=\nu'+\nu''$, we identify $\mathbf{E}_{\mathbf{V}',\mathbf{W}^{\bullet},\Omega^{(N)}}\times \mathbf{E}_{\mathbf{V}'',0,\Omega^{(N)}}$ with $\mathbf{E}_{\mathbf{V}',\mathbf{W}^{\bullet},\Omega^{(N)}}$. \\ 
	{{\rm{(a)}}} For any $L\in \mP_{\nu}(\lambda^{\bullet})$ with $t_{i}^{\ast}(L)=t$, we have 
	$$\mathbf{Res}^{\mathbf{V}\oplus \mathbf{W}^{\bullet}}_{\mathbf{V'}\oplus \mathbf{W}^{\bullet},\mathbf{V}''}(L)= K \oplus \bigoplus_{t_{i}^{\ast}(K')>0} K'[d']^{\oplus n(K',d')},$$ 
	where $K, K'\in \mP_{\nu'}(\lambda^{\bullet}),d'\in \mathbb{Z}, n(K',d')\in \mathbb{N}$ and $K$ is the unique direct summand satisfying $t_{i}^{\ast}(K)=0$.\\ 
	{{\rm{(b)}}} For any $K \in \mP_{\nu'}(\lambda^{\bullet})$ with $t_{i}^{\ast}(K)=0$, we have
	$$\mathbf{Ind}^{\mathbf{V}\oplus \mathbf{W}^{\bullet}}_{\mathbf{V}'\oplus \mathbf{W}^{\bullet},\mathbf{V}''}( K \boxtimes \mathbb{C}_{\mathbf{E}_{\mathbf{V}'',\Omega}} ) = L \oplus \bigoplus_{t_{i}^{\ast}(L')>t} L'[d']^{\oplus n(L',d')}, $$ 
	where $L, L'\in \mP_{\nu}(\lambda^{\bullet}),d'\in \mathbb{Z}, n(L',d')\in \mathbb{N}$ and $L$ is the unique direct summand satisfying $t_{i}^{\ast}(L)=t$.\\
	{{\rm{(b)}}} The maps $L\mapsto K$ and $K\mapsto L$ define a bijection   $\pi^{\ast}_{i,t}:\{ L \in \mP_{\nu}(\lambda^{\bullet})| t^{\ast}_{i}(L)=t \} \rightarrow \{ K \in \mP_{\nu'}(\lambda^{\bullet})| t^{\ast}_{i}(K)=0 \} $.
\end{lemma}

Since $i$ is a sink in $\overline{\Omega^{(N)}}$, we can we define ${^{t}\mathbf{E}_{\bV,\mathbf{W}^{\bullet},\overline{\Omega^{(N)}}}^{i}}\subset \mathbf{E}_{\bV,\mathbf{W}^{\bullet},\overline{\Omega^{(N)}}}$ to be the locally closed subset consisting of $(\bar{x},z)$ such that
$$\mathrm{codim}_{\bV_i}(\sum_{h\in \bar{\Omega}, h''=i}\mathrm{Im} (x_h:\bV_{h'}\rightarrow \bV_i)+(\sum_{1 \leqslant l \leqslant N}\mathrm{Im} (z^{l}_i:\mathbf{W}^{l}_{i^{l}}\rightarrow \bV_i))=t,$$
then all these ${^{t}\mathbf{E}_{\bV,\mathbf{W}^{\bullet},\overline{\Omega^{(N)}}}^{i}}, 0\leqslant t\leqslant \nu_i$ form a  stratification of $\mathbf{E}_{\bV,\mathbf{W}^{\bullet},\overline{\Omega^{(N)}}}$. For any simple equivariant perverse sheaf $L$, there exists a unique integer $t_i(L)$ such that ${^{t_{i}(L)}\mathbf{E}_{\bV,\mathbf{W}^{\bullet},\overline{\Omega^{(N)}}}^{i}}\cap \mathrm{supp}(L)$ is open dense in $\mathrm{supp}(L)$. 

\begin{lemma}\label{lkeyt}
	For any $0 \leqslant t \leqslant \nu_{i}$, let $\nu'=ti$, then $\mathbf{E}_{\mathbf{V}',0,\overline{\Omega^{(N)}}}=\{0\}$ is a point. For any $\nu''\in \mathbb{N}I$ such that $\nu=\nu'+\nu''$, we identify $\mathbf{E}_{\mathbf{V}',0,\overline{\Omega^{(N)}}}\times \mathbf{E}_{\mathbf{V}'',\mathbf{W}^{\bullet},\overline{\Omega^{(N)}}}$ with $\mathbf{E}_{\mathbf{V}'',\mathbf{W}^{\bullet},\overline{\Omega^{(N)}}}$. \\
	{{\rm{(a)}}} For any $L\in \mP_{\nu}(\lambda^{\bullet})$ with $t_{i}(L)=t$, we have 
	$$\mathbf{Res}^{\mathbf{V}\oplus\mathbf{W}^{\bullet}}_{\mathbf{V'},\mathbf{V}''\oplus\mathbf{W}^{\bullet}}(L)= K \oplus \bigoplus_{t_{i}(K')>0} K'[d']^{\oplus n(K',d')},$$ 
	where $K, K'\in \mP_{\nu''}(\lambda^{\bullet}),d'\in \mathbb{Z}, n(K',d')\in \mathbb{N}$ and $K$ is the unique direct summand satisfying $t_{i}(K)=0$.\\ 
	{{\rm{(b)}}} For any $K \in \mP_{\nu''}(\lambda^{\bullet})$ with $t_{i}(K)=0$, we have
	$$\mathbf{Ind}^{\mathbf{V}\oplus\mathbf{W}^{\bullet}}_{\mathbf{V}',\mathbf{V}''\oplus\mathbf{W}^{\bullet}}(\mathbb{C}
	_{\mathbf{E}_{\mathbf{V}',\bar{\Omega}}} \boxtimes K) = L \oplus \bigoplus_{t_{i}(L')>t} L'[d']^{\oplus n(L',d')}, $$ 
	where $L, L'\in \mP_{\nu}(\lambda^{\bullet}),d'\in \mathbb{Z}, n(L',d')\in \mathbb{N}$ and $L$ is the unique direct summand satisfying $t_{i}(L)=t$.\\
	{{\rm{(c)}}} The maps $L\mapsto K$ and $K\mapsto L$ define a bijection  $\pi_{i,t}:\{ L \in \mP_{\nu}(\lambda^{\bullet})| t_{i}(L)=t \} \rightarrow \{ K \in \mP_{\nu''}(\lambda^{\bullet})| t_{i}(K)=0 \} $.
\end{lemma}

As a corollary of Proposition \ref{Cry}, we have the following proposition.

\begin{proposition}\label{Cryt}
	The bijection $\Psi_{Q^{(N)}}$ of $Q^{(N)}$ restricts to an injective map, still denoted by $\Psi=\Psi_{Q^{(N)}}$, $$\Psi: \bigcup_{\nu \in \mathbb{N}I}\mathcal{P}_{\nu}(\lambda^{\bullet}) \rightarrow \bigcup_{\nu \in \mathbb{N}I} {\rm{Irr}} \Lambda_{\bV,\mathbf{W}^{\bullet}}$$ such that
	\begin{enumerate}
		\item For any simple perverse sheaf $L \in \mP_{\nu}(\lambda^{\bullet})$ and $i \in I$, $t_{i}(L)=\epsilon_{i}(\Psi(L))$ and $t^{\ast}_{i}(L)=\epsilon^{\ast}_{i}(\Psi(L))$.
		\item  For any simple perverse sheaf $L\in \mP_{\nu}(\lambda^{\bullet})$, we have $$\Psi(L) \subseteq SS(L) \subseteq  \bigcup_{\epsilon_{i}(Z') \geqslant t_{i}(L), \epsilon^{\ast}_{i}(Z') \geqslant t^{\ast}_{i}(L) }Z'.$$
		\item For any simple perverse sheaves $L \in \mP_{\nu}(\lambda^{\bullet})$ and $K \in \mP_{\nu''}(\lambda^{\bullet})$,  $\pi_{i,t}(L)=K$ if and only if $\rho_{i,t}(\Psi(L) )=\Psi(K). $
	\end{enumerate}
\end{proposition} 
We only have an injection but not a bijection, because simple objects in $\mP_{\nu}(\lambda^{\bullet})$ arise from certain flag types, but not from arbitary flag types.

\subsubsection{The analysis of $\mP^{\geqslant \mu}_{\nu}(\lambda^{\bullet})$ and $\mP^{> \mu}_{\nu}(\lambda^{\bullet})$}
In this subsection, we always assume that $wt^{-1}(\mu) \in \mathbb{N}I$ exists and $\mu \leqslant \sum_{1\leqslant l \leqslant N} \lambda^{l} $. 
\begin{lemma}\label{indlemma}
		For any $0 \leqslant t \leqslant \nu_{i}$ and $L\in \mP_{\nu}(\lambda^{\bullet})$ with $t_{i}(L)=t$, let $\nu''=\nu -ti$, we assume $ \pi_{i,t} (L)=K$.  Then $L \in \mP^{\geqslant \mu}_{\nu}(\lambda^{\bullet})$ if and only if $ K \in \mP^{\geqslant \mu}_{\nu''}(\lambda^{\bullet})$, and $L \in \mP^{> \mu}_{\nu}(\lambda^{\bullet})$ if and only if $ K \in \mP^{> \mu}_{\nu''}(\lambda^{\bullet})$.
\end{lemma}
\begin{proof}
	Since $\mP^{> \mu}_{\nu}(\lambda^{\bullet})$ is a finite union of $\mP^{\geqslant \mu'}_{\nu}(\lambda^{\bullet})$, we only need to prove the statement for $\mP^{\geqslant \mu}_{\nu}(\lambda^{\bullet})$.
	
	Assume $ K \in \mP^{\geqslant \mu}_{\nu''}(\lambda^{\bullet})$, then $SS(K) \subseteq \mathbf{E}^{\geqslant \mu}_{\bV'',\mathbf{W}^{\bullet},H^{(N)}}$. Since $\mathcal{F}^{(t)}_{i}$ is defined by the induction functor, $L$ is a direct summand of $\mathcal{F}^{(t)}_{i}(K)$, whose singular support is contained in $\mathbf{E}^{\geqslant \mu}_{\bV,\mathbf{W}^{\bullet},H^{(N)}}$. Hence $L \in \mP^{\geqslant \mu}_{\nu}(\lambda^{\bullet})$.
	
	Assume $L \in \mP^{\geqslant \mu}_{\nu}(\lambda^{\bullet})$, we claim that $\mathbf{Res}^{\mathbf{V}\oplus\mathbf{W}^{\bullet}}_{\mathbf{V'},\mathbf{V}''\oplus\mathbf{W}^{\bullet}}(L)$ has singular support contained in $\mathbf{E}^{\geqslant \mu}_{\bV'',\mathbf{W}^{\bullet},H^{(N)}} $.
	
	Fix a decomposition $\bV=\bV''\oplus \bV'$ such that dimension vector of $\bV'$ is $ti$ and  consider the following diagram which defines the functor $\mathbf{Res}^{\mathbf{V}\oplus\mathbf{W}^{\bullet}}_{\mathbf{V'},\mathbf{V}''\oplus\mathbf{W}^{\bullet}}$,
	$$ \mathbf{E}_{\bV'',\mathbf{W}^{\bullet},\overline{\Omega^{(N)}}} \xleftarrow{\kappa} F \xrightarrow{\iota}  \mathbf{E}_{\bV,\mathbf{W}^{\bullet},\overline{\Omega^{(N)}}},$$
	where $F$ is the subset of $\mathbf{E}_{\bV,\mathbf{W}^{\bullet},\overline{\Omega^{(N)}}}$ which consists of those $(\bar{x},z)$ such that $\bV'' \oplus \mathbf{W}^{\bullet}$ is $(\bar{x},z)$-stable. Then the cotangent correspondence of $\iota$ is 
	$$T^{\ast} F \xleftarrow{d\iota'} F \times_{\mathbf{E}_{\bV,\mathbf{W}^{\bullet},\overline{\Omega^{(N)}}} }\mathbf{E}_{\bV,\mathbf{W}^{\bullet},H^{(N)}}  \xrightarrow{\iota_{\pi} } \mathbf{E}_{\bV,\mathbf{W}^{\bullet},H^{(N)}}, $$ 
    and the cotangent correspondence of $\kappa$ is
    $$T^{\ast} F \xleftarrow{ d\kappa '} F \times_{\mathbf{E}_{\bV'',\mathbf{W}^{\bullet},\overline{\Omega^{(N)}}} }\mathbf{E}_{\bV'',\mathbf{W}^{\bullet},H^{(N)}} \xrightarrow{\kappa_{\pi}} \mathbf{E}_{\bV,\mathbf{W}^{\bullet},H^{(N)}}. $$
    
    By \cite[Theorem 3.2]{Nasing}, the singular support of $\mathbf{Res}^{\mathbf{V}\oplus\mathbf{W}^{\bullet}}_{\mathbf{V'},\mathbf{V}''\oplus\mathbf{W}^{\bullet}}(L)$ is contained in the subset $ \kappa_{\pi}(d\kappa')^{-1} d\iota' \iota^{-1}_{\pi} ( \Lambda_{\bV,\mathbf{W}^{\bullet}} ). $ Notice that any point $(x',\bar{x}',y',z')$ in  $ \kappa_{\pi}(d\kappa')^{-1} d\iota' \iota^{-1}_{\pi} ( \Lambda_{\bV,\mathbf{W}^{\bullet}}) $ stablizes the subspace $\bV'' \oplus \mathbf{W}^{\bullet}$. If $(x,\bar{x},y,z)$ in $\mathbf{E}^{\geqslant \mu}_{\bV,\mathbf{W}^{\bullet},H^{(N)}} \cap  \Lambda_{\bV,\mathbf{W}^{\bullet}}$, we choose $\mathbf{T}$ such that its dimension vector $\leqslant wt^{-1}(\mu)$ and $\mathbf{T}\oplus \mathbf{W}^{\bullet}$ is $(x,\bar{x},y,z)$-stable. For any point $(x',\bar{x}',y',z')$ in the subset $ \kappa_{\pi}(d\kappa')^{-1} d\iota' \iota^{-1}_{\pi} ((x,\bar{x},y,z))$, we take $\mathbf{T}'= \mathbf{T}\cap \bV''$, then the dimension vector of $\mathbf{T}'$ is $\leqslant wt^{-1}(\mu)$ and $\mathbf{T}'\oplus \mathbf{W}^{\bullet}$ is $(x',\bar{x}',y',z')$-stable. Hence $ \kappa_{\pi}(d\kappa')^{-1} d\iota' \iota^{-1}_{\pi} ( \mathbf{E}^{\geqslant \mu}_{\bV,\mathbf{W}^{\bullet},H^{(N)}} \cap \Lambda_{\bV,\mathbf{W}^{\bullet}}) \subseteq \mathbf{E}^{\geqslant \mu}_{\bV'',\mathbf{W}^{\bullet},H^{(N)}} $ and we finish the proof.
\end{proof}

\begin{lemma}\label{casexiao}
	If $\nu < wt^{-1}(\mu)$, then
	$$\mD_{\nu}^{\geqslant \mu}(\lambda^{\bullet})=\mD_{\nu}^{> \mu}(\lambda^{\bullet})=\mD_{\nu}(\lambda^{\bullet}), $$
	 $$\mP_{\nu}^{\geqslant \mu}(\lambda^{\bullet})=\mP_{\nu}^{> \mu}(\lambda^{\bullet})=\mP_{\nu}(\lambda^{\bullet}). $$
\end{lemma}
\begin{proof}
	Notice that when $\nu < wt^{-1}(\mu)$, we can take $\bV''=\bV$ for any point $(x,\bar{x},y,z)$ in $\mathbf{E}_{\bV'',\mathbf{W}^{\bullet},H^{(N)}}$, hence $\mathbf{E}^{\geqslant \mu}_{\bV,\mathbf{W}^{\bullet},H^{(N)}}= \mathbf{E}^{> \mu}_{\bV,\mathbf{W}^{\bullet},H^{(N)}}=\mathbf{E}_{\bV,\mathbf{W}^{\bullet},H^{(N)}}$.
\end{proof}

\begin{lemma}\label{casenocompare}
	If $\nu$ and $wt^{-1}(\mu)$ can not be compared, then 
	$$\mD_{\nu}^{\geqslant \mu}(\lambda^{\bullet})=\mD_{\nu}^{> \mu}(\lambda^{\bullet}), $$
	$$\mP_{\nu}^{\geqslant \mu}(\lambda^{\bullet})=\mP_{\nu}^{> \mu}(\lambda^{\bullet}). $$
\end{lemma}
\begin{proof}
	Assume $\bV''$ is a subspace of $\bV$ such that its dimension vector $\nu'' \leqslant wt^{-1}(\mu)$ and $\bV''\oplus \mathbf{W}^{\bullet}$ is  $(x,\bar{x},y,z)$-stable for some   $(x,\bar{x},y,z)$ in $\mathbf{E}^{\geqslant \mu}_{\bV'',\mathbf{W}^{\bullet},H^{(N)}}$, then we must have $\nu'' < wt^{-1}(\mu)$. Otherwise, $wt^{-1}(\mu)= \nu'' \leqslant \nu$ can be compared.  Hence
	$(x,\bar{x},y,z)$ belongs to $\mathbf{E}^{\geqslant \mu'}_{\bV'',\mathbf{W}^{\bullet},H^{(N)}} \subseteq  \mathbf{E}^{> \mu}_{\bV'',\mathbf{W}^{\bullet},H^{(N)}}$ for $\mu'=\mu +wt^{-1}(\mu)-\nu'' > \mu$, and
	 $\mathbf{E}^{\geqslant \mu}_{\bV,\mathbf{W}^{\bullet},H^{(N)}}= \mathbf{E}^{> \mu}_{\bV,\mathbf{W}^{\bullet},H^{(N)}}$. We finish the proof.
\end{proof}

\begin{lemma}\label{caseda}
	For $\nu > wt^{-1}(\mu)$ and $L \in \mP_{\nu}^{\geqslant \mu}(\lambda^{\bullet})$, there exists $i \in I$ such thta $t_{i}(L)>0$.
\end{lemma}

\begin{proof}
	By Proposition \ref{Cryt} and $L \in \mP_{\nu}^{\geqslant \mu}(\lambda^{\bullet})$, we can see that $SS(L) \subseteq \Lambda^{\geqslant \mu}_{\bV,\mathbf{W}^{\bullet}}= \Lambda_{\bV,\mathbf{W}^{\bullet}} \cap  \mathbf{E}^{\geqslant \mu}_{\bV'',\mathbf{W}^{\bullet},H^{(N)}}$.

	We claim that when $\nu > wt^{-1}(\mu)$, $\Lambda^{\geqslant \mu}_{\bV,\mathbf{W}^{\bullet}} \subseteq \bigcup_{i \in I, t>0} \Lambda_{\bV,\mathbf{W}^{\bullet},i,t}.$ In particular, there exits $i \in I$ such that $\epsilon_{i}(\Phi(L))=t>0$ and hence $t_{i}(L)>0$. 
	
	Given arbitrary $(x,\bar{x},y,z)$ in $\Lambda^{\geqslant \mu}_{\bV,\mathbf{W}^{\bullet}}$, take a subspace $\bV''$ of dimension vector $\leqslant wt^{-1}(\mu)$ such that $\bV'' \oplus \mathbf{W}^{\bullet}$ is $(x,\bar{x},y,z)$. Then $(x,\bar{x},y,z)$ induces a quotient  $(x',\bar{x}',y',z')$  with dimension vector $\nu' \in \mathbb{N}I$ with $\nu' \geqslant \nu- wt^{-1}(\mu) >0$. This quotient $(x',\bar{x}',y',z')$ is a nilpotent representation in Lusztig's nilpotent variety, hence it admits a simple quotient supported at some single vertex $i \in I$ by \cite[Lemma 12.6]{MR1088333}. In particular, it means that $(x,\bar{x},y,z)$ has a simple quotient supported at some single vertex $i \in I$ and $(x,\bar{x},y,z) \in \bigcup_{t \geqslant 1} \Lambda_{\bV,\mathbf{W}^{\bullet},i,t}$. 
\end{proof}

\begin{lemma}\label{casedeng}
	For $\nu=wt^{-1}(\mu)$ and $L \in \mP_{\nu}^{\geqslant \mu}(\lambda^{\bullet})$, then $L \notin \mP_{\nu}^{> \mu}(\lambda^{\bullet})$ if and only if $t_{i}(L)=0$ for any $i \in I$.
\end{lemma}

\begin{proof}
	If $L \in \mP_{\nu}^{> \mu}(\lambda^{\bullet})$, then $L \in\mP_{\nu}^{\geqslant \mu'}(\lambda^{\bullet}) $ for some $\mu' > \mu$. By the lemma above, $\nu > wt^{-1}(\mu')$ implies that there exists some $i$ such that $t_{i}(L)>0$.
	
	Conversely, if $t_{i}(L)=t > 0$, then by Proposition \ref{Cryt}, $SS(L) \subseteq \bigcup_{t' \geqslant t} \Lambda_{\bV,\mathbf{W}^{\bullet},i,t'}$. In particular, any $(x,\bar{x},y,z) \in SS(L)$ admits a quotient $S^{\oplus t}_{i}$, where $S_{i}$ is the simple representation of the doubled quiver supported at $i \in I$. Hence any $(x,\bar{x},y,z) \in SS(L)$ admits a invariant subspace $\bV'' \oplus \mathbf{W}^{\bullet}$ such that $\bV''$ has dimension vector $\nu -ti$. It implies that $L$ belongs to $\mP_{\nu}^{\geqslant \mu+ti}(\lambda^{\bullet}) \subseteq \mP_{\nu}^{> \mu}(\lambda^{\bullet})$. 
\end{proof}

When $ \nu=wt^{-1}(\mu)$, we denote the set  $ \mP_{\nu}^{\geqslant \mu}(\lambda^{\bullet})\backslash \mP_{\nu}^{> \mu}(\lambda^{\bullet})$ by $\mP^{\mu,hi}$.

\subsubsection{The subquotients of the tensor product}
 
We denote the following direct sum of Grothendieck groups by 
$$\mK^{\mu}(\lambda^{\bullet}) =\bigoplus_{\nu \in \mathbb{N}I}\mK(\mL^{\mu}_{\nu}(\lambda^{\bullet})),$$
$$\mK^{>\mu}(\lambda^{\bullet}) =\bigoplus_{\nu \in \mathbb{N}I}\mK(\mL^{>\mu}_{\nu}(\lambda^{\bullet})),$$
$$\mK^{\geqslant \mu}(\lambda^{\bullet}) =\bigoplus_{\nu \in \mathbb{N}I}\mK(\mL^{\geqslant\mu}_{\nu}(\lambda^{\bullet})).$$
In this subsection, we prove that the Grothendieck groups $\mK^{\mu}(\lambda^{\bullet}),\mK^{>\mu}(\lambda^{\bullet})$ and $\mK^{\geqslant \mu}(\lambda^{\bullet})$  are canonically isomorphic the based module $M[\geqslant \mu]/M[> \mu],M[> \mu]$ and $M[\geqslant \mu]$ respectively, where $M$ is the tensor product of $L(\lambda^{1}),L(\lambda^{2}),\cdots, L(\lambda^{N})$.

Notice that the inclusion functors 
$$ \mL^{>\mu}_{\nu}(\lambda^{\bullet}) \rightarrow \mL_{\nu} (\lambda^{\bullet}),$$
$$ \mL^{\geqslant \mu}_{\nu}(\lambda^{\bullet}) \rightarrow \mL_{\nu} (\lambda^{\bullet}),$$
and the localization functor
$$ \mL^{\geqslant \mu}_{\nu}(\lambda^{\bullet}) \rightarrow \mL^{\mu}_{\nu} (\lambda^{\bullet})$$
commute with functors $\mathcal{E}^{(r)}_{i},\mathcal{F}^{(r)}_{i}$, hence they induce injective morphisms of $_{\mathcal{A}}\mathbf{U}$-modules
\begin{equation}\label{ig}
	\mK^{>\mu}(\lambda^{\bullet}) \rightarrow \mK(\lambda^{\bullet}),
\end{equation}
\begin{equation}\label{ige}
	\mK^{\geqslant \mu}(\lambda^{\bullet}) \rightarrow \mK(\lambda^{\bullet}),
\end{equation}
and surjective morphism
\begin{equation}\label{q}
	\mK^{\geqslant \mu}(\lambda^{\bullet}) \rightarrow \mK^{\mu}(\lambda^{\bullet}) .
\end{equation}

\begin{lemma} \label{subquotient}
	With the actions of $\mathcal{E}^{(r)}_{i},\mathcal{F}^{(r)}_{i}, i\in I$, the Grothendieck group $\mK^{\mu}(\lambda^{\bullet})$ is an integrable $_{\mathcal{A}}\mathbf{U}$-module, and is isomorphic to the direct sum of $n_{\mu}$ copies of the irreducible highest weight module $L(\mu)$,
	$$ \mK^{\mu}(\lambda^{\bullet}) \rightarrow {_{\mathcal{A}}L}(\mu)^{\oplus n_{\mu}}, $$ 
	where $n_{\mu}$ is the number of simple perverse sheaves in $ S_{1}(\mu)=\mP^{\mu,hi} \backslash (\mP^{\mu,hi} \cap \mN_{wt^{-1}(\mu)})$.
	Moreover, if a simple perverse sheave in $S_{1}(\mu)$, then its image in the Grothendieck group provides a highest weight vector of a direct summand. 
\end{lemma}

\begin{proof}
By Lemma \ref{casexiao} and \ref{casenocompare}, the $\mathcal{A}$-module $\mK(\mL^{\mu}_{\nu}(\lambda^{\bullet}))=0$ if $\nu < wt^{-1}(\mu)$ or $\nu$ and $wt^{-1}(\mu)$ can not be compared. We only need to study $\mK(\mL^{\mu}_{\nu}(\lambda^{\bullet}))$ for $\nu \geqslant wt^{-1}(\mu) $.

We claim that $\mK^{\mu}(\lambda^{\bullet})$ equals to the $_{\mathcal{A}}\mathbf{U}^{-}$-module generated by images of simple perverse sheaves in $\mP^{\mu,hi} \backslash (\mP^{\mu,hi} \cap \mN_{wt^{-1}(\mu)})$.
\begin{enumerate}
	\item[(1)] If $\nu = wt^{-1}(\mu)$, by Lemma \ref{casedeng}, $L \in \mP^{\geqslant \mu}_{\nu}(\lambda^{\bullet})$ is nonzero in $\mL^{\mu}_{\nu}(\lambda^{\bullet})$ if and only if it belongs to $\mP^{\mu,hi} \backslash (\mP^{\mu,hi} \cap \mN_{wt^{-1}(\mu)})$. 
    \item[(2)]If $\nu > wt^{-1}(\mu)$ and $L \in \mL^{\mu}_{\nu}(\lambda^{\bullet})$ is a nonzero simple perverse sheaf, then by Lemma \ref{caseda}, there exists $i \in I $ such that $t_{i}(L)=t>0$. By Lemma \ref{lkeyt}, there is some $K$ in $\mP_{\nu'}(\lambda^{\bullet})$ such that $$\mathcal{F}^{(t)}_{i}K=L \oplus \bigoplus_{t_{i}(L')>t} L'[d']^{\oplus n(L',d')}.$$
	
	Notice that $K$ doesn't belong to $\mD^{> \mu}_{\nu'}(\lambda^{\bullet})$ or $\mN_{\nu'}$. Otherwise, the fact $\mathcal{F}^{(t)}_{i}K$ belongs to $\mD^{> \mu}_{\nu}(\lambda^{\bullet})$ or $\mN_{\nu}$ implies that $L \in \mL^{\mu}_{\nu}(\lambda^{\bullet})$ is a zero object. 
	
	By increasing induction on $\nu$ and decreasing induction on $t$, we can prove that the image $[L]$ of $L$ in  $\mK^{\mu}(\lambda^{\bullet})$ is generated by some simple perverse sheaf $L_{0}$ with $t_{i}(L_{0})=0$ for all $i$, and $L_{0}$ does not belongs to $\mD^{> \mu}_{\nu'}(\lambda^{\bullet})$ and $\mN_{\nu'}$. It forces that $\nu' = wt^{-1}(\mu)$ and $L_{0} \in \mP^{\mu,hi} \backslash (\mP^{\mu,hi} \cap \mN_{wt^{-1}(\mu)})$.
\end{enumerate}

Since $\mK^{\mu}(\lambda^{\bullet})$ equals to the $_{\mathcal{A}}\mathbf{U}^{-}$-module generated by images of simple perverse sheaves in $S_{1}(\mu)$, and it is obvious that simple perverse sheaves in $S_{1}(\mu)$ provide highest weight vectors in Grothendieck groups, the statement follows from the fact that $\mK^{\mu}(\lambda^{\bullet})$ is an integrable module.
\end{proof}

\begin{theorem}\label{main1}
Let	$M$ be the tensor product  $L(\lambda^{N})\otimes L(\lambda^{N-1})\otimes \cdots\otimes L(\lambda^{1})$ and $B$ be its canonical basis in the sense of Theorem \ref{CBT}.
 Given any dominant weight $\mu \leqslant \sum_{1\leqslant l \leqslant N} \lambda^{l}$ such that $wt^{-1}(\mu)$ exists, we have the following statements:
 
 (1) Under the composition of $\chi^{\lambda^{\bullet}}$ in Theorem \ref{tensor}(2) and the morphism defined by formula (\ref{ig}), the Grothendieck group $\mK^{>\mu}(\lambda^{\bullet})$ is canonically  isomorphic to the module $_{\mathcal{A}} M[> \mu]$. 
 We still denote the composition by $\chi^{\lambda^{\bullet}} $, then there is a commutative diagram
 \[
 \xymatrix{
 \mK^{>\mu}(\lambda^{\bullet}) \ar[r]^{(\ref{ig})} \ar[d]_{\chi^{\lambda^{\bullet}}} &  \mK(\lambda^{\bullet}) \ar[d]^{\chi^{\lambda^{\bullet}}} \\
 _{\mathcal{A}}M[> \mu] \ar[r] &  _{\mathcal{A}}M.
 }
 \]
 Moreover, the images of nonzero simple perverse sheaves in $\mL^{>\mu}_{\nu}(\lambda^{\bullet}),\nu \in \mathbb{N}I$ provide an $\mathcal{A}$-basis, which coincides with $B \cap M[> \mu]$.
 
 (2) Under the composition of $\chi^{\lambda^{\bullet}}$ and the morphism defined by formula (\ref{ige}), the Grothendieck group $\mK^{\geqslant \mu}(\lambda^{\bullet})$ is canonically  isomorphic to the module $_{\mathcal{A}} M[\geqslant  \mu]$, and  there is a commutative diagram
  \[
 \xymatrix{
 	\mK^{\geqslant\mu}(\lambda^{\bullet}) \ar[r]^{(\ref{ige})} \ar[d]_{\chi^{\lambda^{\bullet}}} &  \mK(\lambda^{\bullet}) \ar[d]^{\chi^{\lambda^{\bullet}}} \\
 	_{\mathcal{A}}M[\geqslant \mu] \ar[r] &  _{\mathcal{A}}M.
 	}
 \]
 
  Moreover, the images of nonzero simple perverse sheaves in $\mL^{\geqslant \mu}_{\nu}(\lambda^{\bullet}),\nu \in \mathbb{N}I$ provide an $\mathcal{A}$-basis, which coincides with $B \cap M[\geqslant \mu]$.
 
 (3)  The canonical isomorphism $\chi^{\lambda^{\bullet}}$ induced a canonical isomorphism  $\mK^{ \mu}(\lambda^{\bullet}) \rightarrow  _{\mathcal{A}} (M[\geqslant  \mu]/M[>\mu])$. We still denote this isomorphism by $\chi^{\lambda^{\bullet}}$, then there is a commutative diagram
   \[
 \xymatrix{
 	\mK^{\geqslant\mu}(\lambda^{\bullet}) \ar[r]^{(\ref{q})} \ar[d]_{\chi^{\lambda^{\bullet}}} &  \mK^{\mu}(\lambda^{\bullet}) \ar[d]^{\chi^{\lambda^{\bullet}}} \\
 	_{\mathcal{A}}M[\geqslant \mu] \ar[r]^-{\pi_{\geqslant \mu}} &  _{\mathcal{A}}(M[\geqslant \mu]/M[> \mu]).
 }
 \]
 
 Moreover, the images of nonzero simple perverse sheaves in $\mL^{ \mu}_{\nu}(\lambda^{\bullet}),\nu \in \mathbb{N}I$ provide an $\mathcal{A}$-basis, which coincides with $\pi_{\geqslant \mu}(B \cap M[\geqslant \mu] )$.
 
\end{theorem}

\begin{proof}
	Since morphisms (\ref{ig}) (\ref{ige}) and (\ref{q}) send the image of a simple perverse sheaf to itself or zero, they are morphisms of based module. Hence it suffices to show the  Grothendieck groups are isomorphic to the corresponding based modules. 
	
	If $\mu = \sum_{1\leqslant l \leqslant N} \lambda^{l}$, then $\mD^{> \mu}_{\nu}(\lambda^{\bullet})$ is empty and $\mL^{\geqslant \mu}_{\nu}(\lambda^{\bullet})= \mL^{\mu}_{\nu}(\lambda^{\bullet})$. By Lemma \ref{subquotient}, we can see that $$\mK^{\geqslant\mu}(\lambda^{\bullet}) = \mK^{\mu}(\lambda^{\bullet}) \cong {_{\mathcal{A}}L}( \sum_{1\leqslant l \leqslant N} \lambda^{l})= {_{\mathcal{A}}M}[\geqslant \sum_{1\leqslant l \leqslant N} \lambda^{l}] .$$
	
	We assume that for any $\mu' > \mu$, $\mK^{\geqslant \mu'}(\lambda^{\bullet})$ is canonically  isomorphic to a submodule of the module $_{\mathcal{A}} M[\geqslant  \mu']$ via  $\chi^{\lambda^{\bullet}}$. Since $\mK^{> \mu}(\lambda^{\bullet})=\sum_{\mu'>\mu}\mK^{\geqslant \mu'}(\lambda^{\bullet})$ in $\mK(\lambda^{\bullet})$, we can see that $\mK^{> \mu}(\lambda^{\bullet})$ is canonically  isomorphic to a submodule of the module $_{\mathcal{A}} M[> \mu]$ via  $\chi^{\lambda^{\bullet}}$.
	
	Notice that the inclusion functor and the localization functor induces a short exact sequence
	$$ 0 \rightarrow \mK^{> \mu}(\lambda^{\bullet}) \rightarrow \mK^{\geqslant  \mu}(\lambda^{\bullet}) \rightarrow \mK^{ \mu}(\lambda^{\bullet}) \rightarrow 0. $$
	By Lemma \ref{subquotient} and the exact sequence above, the Grothendieck group $\mK^{\geqslant  \mu}(\lambda^{\bullet})$ decomposes as a direct sum of highest weight modules with the highest weight $\geqslant \mu$, hence it is canonically isomorphic to  a submodule of the module $_{\mathcal{A}} M[\geqslant \mu]$ via  $\chi^{\lambda^{\bullet}}$. Finally, we can prove that for any dominant weight $\mu \leqslant \sum_{1\leqslant l \leqslant N} \lambda^{l}$ such that $wt^{-1}(\mu)$ exists, $\mK^{\geqslant  \mu}(\lambda^{\bullet})$ is canonically isomorphic to  a submodule of the module $_{\mathcal{A}} M[\geqslant \mu]$. 
	
	For $\mu > \mu' $, we claim that the multiplicity of $L(\mu)$ in $\mK^{\geqslant  \mu'}(\lambda^{\bullet})$ or $\mK^{> \mu'}(\lambda^{\bullet})$ is exactly $n_{\mu}$ in Lemma \ref{subquotient}. Notice that for any $\nu \leqslant wt^{-1}(\mu)<wt^{-1}(\mu')$, we have $\mD^{\geqslant \mu}_{\nu}(\lambda^{\bullet})=\mD^{\geqslant \mu'}_{\nu}(\lambda^{\bullet})=\mD^{> \mu'}_{\nu}(\lambda^{\bullet})=\mD_{\nu}(\lambda^{\bullet}).$
	Consider the inclusion maps $$\mK^{\geqslant \mu} \rightarrow \mK^{\geqslant  \mu'}(\lambda^{\bullet}),$$
	$$\mK^{\geqslant \mu} \rightarrow \mK^{>  \mu'}(\lambda^{\bullet}),$$
	the argument above implies that the quotients of these maps can not contain direct summands isomorphic to $L(\mu'')$ with $\mu'' \geqslant \mu$. In particular, the multiplicity of $L(\mu)$ in $\mK^{\geqslant  \mu'}(\lambda^{\bullet})$ or $\mK^{> \mu'}(\lambda^{\bullet})$ is equal to the  multiplicity of $L(\mu)$ in $\mK^{\geqslant  \mu}(\lambda^{\bullet})$, which is $n_{\mu}$ by the short exact sequence above.
	
	Notice that when $wt(\nu) \geqslant 0$, $\mD_{\nu}^{\geqslant 0}(\lambda^{\bullet})$ is the whole triangulated category, hence for any $\mu \geqslant 0$, the $\mu$-weight spaces of  $\mK^{\geqslant  0}(\lambda^{\bullet})$ and $\mK(\lambda^{\bullet})$  are the same. In particular, $\mK^{\geqslant  0}(\lambda^{\bullet})=\mK(\lambda^{\bullet})$ is isomorphic to $_{\mathcal{A}} M={_{\mathcal{A}}M}[\geqslant 0]$.  From the proof, we can see that the inclusion $\chi^{\lambda^{\bullet}} (\mK^{\geqslant  \mu}(\lambda^{\bullet}) )  \subseteq {_{\mathcal{A}} M}[\geqslant \mu]$ should be  $\chi^{\lambda^{\bullet}} (\mK^{\geqslant  \mu}(\lambda^{\bullet}) )  = {_{\mathcal{A}} M}[\geqslant \mu]$ for any $\mu$. We are done.
\end{proof}

In particular, the category $\mL^{0}_{\nu}(\lambda^{\bullet})$ provides a sheaf theoretic realization of coinvariants and their distinguished basis, which generalize the results in \cite{CB2}. As mentioned in \cite{MR1227098}, if we take $N=2k$ and assume  for any $1\leqslant l \leqslant N$, the weight $\lambda^{l}=\lambda$ is self dual,i.e. $L(\lambda) \cong L(\lambda)^{\ast}$, then simple perverse sheaves in $\mL^{\geqslant 0}_{\nu}(\lambda^{\bullet})$ provide a distinguished (dual)  basis for the algebra  ${\textrm{End}}_{\mathbf{U}(\mathfrak{g})}(L(\lambda)^{\otimes k})$, which is similar to Kazhdan-Lusztig bases defined in \cite{KL}.

As an corollary, we also know that the decomposition number  is determined by  perverse sheaves.

\begin{corollary}
	We have the following decomposition of $\mathbf{U}$-module 
	$$ L(\lambda^{N})\otimes L(\lambda^{N-1})\otimes \cdots\otimes L(\lambda^{1}) \cong \bigoplus_{\mu} L(\mu)^{\oplus n_{\mu}}, $$
	where $n_{\mu}$ is the number of simple perverse sheaves in $S_{1}(\mu)=\mP^{\mu,hi} \backslash (\mP^{\mu,hi} \cap \mN_{wt^{-1}(\mu)})$.
\end{corollary}

\subsection{Generalization to restricted modules}

Let $J$ be a subset of $I$ and denote the associated Levi subalgebra by $\mathfrak{l}=\mathfrak{l}_{J}$. We assume $M$ is the restriction $\mathbf{res}^{\mathfrak{g}}_{\mathfrak{l}} L(\lambda^{N})\otimes L(\lambda^{N-1})\otimes \cdots\otimes L(\lambda^{1})$ of the tensor product of $L(\lambda^{1}),L(\lambda^{2}),\cdots, L(\lambda^{N})$. Together with the canonical basis $B=B(\lambda^{\bullet})$ of $L(\lambda^{N})\otimes L(\lambda^{N-1})\otimes \cdots\otimes L(\lambda^{1})$,  it becomes a based $\mathbf{U}(\mathfrak{l})$-module. In this section, we give a geometric realization of $M[\geqslant \mu_{J}]$ and $M[> \mu_{J}]$ for given dominant integral weight $\mu_{J}$ of $\mathfrak{l}$.

\subsubsection{The thick subcategory $\mD^{\geqslant \mu_{J}},\mD^{>\mu_{J}}$ and localizations}

For any dimension vector $\nu \in \mathbb{N}I$,  we can regard  vertices in $I \backslash J$ as framing, and define its associated $\mathfrak{l}$-weight $wt_{J}(\nu)$. More precisely, we define
$$wt_{J}(\nu)=\sum_{1\leqslant l \leqslant N, j\in J}\lambda^{l}_{j} \Lambda_{j} + \sum_{h \in H,h'=j \in J,h'' \notin J}  \nu_{h''}\Lambda_{j}   -\sum_{j \in J} \nu_{j} j. $$  
In particular, we can see that $wt_{I}=wt$.

\begin{definition}
	Given a fixed dominant weight $\mu_{J}$ of $\mathfrak{l}$, we say a point $(x,\bar{x},y,z)\in \mathbf{E}_{\bV,\mathbf{W}^{\bullet},H^{(N)}}$ is dominated by a dominant weight $\geqslant \mu_{J}$, if there exists a subspace $\bV''$  of $\bV$ such that  $\bV'' \oplus \mathbf{W}^{\bullet}$ is  $(x,\bar{x},y,z)$-stable, and its dimension vector $\nu''$ satisfies the following conditions
	\begin{enumerate}
		\item[(1)] $\nu''_{i}=\nu_{i}$ for any $ i\in I$ but $i \notin J$,
		\item[(2)] $wt_{J}(\nu'')\geqslant \mu_{J}. $
	\end{enumerate}
	
	Define $\mathbf{E}^{\geqslant \mu_{J}}_{\bV,\mathbf{W}^{\bullet},H^{(N)}}$ to be the subset of $\mathbf{E}_{\bV,\mathbf{W}^{\bullet},H^{(N)}}$, which consists of those points  dominated by the dominant weight $\geqslant \mu_{J}$, and
	we also denote $\mathbf{E}^{> \mu_{J}}_{\bV,\mathbf{W}^{\bullet},H^{(N)}}=\bigcup_{\mu_{J}'>\mu_{J}} \mathbf{E}^{\geqslant \mu'}_{\bV,\mathbf{W}^{\bullet},H^{(N)}}.$
	
	We say a point $(x,\bar{x},y,z)$ in $\mathbf{E}_{\bV,\mathbf{W}^{\bullet},H^{(N)}}$ is strictly dominated by  $\mu_{J}$ if it belongs to the subset  $\mathbf{E}^{\geqslant \mu_{J}}_{\bV,\mathbf{W}^{\bullet},H^{(N)}} \backslash \mathbf{E}^{> \mu_{J}}_{\bV,\mathbf{W}^{\bullet},H^{(N)}}$
\end{definition}

By a similar argument of Lemma \ref{close}, one can prove that $\mathbf{E}^{\geqslant \mu_{J}}_{\bV,\mathbf{W}^{\bullet},H^{(N)}}$ and $\mathbf{E}^{> \mu_{J}}_{\bV,\mathbf{W}^{\bullet},H^{(N)}}$ are close subsets of $\mathbf{E}_{\bV,\mathbf{W}^{\bullet},H^{(N)}}$. In particular, we can define  thick subcategories and localizations by micro-local conditions. 

\begin{definition}
	Assume that $\mu_{J}$ is a  dominant integral weight  of $\mathfrak{l}$  and $\nu \in \mathbb{N}I$ is a dimension vector.
	\begin{enumerate}
		\item [(1)]	Let $\mD_{\nu}^{\geqslant \mu_{J}}(\lambda^{\bullet})$ be the full subcategory of $\mD_{\nu}(\lambda^{\bullet})$, which consists of objects whose singular supports are contained in $\mathbf{E}^{\geqslant \mu_{J}}_{\bV,\mathbf{W}^{\bullet},H^{(N)}}$. Similarly, let $\mD_{\nu}^{> \mu_{J}}(\lambda^{\bullet})$ be the full subcategory of $\mD_{\nu}(\lambda^{\bullet})$, which consists of objects whose singular supports are contained in $\mathbf{E}^{> \mu_{J}}_{\bV,\mathbf{W}^{\bullet},H^{(N)}}$.
		\item[(2)] 	We denote $\mP^{\geqslant \mu_{J}}_{\nu}(\lambda^{\bullet})$ and  $\mP^{>\mu_{J}}_{\nu}(\lambda^{\bullet})$ by the subsets of $\mP_{\nu}(\lambda^{\bullet})$, which consist of simple perverse sheaves in $\mD_{\nu}^{\geqslant \mu_{J}}(\lambda^{\bullet})$ and $\mD_{\nu}^{> \mu_{J}}(\lambda^{\bullet})$ respectively.
		\item [(3)] 
		We define $\mD_{\nu}^{\geqslant \mu_{J}}(\lambda^{\bullet})/\mN_{\nu}$ to be the full subcategory of $\mD_{\nu}(\lambda^{\bullet})/\mN_{\nu}$, which consists of objects isomorphic to objects of $\mD_{\nu}^{\geqslant \mu_{J}}(\lambda^{\bullet})$. Similarly, define $\mD_{\nu}^{> \mu_{J}}(\lambda^{\bullet})/\mN_{\nu}$ to be the full subcategory of $\mD_{\nu}(\lambda^{\bullet})/\mN_{\nu}$, which consists of objects isomorphic to objects of $\mD_{\nu}^{> \mu_{J}}(\lambda^{\bullet})$.
		\item [(4)] Let $\mN_{\nu}^{>\mu_{J}}$ be the thick subcategory of $\mD_{\nu}(\lambda^{\bullet})$ generated by $\mN_{\nu}$ and $\mD^{>\mu_{J}}_{\nu}(\lambda^{\bullet})$, define the localization $\mD_{\nu}(\lambda^{\bullet})/\mN_{\nu}^{>\mu_{J}}$ to be the Verdier quotient of $\mD_{\nu}(\lambda^{\bullet})$ with respect to the thick subcategory.
	\end{enumerate}
\end{definition}

We can regard $I \backslash J$ as framed vertices, then by  similar argument as  Proposition \ref{keyprop}, \ref{keyprop2} and \ref{keyprop3}, we have the following proposition.
\begin{proposition}\label{keypropg}
	For any $j\in J,r\in \mathbb{N}_{>0}$ and $\nu=rj+\nu''\in\mathbb{N}I$, the functors 
	$$\mathcal{E}^{(r)}_{j}:\mD_{\nu}^{\geqslant \mu_{J}}(\lambda^{\bullet})/\mN_{\nu} \rightarrow \mD_{\nu''}^{\geqslant \mu_{J}}(\lambda^{\bullet})/\mN_{\nu''},$$
	$$\mathcal{E}^{(r)}_{j}:\mD_{\nu}^{> \mu_{J}}(\lambda^{\bullet})/\mN_{\nu} \rightarrow \mD_{\nu''}^{> \mu_{J}}(\lambda^{\bullet})/\mN_{\nu''}$$
	and
	$$\mathcal{F}^{(r)}_{j}:\mD_{\nu''}^{\geqslant \mu_{J}}(\lambda^{\bullet})/\mN_{\nu''} \rightarrow \mD_{\nu}^{\geqslant \mu_{J}}(\lambda^{\bullet})/\mN_{\nu},$$
	$$\mathcal{F}^{(r)}_{j}:\mD_{\nu''}^{> \mu_{J}}(\lambda^{\bullet})/\mN_{\nu''} \rightarrow \mD_{\nu}^{> \mu_{J}}(\lambda^{\bullet})/\mN_{\nu}$$
	are well-defined. 
	They also induce well-defined functors
	$$\mathcal{E}^{(r)}_{j}:\mD_{\nu}(\lambda^{\bullet})/\mN_{\nu}^{> \mu_{J}} \rightarrow \mD_{\nu''}(\lambda^{\bullet})/\mN_{\nu''}^{>\mu_{J}},$$
	$$\mathcal{F}^{(r)}_{j}:\mD_{\nu''}(\lambda^{\bullet})/\mN_{\nu''}^{>\mu_{J}} \rightarrow \mD_{\nu}(\lambda^{\bullet})/\mN_{\nu}^{>\mu_{J}}.$$
\end{proposition}

\begin{definition}
	For any dominant integral weight $\mu_{J}$ of $\mathfrak{l}$ and dimension vector $\nu \in \mathbb{N}I$, we define the following localizations. 
	\begin{enumerate}
		\item [(1)] 	Let $\mL^{\geqslant \mu_{J}}_{\nu}(\lambda^{\bullet})$ be the full subcategory of $\mD_{\nu}^{\geqslant \mu_{J}}(\lambda^{\bullet})/\mN_{\nu}$, which consists of objects isomorphic to objects of $\mQ_{\nu}(\lambda^{\bullet})$.
		\item [(2)]  Let $\mL^{> \mu_{J}}_{\nu}(\lambda^{\bullet})$ be the full subcategory of $\mD_{\nu}^{> \mu_{J}}(\lambda^{\bullet})/\mN_{\nu}$, which consists of objects isomorphic to objects of $\mQ_{\nu}(\lambda^{\bullet})$. 
		\item [(3)] Let $\mL^{\mu_{J}}_{\nu}(\lambda^{\bullet})$ be the full subcategory of $\mD_{\nu}(\lambda^{\bullet})/\mN_{\nu}^{>\mu_{J}}$, which consists of objects isomorphic to objects of $\mL^{\geqslant \mu_{J}}_{\nu}(\lambda^{\bullet})$.
	\end{enumerate}
\end{definition}
\begin{corollary}\label{keycog}
	For any $j\in J,r\in \mathbb{N}_{>0}$ and $\nu=rj+\nu''\in\mathbb{N}I$, the functors $\mathcal{E}^{(r)}_{j}$ and $\mathcal{F}^{(r)}_{j}$ restrict to the following functors between localizations
	$$\mathcal{E}^{(r)}_{j}:\mL_{\nu}^{\geqslant \mu_{J}}(\lambda^{\bullet}) \rightarrow \mL_{\nu''}^{\geqslant \mu_{J}}(\lambda^{\bullet}),$$
	$$\mathcal{F}^{(r)}_{j}:\mL_{\nu''}^{\geqslant \mu_{J}}(\lambda^{\bullet}) \rightarrow \mL_{\nu}^{\geqslant \mu_{J}}(\lambda^{\bullet}),$$
	$$\mathcal{E}^{(r)}_{j}:\mL_{\nu}^{> \mu_{J}}(\lambda^{\bullet}) \rightarrow \mL_{\nu''}^{> \mu_{J}}(\lambda^{\bullet}),$$
	$$\mathcal{F}^{(r)}_{j}:\mL_{\nu''}^{> \mu_{J}}(\lambda^{\bullet}) \rightarrow \mL_{\nu}^{> \mu_{J}}(\lambda^{\bullet}),$$
		$$\mathcal{E}^{(r)}_{j}:\mL^{\mu_{J}}_{\nu}(\lambda^{\bullet}) \rightarrow \mL^{\mu_{J}}_{\nu''}(\lambda^{\bullet}),$$
	$$\mathcal{F}^{(r)}_{j}:\mL^{\mu_{J}}_{\nu''}(\lambda^{\bullet}) \rightarrow \mL^{\mu_{J}}_{\nu}(\lambda^{\bullet}).$$
\end{corollary} 

\subsubsection{The structure of the Grothendieck groups}

With the action of functors $\mathcal{E}^{(r)}_{j}$, $\mathcal{F}^{(r)}_{j}$ and $\mathcal{K}^{\pm}_{j}$, $j \in J, r \in \mathbb{N}_{>0}$, the Grothendieck groups
$$\mK^{\mu_{J}}(\lambda^{\bullet}) =\bigoplus_{\nu \in \mathbb{N}I}\mK(\mL^{\mu_{J}}_{\nu}(\lambda^{\bullet})),$$
$$\mK^{>\mu_{J}}(\lambda^{\bullet}) =\bigoplus_{\nu \in \mathbb{N}I}\mK(\mL^{>\mu_{J}}_{\nu}(\lambda^{\bullet})),$$
$$\mK^{\geqslant \mu_{J}}(\lambda^{\bullet}) =\bigoplus_{\nu \in \mathbb{N}I}\mK(\mL^{\geqslant\mu_{J}}_{\nu}(\lambda^{\bullet}))$$
become  integrable $_{\mathcal{A}}\mathbf{U}(\mathfrak{l})$-modules. Then the localization functor and inclusion functor induce  morphisms of $_{\mathcal{A}}\mathbf{U}(\mathfrak{l})$-modules
\begin{equation}\label{igg}
	\mK^{>\mu_{J}}(\lambda^{\bullet}) \rightarrow \mK(\lambda^{\bullet}),
\end{equation}
\begin{equation}\label{igeg}
	\mK^{\geqslant \mu_{J}}(\lambda^{\bullet}) \rightarrow \mK(\lambda^{\bullet}),
\end{equation}
and surjective morphism
\begin{equation}\label{qg}
	\mK^{\geqslant \mu_{J}}(\lambda^{\bullet}) \rightarrow \mK^{\mu}(\lambda^{\bullet}) ,
\end{equation}
and a short exact sequence of  $_{\mathcal{A}}\mathbf{U}(\mathfrak{l})$-modules
\begin{equation}\label{exactg} 
	0 \rightarrow \mK^{> \mu_{J}}(\lambda^{\bullet}) \rightarrow \mK^{\geqslant  \mu_{J}}(\lambda^{\bullet}) \rightarrow \mK^{ \mu_{J}}(\lambda^{\bullet}) \rightarrow 0.
\end{equation}

Now we can state our second main result. It is an analogy of Theorem \ref{main1} and we will prove this theorem in the rest of this section.  
\begin{theorem}\label{main2}
	Let	$M$ be the restricted $\mathbf{U}(\mathfrak{l})$-module  $\mathbf{res}^{\mathfrak{g}}_{\mathfrak{l}}L(\lambda^{1})\otimes \cdots\otimes L(\lambda^{N})$ and $B$ be its canonical basis.
	Given any dominant weight $\mu_{J}$ of $\mathfrak{l}$, we have the following statements:
	
	(1) Under the composition of $\chi^{\lambda^{\bullet}}$ in Theorem \ref{tensor}(2) and the morphism defined by formula (\ref{igg}), the Grothendieck group $\mK^{>\mu_{J}}(\lambda^{\bullet})$ is canonically  isomorphic to the module $_{\mathcal{A}} M[> \mu_{J}]$, and there is a commutative diagram
	\[
	\xymatrix{
		\mK^{>\mu_{J}}(\lambda^{\bullet}) \ar[r]^{(\ref{igg})} \ar[d]_{\chi^{\lambda^{\bullet}}} &  \mK(\lambda^{\bullet}) \ar[d]^{\chi^{\lambda^{\bullet}}} \\
		_{\mathcal{A}}M[> \mu_{J}] \ar[r] &  _{\mathcal{A}}M.
	}
	\]
	Moreover, the images of nonzero simple perverse sheaves in $\mL^{>\mu_{J}}_{\nu}(\lambda^{\bullet}),\nu \in \mathbb{N}I$ provide an $\mathcal{A}$-basis, which coincides with $B \cap M[> \mu_{J}]$.
	
	(2) Under the composition of $\chi^{\lambda^{\bullet}}$ and the morphism defined by formula (\ref{igeg}), the Grothendieck group $\mK^{\geqslant \mu_{J}}(\lambda^{\bullet})$ is canonically  isomorphic to the module $_{\mathcal{A}} M[\geqslant  \mu_{J}]$, and  there is a commutative diagram
	\[
	\xymatrix{
		\mK^{\geqslant\mu_{J}}(\lambda^{\bullet}) \ar[r]^{(\ref{igeg})} \ar[d]_{\chi^{\lambda^{\bullet}}} &  \mK(\lambda^{\bullet}) \ar[d]^{\chi^{\lambda^{\bullet}}} \\
		_{\mathcal{A}}M[\geqslant \mu_{J}] \ar[r] &  _{\mathcal{A}}M.
	}
	\]
	
	Moreover, the images of nonzero simple perverse sheaves in $\mL^{\geqslant \mu_{J}}_{\nu}(\lambda^{\bullet}),\nu \in \mathbb{N}I$ provide an $\mathcal{A}$-basis, which coincides with $B \cap M[\geqslant \mu_{J}]$.
	
	(3)  The canonical isomorphism $\chi^{\lambda^{\bullet}}$ induced a canonical isomorphism  $$\mK^{ \mu_{J}}(\lambda^{\bullet}) \rightarrow  _{\mathcal{A}} (M[\geqslant  \mu_{J}]/M[>\mu_{J}]),$$ and there is a commutative diagram
	\[
	\xymatrix{
		\mK^{\geqslant\mu_{J}}(\lambda^{\bullet}) \ar[r]^{(\ref{qg})} \ar[d]_{\chi^{\lambda^{\bullet}}} &  \mK^{\mu_{J}}(\lambda^{\bullet}) \ar[d]^{\chi^{\lambda^{\bullet}}} \\
		_{\mathcal{A}}M[\geqslant \mu_{J}] \ar[r]^-{\pi_{\geqslant \mu_{J}}} &  _{\mathcal{A}}(M[\geqslant \mu_{J}]/M[> \mu_{J}]).
	}
	\]
	
	Moreover, the images of nonzero simple perverse sheaves in $\mL^{ \mu_{J}}_{\nu}(\lambda^{\bullet}),\nu \in \mathbb{N}I$ provide an $\mathcal{A}$-basis, which coincides with $\pi_{\geqslant \mu_{J}}(B \cap M[\geqslant \mu_{J}] )$.
\end{theorem}

By a similar argument as Lemma \ref{indlemma}-\ref{subquotient}, we have the following lemma.
\begin{lemma}\label{indlemmag}
	For any $0 \leqslant t \leqslant \nu_{j}$ and $L\in \mP_{\nu}(\lambda^{\bullet})$ with $t_{j}(L)=t$, let $\nu''=\nu -tj$, we assume $ \pi_{j,t} (L)=K$.  Then $L \in \mP^{\geqslant \mu_{J}}_{\nu}(\lambda^{\bullet})$ if and only if $ K \in \mP^{\geqslant \mu_{J}}_{\nu''}(\lambda^{\bullet})$, and $L \in \mP^{> \mu_{J}}_{\nu}(\lambda^{\bullet})$ if and only if $ K \in \mP^{> \mu_{J}}_{\nu''}(\lambda^{\bullet})$. Moreover, we have the following statements:
	\begin{enumerate}
		\item [(1)] If $wt_{J}(\nu) > \mu_{J}$, then
		$$\mD_{\nu}^{\geqslant \mu_{J}}(\lambda^{\bullet})=\mD_{\nu}^{> \mu_{J}}(\lambda^{\bullet})=\mD_{\nu}(\lambda^{\bullet}), $$
		$$\mP_{\nu}^{\geqslant \mu_{J}}(\lambda^{\bullet})=\mP_{\nu}^{> \mu_{J}}(\lambda^{\bullet})=\mP_{\nu}(\lambda^{\bullet}). $$
		
		\item[(2)] 	If  $wt_{J}(\nu)$ and $\mu_{J}$ can not be compared, then 
		$$\mD_{\nu}^{\geqslant \mu_{J}}(\lambda^{\bullet})=\mD_{\nu}^{> \mu_{J}}(\lambda^{\bullet}), $$
		$$\mP_{\nu}^{\geqslant \mu_{J}}(\lambda^{\bullet})=\mP_{\nu}^{> \mu_{J}}(\lambda^{\bullet}). $$
		\item[(3)] If $wt_{J}(\nu) < \mu_{J}$ and $L \in \mP_{\nu}^{\geqslant \mu_{J}}(\lambda^{\bullet})$, then  exists $j \in J$ such thta $t_{j}(L)>0$.
		\item [(4)] If $wt_{J}(\nu) = \mu_{J}$ and $L \in \mP_{\nu}^{\geqslant \mu_{J}}(\lambda^{\bullet})$,  then $L \notin \mP_{\nu}^{> \mu}(\lambda^{\bullet})$ if and only if $t_{j}(L)=0$ for any $j \in J$.
	\end{enumerate}
	In particular, the $_{\mathcal{A}}\mathbf{U}(\mathfrak{l})$ module $\mK^{ \mu_{J}}(\lambda^{\bullet})$ is isomorphic to a (maybe infinite) direct sum  of some copies of the irreducible highest weight module $L(\mu_{J})$.
\end{lemma}

Now we can prove our theorem.

\begin{proof}[Proof of Theorem \ref{main2}]

Let $pr_{I \backslash J}: \mathbb{N}I \rightarrow  \mathbb{N}(I \backslash J)$ be the projection map, and assume $\nu^{\backslash} \in \mathbb{N}(I \backslash J)$ is a dimension vector vanishing on $J$, 
we consider the following $\mathcal{A}$-submodules
$$\mK^{\mu_{J}}(\lambda^{\bullet})_{\nu^{\backslash} } =\bigoplus_{\nu \in \mathbb{N}I,pr_{I \backslash J}=\nu^{\backslash} }\mK(\mL^{\mu_{J}}_{\nu}(\lambda^{\bullet})),$$
$$\mK^{>\mu_{J}}(\lambda^{\bullet})_{\nu^{\backslash} } =\bigoplus_{\nu \in \mathbb{N}I,pr_{I \backslash J}=\nu^{\backslash} }\mK(\mL^{>\mu_{J}}_{\nu}(\lambda^{\bullet})),$$
$$\mK^{\geqslant \mu_{J}}(\lambda^{\bullet})_{\nu^{\backslash} } =\bigoplus_{\nu \in \mathbb{N}I,pr_{I \backslash J}=\nu^{\backslash} }\mK(\mL^{\geqslant\mu_{J}}_{\nu}(\lambda^{\bullet})).$$

Since $\mathcal{E}^{(r)}_{j}$, $\mathcal{F}^{(r)}_{j}$ and $\mathcal{K}^{\pm}_{j}$, $j \in J, r \in \mathbb{N}_{>0}$ don't change $\nu_{i}, i \in (I \backslash J)$, we can see that $\mK^{\mu_{J}}(\lambda^{\bullet})_{\nu^{\backslash} }$, $\mK^{>\mu_{J}}(\lambda^{\bullet})_{\nu^{\backslash} }$ and $\mK^{\geqslant \mu_{J}}(\lambda^{\bullet})_{\nu^{\backslash} }$ are $_{\mathcal{A}}\mathbf{U}(\mathfrak{l})$-modules, and we have the following decompositions of $_{\mathcal{A}}\mathbf{U}(\mathfrak{l})$-modules
$$ \mK(\lambda^{\bullet})=\bigoplus_{\nu^{\backslash} \in \mathbb{N}(I \backslash J)} \mK(\lambda^{\bullet})_{\nu^{\backslash} }, $$
$$ \mK^{\mu_{J}}(\lambda^{\bullet})=\bigoplus_{\nu^{\backslash} \in \mathbb{N}(I \backslash J)} \mK^{\mu_{J}}(\lambda^{\bullet})_{\nu^{\backslash} }, $$
$$ \mK^{>\mu_{J}}(\lambda^{\bullet})=\bigoplus_{\nu^{\backslash} \in \mathbb{N}(I \backslash J)} \mK^{>\mu_{J}}(\lambda^{\bullet})_{\nu^{\backslash} }, $$
$$ \mK^{\geqslant \mu_{J}}(\lambda^{\bullet})=\bigoplus_{\nu^{\backslash} \in \mathbb{N}(I \backslash J)} \mK^{\geqslant \mu_{J}}(\lambda^{\bullet})_{\nu^{\backslash} }. $$

We can naturally regard $\nu^{\backslash}$ as a dimension vector in $\mathbb{N}I$, then $wt_{J}(\nu^{\backslash})$ is defined. Notice that for any $\nu^{\backslash}$, the $\mathfrak{l}$-weights of $\mK(\lambda^{\bullet})_{\nu^{\backslash} }$ are always $\leqslant wt_{J}(\nu^{\backslash})$. So do the $\mathfrak{l}$-weights of $\mK^{\mu_{J}}(\lambda^{\bullet})_{\nu^{\backslash} }$, $\mK^{>\mu_{J}}(\lambda^{\bullet})_{\nu^{\backslash} }$  and $\mK^{\geqslant \mu_{J}}(\lambda^{\bullet})_{\nu^{\backslash} }$.

From now we fix a dimension vector $\nu^{\backslash}$.

(1) When $\nu= \nu^{\backslash} \in \mathbb{N}I$ and $\mu_{J}=wt_{J}(\nu^{\backslash})$, we can see that  $\mD_{\nu}^{> \mu_{J}}(\lambda^{\bullet})$ is empty, hence $\mK^{\geqslant \mu_{J}}(\lambda^{\bullet})_{\nu^{\backslash} }= \mK^{ \mu_{J}}(\lambda^{\bullet})_{\nu^{\backslash} }$ is a  direct sum  of some copies of the irreducible highest weight module $L(\mu_{J})$.

(2)Restrict the short exact sequence (\ref{exactg}) to its direct summands, we have the following short exact sequence
$$0 \rightarrow \mK^{> \mu_{J}}(\lambda^{\bullet})_{\nu^{\backslash} } \rightarrow \mK^{\geqslant  \mu_{J}}(\lambda^{\bullet})_{\nu^{\backslash} } \rightarrow \mK^{ \mu_{J}}(\lambda^{\bullet})_{\nu^{\backslash} } \rightarrow 0.$$
By decreasing induction on $\mu_{J}$ as what we have done in the proof of Theorem \ref{main1}, we can see that for any $\mu_{J} \leqslant wt_{J}(\nu^{\backslash})$, the $_{\mathcal{A}}\mathbf{U}(\mathfrak{l})$-module $\mK^{\geqslant \mu_{J}}(\lambda^{\bullet})_{\nu^{\backslash} }$ is a  direct sum  of irreducible highest weight modules with the highest weights greater or equal to $ \mu_{J}$ and less or equal to $wt_{J}(\nu^{\backslash})$. 

In particular, the $_{\mathcal{A}}\mathbf{U}(\mathfrak{l})$-module $\mK^{\geqslant \mu_{J}}(\lambda^{\bullet})_{\nu^{\backslash} }$ is  isomorphic to a  direct sum  of irreducible highest weight modules with the highest weights greater or equal to $ \mu_{J}$. We denote its image of $\chi^{\lambda^{\bullet}}$ by $M^{\geqslant \mu_{J}}_{\nu^{\backslash} },$
and also denote the image of $\mK(\lambda^{\bullet})_{\nu^{\backslash} }$ under $\chi^{\lambda^{\bullet}}$ by $M_{\nu^{\backslash} },$
then $$\bigoplus_{\nu^{\backslash} \in \mathbb{N}(I \backslash J)} M^{\geqslant \mu_{J}}_{\nu^{\backslash} } \subseteq {_{\mathcal{A}}M}[\geqslant \mu_{J}], $$
$$\bigoplus_{\nu^{\backslash} \in \mathbb{N}(I \backslash J)} M_{\nu^{\backslash} } = {_{\mathcal{A}}M}. $$
Let ${_{\mathcal{A}}M}[\geqslant \mu_{J}]_{\nu^{\backslash}}$ be $M_{\nu^{\backslash}} \cap {_{\mathcal{A}}M}[\geqslant \mu_{J}]$, we also have 
$$ \bigoplus_{\nu^{\backslash} \in \mathbb{N}(I \backslash J)}  {_{\mathcal{A}}M}[\geqslant \mu_{J}]_{\nu^{\backslash}} =  {_{\mathcal{A}}M}[\geqslant \mu_{J}].$$

For any $\nu$ such that $wt_{J}(\nu) \geqslant 0$,  $\mK(\mathcal{L}^{\geqslant 0}_{\nu}(\lambda^{\bullet})) \cong \mK(\mathcal{L}_{\nu}(\lambda^{\bullet}))$, it implies that  $M_{\nu^{\backslash}}$ and $M^{\geqslant 0}_{\nu^{\backslash} }$ have the same $\mu'_{J}$-weight space for $\mu'_{J} \geqslant 0$. In particular, $M_{\nu^{\backslash}}$ and $M^{\geqslant 0}_{\nu^{\backslash} }$ are isomorphic to each other any $\nu^{\backslash}$. From (1) and (2), after a similar argument as Theorem \ref{main1}, we can prove that each $\mK^{\geqslant \mu_{J}}(\lambda^{\bullet})_{\nu^{\backslash} }$ is canonically isomorphic to $_{\mathcal{A}}M^{\geqslant \mu_{J}}_{\nu^{\backslash} }$, hence $\mK^{\geqslant \mu_{J}}(\lambda^{\bullet})= \bigoplus_{\nu^{\backslash} \in \mathbb{N}(I \backslash J)}\mK^{\geqslant \mu_{J}}(\lambda^{\bullet})_{\nu^{\backslash}} $ is canonically isomorphic to $ \bigoplus_{\nu^{\backslash} \in \mathbb{N}(I \backslash J)}{_{\mathcal{A}}M}[\geqslant \mu_{J}]_{\nu^{\backslash}}={_{\mathcal{A}}M}[\geqslant \mu_{J}]$. The other statements can be easily proved. 
\end{proof}

When $ \nu \in wt_{J}^{-1}(\mu_{J})$, we denote the set  $ \mP_{\nu}^{\geqslant \mu_{J}}(\lambda^{\bullet})\backslash \mP_{\nu}^{> \mu_{J}}(\lambda^{\bullet})$ by $\mP_{\nu}^{\mu_{J},hi}$. Then we have the following corollary. 
\begin{corollary}
	Denote the set $\bigcup_{\nu \in wt^{-1}_{J}(\mu_{J})} \mP_{\nu}^{\mu_{J},hi} \backslash (\mP_{\nu}^{\mu_{J},hi} \cap \mN_{\nu})$ by $S_{1}(\mu_{J})$, then we have the following restriction rule $$\mathbf{res}^{\mathfrak{g}}_{\mathfrak{l}} L(\lambda^{N})\otimes L(\lambda^{N-1})\otimes \cdots\otimes L(\lambda^{1}) \cong \bigoplus_{\mu_{J}}\bigoplus_{\alpha \in S_{1}(\mu_{J})}L(\mu_{J})^{\alpha},$$ 
	where each $L(\mu_{J})^{\alpha}$ is a copy of irreducible highest weight $\mathbf{U}(\mathfrak{l})$-module with the highest weight $\mu_{J}$.
\end{corollary}

\section{Decomposition rule and restriction rule}
In \cite{fang2025lusztigsheavescharacteristiccycles}, the authors use characteristic cycles to build isomorphisms of $\mathfrak{g}$-modules from Lusztig's perverse sheaves to BM homology groups of Nakajima's quiver variety.  In this section, we use this result to construct certain subquotients of $\mathfrak{g}$-modules arising from quiver varieties. Moreover, we prove the decomposition and restriction numbers, if they are finite, are equal to the dimensions of  BM homology groups of certain locally closed subsets of quiver varieties.

\subsection{quiver variety and tensor product variety}
In this section, we recall some results about the quiver variety and the tensor product variety in  \cite{MR1302318}, \cite{MR1604167} , \cite{MR1865400} and \cite{MR3077693}.
\subsubsection{Nakajima's quiver variety}
Consider the framed graph $(I^{(1)},H^{(1)})$ of the graph $(I,H)$ and the framed quiver $Q^{(1)}=(I^{(1)},H^{(1)},\Omega^{(1)})$ of the quiver $Q=(I,H,\Omega)$. For any $\nu\in \mathbb{N}I$ and $\omega\in \mathbb{N}I^1$, we fix an $I$-graded $\mathbb{C}$-vector space $\bV$ and an $I^1$-graded $\mathbb{C}$-vector space $\mathbf{W}$ such that their dimension vectors are $\nu$ and $\omega$, then $\bV\oplus \mathbf{W}$ is an $I^{(1)}$-graded $\mathbb{C}$-vector space such that its dimension vector is $\nu+\omega$. There is an affine space 
$$\mathbf{E}_{\bV,\mathbf{W},H^{(1)}}=\bigoplus_{h \in H}\mathbf{Hom}(\mathbf{V}_{h'},\mathbf{V}_{h''}) \oplus 
\bigoplus_{i \in I N}\mathbf{Hom}(\mathbf{V}_{i},\mathbf{W}_{i})  \oplus \bigoplus_{i \in I}\mathbf{Hom}(\mathbf{W}_{i},\mathbf{V}_{i}).  $$
which has a  $G_\bV$-action on it.

We fix a function $\varepsilon:H^{(1)} \rightarrow\mathbb{C}^*$ such that $\varepsilon(h)+\varepsilon(\overline{h})=0$ for any $h\in H$, and $\varepsilon(i\rightarrow i^1)=-1,\varepsilon(i^1\rightarrow i)=1$ for any $i\in I$. Then the corresponding moment map attached to be $G_\bV$-action on the symplectic vector space $\mathbf{E}_{\bV,\mathbf{W},H^{(1)}}$ is given by 
\begin{align*}
	\mu_{\bV\oplus \mathbf{W}}:\mathbf{E}_{\bV,\mathbf{W},H^{(1)}} &\rightarrow \bigoplus_{i\in I}\mathrm{End}(\bV_i)\\
	(x,\bar{x},y,z)&\mapsto(\sum_{h\in H, h''=i}\varepsilon(h)x_h x_{\overline{h}}+z_iy_i)_{i\in I}  ,
\end{align*}
where $	(x,\bar{x},y,z)$ is the notations in Section \ref{section thick}.
Let $\Lambda_{\bV\oplus \mathbf{W}}$ be the subset of $\mathbf{E}_{\bV,\mathbf{W},H^{(1)}}$, which consists of nilpotent elements in $\mu_{\bV\oplus \mathbf{W}}^{-1}(0). $ We also define the subset
\begin{align*}
	\Lambda_{\bV,\mathbf{W}}=\Lambda_\bV\times \bigoplus_{i \in I} \Hom(\bV_{i},\mathbf{W}_{i}) \subseteq \Lambda_{\bV\oplus \mathbf{W}}.
\end{align*}
where $\Lambda_{\bV}$ is Lusztig's nilpotent variety in Section 4.3.

An element $(x,\bar{x},y,z)$ of $\mathbf{E}_{\bV,\mathbf{W},H^{(1)}}$ is said to be stable, if the zero space is the unique $I$-graded subspace $\bV'$ of $\bV$ such that $(x,\bar{x})(\bV')\subset \bV'$ and $y_i(\bV'_i)=0$ for any $h\in H,i\in I$, see \cite[Lemma 3.8]{MR1604167}. We define $\Lambda_{\bV,\mathbf{W}}^s,\Lambda_{\bV\oplus \mathbf{W}}^{s}$ to be the subset of $\Lambda_{\bV,\mathbf{W}}$ and $\Lambda_{\bV\oplus \mathbf{W}}$ consisting of stable elements respectively. 

By \cite[Lemma 3.10]{MR1604167}, the group $G_{\bV}$ acts freely on $\mu_{\bV\oplus \mathbf{W}}^{-1}(0)^s$ and $\Lambda_{\bV,\mathbf{W}}^s$. Nakajima's quiver varieties 
$\mathfrak{M}(\nu,\omega)$ and  $\mathfrak{L}(\nu,\omega)$ 
are defined to be the geometric quotients of $\mu_{\bV\oplus \mathbf{W}}^{-1}(0)^s$ and $\Lambda_{\bV,\mathbf{W}}^s$ by $G_{\bV}$ respectively. We denote by $[x,\bar{x},y,z]$ the $G_\bV$-orbit of $(x,\bar{x},y,z)$, considered as an element in the geometric quotients.

\subsubsection{Nakajima's tensor product variety}
Let $\mathbf{W}=\mathbf{W}^{1} \oplus \mathbf{W}^{2}$ be a fixed $I^1$-graded $\mathbb{C}$-vector space decomposition such that $\mathbf{W}^1,\mathbf{W}^{2}$ have dimension vectors $\omega^1$ and $\omega^2$ respectively. There is an one-parameter subgroup 
\begin{align*}
	\rho:\mathbb{G}_m&\rightarrow G_{\mathbf{W}^1}\times G_{\mathbf{W}^2}\subset G_{\mathbf{W}}\\
	t&\mapsto ({\mathrm{Id}}_{\mathbf{W}^{1}},t\, {\mathrm{Id}}_{\mathbf{W}^{2}})
\end{align*}
acting on $\mathfrak{M}(\nu,\omega)$. By \cite[Lemma 3.2]{MR1865400}, the $\rho(\mathbb{G}_m)$-fixed point set $\mathfrak{M}(\nu,\omega)^{\rho(\mathbb{G}_m)}$ is isomorphic to 
$$\bigsqcup_{\nu'+\nu''=\nu} \mathfrak{M}(\nu',\omega^{1}) \times \mathfrak{M}(\nu'',\omega^{2}).$$

Nakajima's tensor product variety is defined to be
$$\tilde{\mathfrak{Z}}(\nu,\omega)=\{ [x,y,z] \in \mathfrak{M}(\nu,\omega)\mid \lim_{t \rightarrow  0} \rho(t). [x,y,z]  \in \bigsqcup_{\nu'+\nu''=\nu} \mathfrak{L} (\nu',\omega^{1}) \times \mathfrak{L}(\nu'',\omega^{2}) \}.$$

Following \cite{Malkin2003Tensor}, the tensor product variety $\tilde{\mathfrak{Z}}$ is a geometric quotient of $\Pi_{\bV,\mathbf{W}^1\oplus \mathbf{W}^2}^{s}$. The variety $\Pi_{\bV,\mathbf{W}^1\oplus \mathbf{W}^2}^{s}$ consists of stable points in $\Pi_{\bV,\mathbf{W}^1\oplus \mathbf{W}^2}$, where $\Pi_{\bV,\mathbf{W}^1\oplus \mathbf{W}^2}$ is the subvariety of $\mathbf{E}_{\bV,\mathbf{W},H^{(1)}}$ consisting of those nilpotent elements $(x,\bar{x},y,z)$ such that $z(\mathbf{W}^{2})=0$ and $\mathbf{S}_{1}(x,\bar{x},y,z) \subset \mathbf{S}_{2}(x,\bar{x},y,z)$. Here $\mathbf{S}_{1}(x,\bar{x},y,z)$ is the smallest $(x,\bar{x})$-stable subspace of $\mathbf{V}$ containing $z(\mathbf{W})$, and $\mathbf{S}_{2}(x,\bar{x},y,z)$ is the largest $(x,\bar{x})$-stable subspace of $\mathbf{V}$ contained in $y^{-1}(\mathbf{W}^2)$. 

\subsubsection{Hecke correspondence and $\mathfrak{g}$-module structure}

For any $i\in I$ and $\nu=\nu'+i\in \mathbb{N}I$, we assume that the fixed $I$-graded $\mathbb{C}$-vector space $\bV,\bV'$ have dimension vector $\nu,\nu'$ respectively and $\bV'$ is a subspace of $\bV$. The Hecke correspondence $\mathfrak{P}_{i}(\nu,\omega)$ is defined to be the subvariety of $\mathfrak{M}(\nu',\omega) \times \mathfrak{M}(\nu,\omega)$ consisting of $([x',\bar{x}',y',z'],[x,\bar{x},y,z])$ such that there exists $\xi=(\xi_{j})_{j \in I} \in \bigoplus_{j \in I}\Hom(\bV'_{j},\bV_{j})$ such that 
$$\xi (x',\bar{x'})=(x,\bar{x})\xi,y\xi=y',\xi z'=z.$$

For a fixed orientation $\Omega\subset H$, let $\mathbf{A}_{\Omega}$ be the adjacency matrix of the quiver $Q=(I,H,\Omega)$, that is, it is a $I\times I$-matrix whose $(i,j)$ entry is the number of arrows $h\in \Omega$ such that $h'=i,h''=j$. Let $\mathbf{C}_{\Omega}=\mathrm{Id}_{I\times I}-\mathbf{A}_{\Omega}$. In \cite{MR1604167} and \cite{MR1865400}, Nakajima defined the following operators
\begin{align*}
	&e_i=(-1)^{\langle i , \mathbf{C}_{\bar{\Omega}}\nu \rangle} [\mathfrak{P}_{i}(\nu,\omega)]:\mathbf{H}^{\mathrm{BM}}_{\mathrm{top}}(\mathfrak{L}(\nu,\omega) ,\mathbb{Q}) \rightarrow \mathbf{H}^{\mathrm{BM}}_{\mathrm{top}}(\mathfrak{L}(\nu',\omega) ,\mathbb{Q}),\\
	&e_i=(-1)^{\langle i , \mathbf{C}_{\bar{\Omega}}\nu \rangle} [\mathfrak{P}_{i}(\nu,\omega)]:\mathbf{H}^{\mathrm{BM}}_{\mathrm{top}}(\tilde{\mathfrak{Z}}(\nu,\omega) ,\mathbb{Q}) \rightarrow \mathbf{H}^{\mathrm{BM}}_{\mathrm{top}}(\tilde{\mathfrak{Z}}(\nu',\omega) ,\mathbb{Q}),\\
	&f_{i}=(-1)^{\langle i , \omega- \mathbf{C}_{\Omega}\nu \rangle} [\mathrm{sw}(\mathfrak{P}_{i}(\nu,\omega))]:\mathbf{H}^{\mathrm{BM}}_{\mathrm{top}}(\mathfrak{L}(\nu',\omega) ,\mathbb{Q}) \rightarrow \mathbf{H}^{\mathrm{BM}}_{\mathrm{top}}(\mathfrak{L}(\nu,\omega) ,\mathbb{Q}),\\
	&f_{i}=(-1)^{\langle i , \omega- \mathbf{C}_{\Omega}\nu \rangle} [\mathrm{sw}(\mathfrak{P}_{i}(\nu,\omega))]:\mathbf{H}^{\mathrm{BM}}_{\mathrm{top}}(\tilde{\mathfrak{Z}}(\nu',\omega) ,\mathbb{Q}) \rightarrow \mathbf{H}^{\mathrm{BM}}_{\mathrm{top}}(\tilde{\mathfrak{Z}}(\nu,\omega) ,\mathbb{Q}),
\end{align*}
where $\mathrm{sw}: \mathfrak{M}(\nu',\omega) \times \mathfrak{M}(\nu,\omega) \rightarrow \mathfrak{M}(\nu,\omega) \times \mathfrak{M}(\nu',\omega) $ is the swapping map, and $\mathbf{H}^{\mathrm{BM}}_{\mathrm{top}}(X,\mathbf{k})$ is the top degree of the Borel-Moore homology of a variety $X$ with coefficients $\mathbf{k}$. 

The fundamental class $[\Delta(\nu,\omega)]$ of the diagonal defines  operators
\begin{align*}
	&[\Delta(\nu,\omega)]:\mathbf{H}^{\mathrm{BM}}_{\mathrm{top}}(\mathfrak{L}(\nu,\omega) ,\mathbb{Q}) \rightarrow \mathbf{H}^{\mathrm{BM}}_{\mathrm{top}}(\mathfrak{L}(\nu,\omega) ,\mathbb{Q}),\\
	&[\Delta(\nu,\omega)]:\mathbf{H}^{\mathrm{BM}}_{\mathrm{top}}(\tilde{\mathfrak{Z}}(\nu,\omega) ,\mathbb{Q}) \rightarrow \mathbf{H}^{\mathrm{BM}}_{\mathrm{top}}(\tilde{\mathfrak{Z}}(\nu,\omega) ,\mathbb{Q}).
\end{align*}

\begin{theorem}[{\cite[Thoerem 10.2]{MR1604167}},{\cite[Theorem 5.2]{MR1865400}}]
	With the operators $e_{i},f_{i},i \in I$ and $[\Delta(\nu,\omega)]$, the direct sum of Borel-Moore homology group $$\mathbf{H}(\mathfrak{L}(\omega))= \bigoplus_{\nu \in \mathbb{N}I} \mathbf{H}^{\mathrm{BM}}_{\mathrm{top}}(\mathfrak{L}(\nu,\omega) ,\mathbb{Q}) $$ becomes an integrable $\mathbf{U}_{1}(\mathfrak{g})$-module, which is isomorphic to the highest weight vector of $L_{1}(\lambda)$.
	
	The direct sum of Borel-Moore homology group $$\mathbf{H}(\tilde{\mathfrak{Z}}(\omega))= \bigoplus_{\nu \in \mathbb{N}I} \mathbf{H}^{\mathrm{BM}}_{\mathrm{top}}(\tilde{\mathfrak{Z}}(\nu,\omega) ,\mathbb{Q}) $$ becomes an integrable $\mathbf{U}_{1}(\mathfrak{g})$-module, which is isomorphic to the tensor products $L_{1}(\lambda^{1}) \otimes L_{1}(\lambda^{2})$.
\end{theorem}

\subsection{decomposition and restriction rule of Kac-Moody algebras}
\subsubsection{Characteristic cycles for framed quivers}
Assume $N=1$ and denote $\mathcal{K}(\lambda^{\bullet})$ by $\mathcal{K}(\lambda)$. We also assume that $\lambda$ is the dominant weight associated to $\omega$.  By Theorem \ref{high}, when $v=-1$, the Grothendieck group  $\mathcal{K}(\lambda)|_{v=-1}$ is an irreducible $_{\mathbb{Z}}\mathbf{U}_{-1}(\mathfrak{g})$-module. Following \cite{fang2025lusztigsheavescharacteristiccycles}, we consider the sign twist $$\psi^{-}=\psi^{-}_{Q^{(1)}}=(-1)^{\langle \nu',\nu''\rangle_{Q^{(1)}}  },$$  $$\psi^{+}=\psi^{+}_{Q^{(1)}}=(-1)^{\langle \nu',\nu''\rangle_{\overline{Q^{(1)}}}}$$ for the framed quiver, here $\langle \nu',\nu''\rangle_{Q^{(1)}}$ is the Euler form of the framed quiver $Q^{(1)}$ and $\langle \nu',\nu''\rangle_{\overline{Q^{(1)}}}$ is the Euler form of the framed quiver with the opposite orientation.
Consider operators $[\mathcal{F}^{(r)}_{i}]^{\psi}$ and $[\mathcal{E}^{(r)}_{i}]^{\psi}$ defined by
\begin{equation}\label{twist}
	\begin{split}
		[\mathcal{F}^{(r)}_{i}]^{\psi}([L])= \psi^{-}(ri,deg([L]))[\mathcal{F}^{(r)}_{i}]([L]), \\
		[\mathcal{E}^{(r)}_{i}]^{\psi}([L])= \psi^{+}(ri,deg([L]))[\mathcal{E}^{(r)}_{i}]([L]).  
	\end{split}
\end{equation}
Then under the action of  $[\mathcal{F}^{(r)}_{i}]^{\psi}$ and $[\mathcal{E}^{(r)}_{i}]^{\psi}$, the Grothendieck group $\mathcal{K}(\lambda)|_{v=-1}$ becomes a $_{\mathbb{Z}}\mathbf{U}_{1}(\mathfrak{g})$-module, which is canonically isomorphic to $L_{1}(\lambda)$ by \cite[Proposition 6.11]{fang2025lusztigsheavescharacteristiccycles}. We denote this twisted module by $mK(\lambda)|_{v=-1}^{\psi}$.

 Since $SS(L) \subseteq \Lambda_{\bV,\mathbf{W}}$ for any Lusztig's sheaf $L \in \mP_{\nu}(\lambda)$, we can construct the characteristic cycle map  $\CC_{\bV}$ following  \cite{hennecart2024geometric},
 \begin{equation}
 	\begin{split}
 		\CC_{\bV} : \mK(\mQ_{\nu}(\lambda))|_{v=-1}^{\psi} \xrightarrow{incl}  & \mK_{0}(\mD^{b}_{G_{I} \times G_{\bV}}(\mathbf{E}_{\mathbf{V},\mathbf{W},\Omega^{(1)}}, \Lambda_{\bV,\mathbf{W}})) \xrightarrow{\mathbf{For}}\\
 		  &\mK_{0}(\mD^{b}_{c}(\mathbf{E}_{\mathbf{V},\mathbf{W},\Omega^{(1)}}, \Lambda_{\bV,\mathbf{W}}))  \xrightarrow{CC} \mathbf{H}^{\mathrm{BM}}_{\mathrm{top}}(\Lambda_{\bV,\mathbf{W}} ,\mathbb{Z}), 
 	\end{split}
 \end{equation}
 where $\mK_{0}(\mD)$ is the Grothendieck group of triangulated category $\mD$,  $\mD^{b}_{G_{I} \times G_{\bV}}(\mathbf{E}_{\mathbf{V},\mathbf{W},\Omega^{(1)}}, \Lambda_{\bV,\mathbf{W}})$ is the full subcategory of $\mD^{b}_{G_{I} \times G_{\bV}}(\mathbf{E}_{\mathbf{V},\mathbf{W},\Omega^{(1)}})$ consisting of objects which have singular supports contained in $\Lambda_{\bV,\mathbf{W}} $, and $CC$ is the usual characteristic cycle map. See details in \cite[Section 2.1]{hennecart2024geometric}.
 
 By \cite[Proposition 6.4]{fang2025lusztigsheavescharacteristiccycles}, a simple Lusztig sheaf $L$ belongs to $\mN_{\nu}$ if and only if its singular support is contained in the subset of unstable points. In particular, the characteristic cycle map  $\bigoplus_{\bV}\CC_{\bV}$ induces a well-defined morphism 
$$  \CC^{s,\omega}:\mK(\lambda)|^{\psi}_{v=-1} \longrightarrow \mathbf{H}^{\mathrm{BM}}_{\mathrm{top}}(\mathfrak{L}(\omega),\mathbb{Z}),$$
which is an isomorphism of $_{\mathbb{Z}}\mathbf{U}_{1}(\mathfrak{g})$-modules by \cite[Theorem 6.16]{fang2025lusztigsheavescharacteristiccycles}.

\subsubsection{Characteristic cycles for $2$-framed quivers}
Assume $N=2$ and denote $\mathcal{K}(\lambda^{\bullet})$ by $\mathcal{K}(\lambda^{1},\lambda^{2})$. We also assume that $\lambda^{1},\lambda^{2}$ are the dominant weights associated to $\omega^{1},\omega^{2}$ respectively. Consider  the following sign twist for $2$-framed quivers
 $$\psi^{-}=\psi^{-}_{Q^{(2)}}=(-1)^{\langle \nu',\nu''\rangle_{Q^{(2)}}  },$$ 
 $$\psi^{+}=\psi^{+}_{Q^{(2)}}=(-1)^{\langle \nu',\nu''\rangle_{\overline{Q^{(2)}}}},$$
 then  the equation (\ref{twist}) defines  actions of  $[\mathcal{F}^{(r)}_{i}]^{\psi}$ and $[\mathcal{E}^{(r)}_{i}]^{\psi}$ on $\mathcal{K}(\lambda^{1},\lambda^{2})|_{v=-1}$. With the operators $[\mathcal{F}^{(r)}_{i}]^{\psi}$ and $[\mathcal{E}^{(r)}_{i}]^{\psi}$, $\mathcal{K}(\lambda^{1},\lambda^{2})|_{v=-1}$ becomes a  $_{\mathbb{Z}}\mathbf{U}_{1}(\mathfrak{g})$-module and we denote this $\mathfrak{g}$-module by $\mathcal{K}(\lambda^{1},\lambda^{2})|^{\psi}_{v=-1}$.
 
 By \cite[Lemma 6.5 and Proposition 6.8]{fang2025lusztigsheavescharacteristiccycles},
 the characteristic cycle map
 \begin{equation}
 	\begin{split}
 		\CC_{\bV} : \mK(\mQ_{\nu}(\lambda^{1},\lambda^{2}))|_{v=-1}^{\psi} \xrightarrow{incl}  & \mK_{0}(\mD^{b}_{G_{I} \times G_{\bV}}(\mathbf{E}_{\mathbf{V},\mathbf{W}^{1}\oplus \mathbf{W}^{2},\Omega^{(2)}}, \Pi_{\bV,\mathbf{W}^{1}\oplus \mathbf{W}^{2}})) \xrightarrow{\mathbf{For}}\\
 		&\mK_{0}(\mD^{b}_{c}(\mathbf{E}_{\mathbf{V},\mathbf{W}^{1}\oplus \mathbf{W}^{2},\Omega^{(2)}}, \Pi_{\bV,\mathbf{W}^{1}\oplus \mathbf{W}^{2}}))  \xrightarrow{CC} \mathbf{H}^{\mathrm{BM}}_{\mathrm{top}}(\Pi_{\bV,\mathbf{W}^{1}\oplus \mathbf{W}^{2}}),\mathbb{Z})
 	\end{split}
 \end{equation}
induces a well-defined morphism
 $$  \CC^{s,\omega^{1},\omega^{2}}:\mK(\lambda^{1},\lambda^{2})|^{\psi}_{v=-1} \longrightarrow \mathbf{H}^{\mathrm{BM}}_{\mathrm{top}}(\tilde{\mathfrak{Z}},\mathbb{Z}),$$
 which is an isomorphism of $_{\mathbb{Z}}\mathbf{U}_{1}(\mathfrak{g})$-modules by \cite[Theorem 6.23]{fang2025lusztigsheavescharacteristiccycles}.

\subsubsection{The restriction rule} 

Now we assume $N=1$ and take a subset $J \subset I$.
\begin{definition}
 	Let $\mu_{J}$ be a dominant integral weight of $\mathfrak{l}=\mathfrak{l}_{J}$. 
 	Given any locally closed subset $X \subseteq \mathbf{E}_{\bV,\mathbf{W},H^{(1)}}$, we denote $X \cap \mathbf{E}^{\geqslant \mu_{J}}_{\bV,\mathbf{W},H^{(1)}}$ by $X ^{\geqslant \mu_{J}}$,  $X \cap \mathbf{E}^{> \mu_{J}}_{\bV,\mathbf{W},H^{(1)}}$ by $X ^{> \mu_{J}}$, and $X^{\mu_{J}}$ by $X ^{\geqslant \mu_{J}} \backslash X ^{> \mu_{J}}$.
\end{definition}

In particular, $\Lambda_{\bV,\mathbf{W}}^{s,\geqslant \mu_{J} }$ and $\Lambda_{\bV,\mathbf{W}}^{s,> \mu_{J} }$ are close subsets of $\Lambda_{\bV,\mathbf{W}}^{s}$ and $G_{\bV}$ acts freely on them. We denote their geometric quotients by $\mathfrak{L}^{\geqslant \mu_{J}}(\nu,\omega)$ and  $\mathfrak{L}^{> \mu_{J}}(\nu,\omega)$ respectively, then  $\mathfrak{L}^{> \mu_{J}}(\nu,\omega) \subseteq \mathfrak{L}^{\geqslant \mu_{J}}(\nu,\omega)$  are close subvarieties of  $\mathfrak{L}(\nu,\omega)$.  The close embeddings induce morphisms 
\begin{equation}\label{i1}
	\mathbf{H}^{\mathrm{BM}}_{\mathrm{top}}(\mathfrak{L}^{> \mu_{J}}(\nu,\omega) ,\mathbb{Z}) \rightarrow \mathbf{H}^{\mathrm{BM}}_{\mathrm{top}}(\mathfrak{L}^{\geqslant \mu_{J}}(\nu,\omega) ,\mathbb{Z}),
\end{equation}
\begin{equation}\label{i2}
	\mathbf{H}^{\mathrm{BM}}_{\mathrm{top}}(\mathfrak{L}^{\geqslant \mu_{J}}(\nu,\omega) ,\mathbb{Z}) \rightarrow \mathbf{H}^{\mathrm{BM}}_{\mathrm{top}}(\mathfrak{L}(\nu,\omega) ,\mathbb{Z}).
\end{equation}

Combine Theorem \ref{main2} and \cite[Theorem 6.16]{fang2025lusztigsheavescharacteristiccycles}, we have the following proposition.

\begin{proposition}\label{main3}
		With the operator $[\Delta(\nu,\omega)]$, the Hecke correspondences $e_{j},f_{j},j \in J$ and their divided powers, the direct sums of Borel-Moore homology groups $$\mathbf{H}(\mathfrak{L}^{> \mu_{J}}(\omega),\mathbb{Z})= \bigoplus_{\nu \in \mathbb{N}I} \mathbf{H}^{\mathrm{BM}}_{\mathrm{top}}(\mathfrak{L}^{> \mu_{J}}(\nu,\omega) ,\mathbb{Z}) $$
		$$\mathbf{H}(\mathfrak{L}^{\geqslant \mu_{J}}(\omega),\mathbb{Z})= \bigoplus_{\nu \in \mathbb{N}I} \mathbf{H}^{\mathrm{BM}}_{\mathrm{top}}(\mathfrak{L}^{\geqslant \mu_{J}}(\nu,\omega) ,\mathbb{Z}) $$ 
		become  integrable $_{\mathbb{Z}}\mathbf{U}_{1}(\mathfrak{l})$-modules, which are isomorphic to  the $\mathbb{Z}$-forms of $\mathbf{res}^{\mathfrak{g}}_{\mathfrak{l}}L_{1}(\lambda)[>\mu_{J}]$ and $\mathbf{res}^{\mathfrak{g}}_{\mathfrak{l}}L_{1}(\lambda)[\geqslant \mu_{J}]$ respectively.  Moreover, we have the following commutative diagram
		\[
		\xymatrix{
		\mK^{>\mu_{J}}(\lambda)|^{\psi}_{v=-1} \ar[d] \ar[r] & \mK^{\geqslant\mu_{J}}(\lambda)|^{\psi}_{v=-1}\ar[d] \ar[r] & \mK(\lambda)|^{\psi}_{v=-1} \ar[d] \\
		\mathbf{H}(\mathfrak{L}^{> \mu_{J}}(\omega),\mathbb{Z}) \ar[r]^{(\ref{i1})} & \mathbf{H}(\mathfrak{L}^{\geqslant \mu_{J}}(\omega),\mathbb{Z}) \ar[r]^{(\ref{i2})} & \mathbf{H}(\mathfrak{L}(\omega),\mathbb{Z}),
		}
		\]
		where all the vertical maps are $\CC^{s,\omega}$.

		The pullback of $\mathfrak{L}^{\mu_{J}}(\omega) \rightarrow \mathfrak{L}^{\geqslant \mu_{J}}(\omega)$ induces a surjective morphism of $\mathfrak{l}$-modules
		$$\pi'_{\geqslant \mu_{J}}:\mathbf{H}(\mathfrak{L}^{\geqslant \mu_{J}}(\omega),\mathbb{Z}) \mathbf{H}(\mathfrak{L}^{\mu_{J}}(\omega),\mathbb{Z}),$$ then the map $\CC^{s,\omega}$ also induces an isomorphism   from
		$$\mathbf{H}(\mathfrak{L}^{ \mu_{J}}(\omega),\mathbb{Z})= \bigoplus_{\nu \in \mathbb{N}I} \mathbf{H}^{\mathrm{BM}}_{\mathrm{top}}(\mathfrak{L}^{ \mu_{J}}(\nu,\omega) ,\mathbb{Z}) $$
		to the $\mathbb{Z}$-forms of $\mathbf{res}^{\mathfrak{g}}_{\mathfrak{l}}L_{1}(\lambda)[\geqslant \mu_{J}]/ \mathbf{res}^{\mathfrak{g}}_{\mathfrak{l}}L_{1}(\lambda)[> \mu_{J}]$, which fits into the following commutative diagram of short exact sequences
		\[
		\xymatrix{
			\mK^{>\mu_{J}}(\lambda)|^{\psi}_{v=-1} \ar[d] \ar[r] & \mK^{\geqslant\mu_{J}}(\lambda)|^{\psi}_{v=-1}\ar[d] \ar[r]^{\pi_{\geqslant \mu_{J}}} & \mK^{\mu_{J}}(\lambda)|^{\psi}_{v=-1} \ar[d] \\
			\mathbf{H}(\mathfrak{L}^{> \mu_{J}}(\omega),\mathbb{Z}) \ar[r] & \mathbf{H}(\mathfrak{L}^{\geqslant \mu_{J}}(\omega),\mathbb{Z}) \ar[r]^{\pi'_{\geqslant \mu_{J}}} & \mathbf{H}(\mathfrak{L}^{\mu_{J}}(\omega),\mathbb{Z}).
		}
		\]
\end{proposition}
\begin{proof}
	By the definition of $\mD^{\geqslant \mu_{J}}_{\nu}(\lambda)$ and $\mD^{> \mu_{J}}_{\nu}(\lambda) $, we can see that $\CC^{s,\omega}$ defines  morphisms $$ \mK^{>\mu_{J}}(\lambda)|^{\psi}_{v=-1} \rightarrow \mathbf{H}(\mathfrak{L}^{> \mu_{J}}(\omega),\mathbb{Z}), $$
	$$ \mK^{\geqslant \mu_{J}}(\lambda)|^{\psi}_{v=-1} \rightarrow \mathbf{H}(\mathfrak{L}^{\geqslant  \mu_{J}}(\omega),\mathbb{Z}). $$
	By \cite[Theorem 6.16]{fang2025lusztigsheavescharacteristiccycles}, these morphisms are injective and $_{\mathbb{Z}}\mathbf{U}(\mathfrak{l})$-linear. Hence $\mK^{\geqslant \mu_{J}}(\lambda)|^{\psi}_{v=-1}$ is isomorphic to a submodule of $\mathbf{res}^{\mathfrak{g}}_{\mathfrak{l}}L_{1}(\lambda)[\geqslant \mu_{J}]$ via the composition of $\chi^{\lambda}$ and $\CC^{s,\omega}$. However, for $\mu_{J}=0$, we can see that $\mK^{\geqslant \mu_{J}}(\lambda)|^{\psi}_{v=-1}$ is isomorphic to $\mathbf{res}^{\mathfrak{g}}_{\mathfrak{l}}L_{1}(\lambda)$. By a similar argument as Theorem \ref{main2}, the proposition can be proved.
\end{proof}

	For any irreducible component $Z_{0}$ of $\mathfrak{L}^{\mu_{J}}(\nu,\omega)$ with $wt_{J}(\nu)=\mu_{J}$,  the closure $Z=\bar{Z_{0}}$ of $Z_{0}$ is a irreducible component in Nakajima's quiver variety $\mathfrak{L}(\nu,\omega)$, we say such $Z$ is strictly dominated by $\mu_{J}$. We denote the set of such irreducible components by $S_{2}(\mu_{J})$.  For any irreducible component $Z \in S_{2}(\mu_{J}) $, there exists   a unique highest weight vector $v_{Z}$ in $\mathbf{H}(\mathfrak{L}^{\geqslant \mu_{J}}(\omega),\mathbb{Z}) \subseteq \mathbf{H}(\mathfrak{L}(\omega),\mathbb{Z})$ such that $v_{Z}-[Z] \in \mathbf{H}(\mathfrak{L}^{> \mu_{J}}(\omega),\mathbb{Z})$.  We still denote the image of $v_{Z}$ in $\mathbf{res}^{\mathfrak{g}}_{\mathfrak{l}}L_{1}(\lambda)$ by $v_{Z}$, then $L_{1}(\mu_{J})^{Z}=\mathbf{U}^{-}(\mathfrak{l}) v_{Z}$ is an irreducible submodule isomorphic to $L_{1}(\mu_{J})$. 
\begin{corollary}\label{resrule}
	With the notations above, we have the following decomposition of $\mathfrak{l}$-modules,
	$$\mathbf{res}^{\mathfrak{g}}_{\mathfrak{l}}L_{1}(\lambda)=\bigoplus_{\mu_{J}}\bigoplus_{Z\in S_{2}(\mu_{J})} L_{1}(\mu_{J})^{Z} .  $$ 
	In particular, if the restriction coefficient $m^{\lambda}_{\mu_{J}}$ is finite, we have 
	$$m^{\lambda}_{\mu_{J}}= \sum_{\nu \in wt^{-1}_{J}(\mu_{J})} \dim \mathbf{H}^{\mathrm{BM}}_{\mathrm{top}}(\mathfrak{L}^{ \mu_{J}}(\nu,\omega),\mathbb{Q}) .$$
\end{corollary}

\subsubsection{The decomposition rule and coinvariants}

Now we assume $N=2$. We take dominant weights $\lambda^{1},\lambda^{2}$ and the associated framing $\mathbf{W}^{1}$ and $\mathbf{W}^{2}$.
\begin{definition}
	Let $\mu$ be a dominant integral weight of $\mathfrak{g}$. 
	Given any locally closed subset $X \subseteq \mathbf{E}_{\bV,\mathbf{W}^{\bullet},H^{(2)}}$, we denote $X \cap \mathbf{E}^{\geqslant \mu}_{\bV,\mathbf{W}{\bullet},H^{(2)}}$ by $X ^{\geqslant \mu}$,  $X \cap \mathbf{E}^{> \mu}_{\bV,\mathbf{W}{\bullet},H^{(2)}}$ by $X ^{> \mu}$, and $X^{\mu}$ by $X ^{\geqslant \mu} \backslash X ^{> \mu}$.
\end{definition}

The decomposition $\mathbf{W}=\mathbf{W}^{1}\oplus \mathbf{W}^{2}$ gives an isomorphism of affine spaces, so $\Pi_{\bV,\mathbf{W}^1\oplus \mathbf{W}^2}^{s}$ can be regarded as a locally closed subset of $\mathbf{E}_{\bV,\mathbf{W}^{\bullet},H^{(2)}}$. In particular, $\Pi_{\bV,\mathbf{W}^1\oplus \mathbf{W}^2}^{s,\geqslant \mu}$ and $\Pi_{\bV,\mathbf{W}^1\oplus \mathbf{W}^2}^{s,>\mu}$ are close subsets of $\Pi_{\bV,\mathbf{W}^1\oplus \mathbf{W}^2}^{s}$. Their geometric quotients $\tilde{\mathfrak{Z}}^{\geqslant \mu}(\nu,\omega)$ and $\tilde{\mathfrak{Z}}^{> \mu}(\nu,\omega)$ are close subvarities of $\tilde{\mathfrak{Z}}(\nu,\omega)$. We also denote $\tilde{\mathfrak{Z}}^{\geqslant \mu}(\nu,\omega) \backslash \tilde{\mathfrak{Z}}^{> \mu}(\nu,\omega)$ by $\tilde{\mathfrak{Z}}^{\mu}(\nu,\omega)$.
The push-forward of close embeddings induce morphisms
\begin{equation}\label{i3}
	\mathbf{H}^{\mathrm{BM}}_{\mathrm{top}}(\tilde{\mathfrak{Z}}^{> \mu}(\nu,\omega) ,\mathbb{Z}) \rightarrow \mathbf{H}^{\mathrm{BM}}_{\mathrm{top}}(\tilde{\mathfrak{Z}}^{\geqslant \mu}(\nu,\omega) ,\mathbb{Z}),
\end{equation}
\begin{equation}\label{i4}
	\mathbf{H}^{\mathrm{BM}}_{\mathrm{top}}(\tilde{\mathfrak{Z}}^{\geqslant \mu}(\nu,\omega) ,\mathbb{Z}) \rightarrow \mathbf{H}^{\mathrm{BM}}_{\mathrm{top}}(\tilde{\mathfrak{Z}}(\nu,\omega) ,\mathbb{Z}),
\end{equation}
and the pull-back of the open embedding induce a morphism
\begin{equation}
	\pi'_{\geqslant \mu}:\mathbf{H}^{\mathrm{BM}}_{\mathrm{top}}(\tilde{\mathfrak{Z}}^{\geqslant \mu}(\nu,\omega) ,\mathbb{Z}) \rightarrow \mathbf{H}^{\mathrm{BM}}_{\mathrm{top}}(\tilde{\mathfrak{Z}}^{\mu}(\nu,\omega) ,\mathbb{Z}).
\end{equation}

By a similar argument as Proposition \ref{main3}, we have the following proposition.

\begin{proposition}\label{main4}
	Let $M$ be the tensor product $L_{1}(\lambda^{2}) \otimes L_{1}(\lambda^{1})$ of irreducible highest weight $\mathfrak{g}$-modules. With the operator $[\Delta(\nu,\omega)]$, the Hecke correspondences $e_{i},f_{i},i \in I$ and their divided powers, the direct sums of Borel-Moore homology groups $$\mathbf{H}(\tilde{\mathfrak{Z}}^{> \mu}(\omega),\mathbb{Z})= \bigoplus_{\nu \in \mathbb{N}I} \mathbf{H}^{\mathrm{BM}}_{\mathrm{top}}(\tilde{\mathfrak{Z}}^{> \mu}(\nu,\omega) ,\mathbb{Z}) $$
    $$\mathbf{H}(\tilde{\mathfrak{Z}}^{\geqslant \mu}(\omega),\mathbb{Z})= \bigoplus_{\nu \in \mathbb{N}I} \mathbf{H}^{\mathrm{BM}}_{\mathrm{top}}(\tilde{\mathfrak{Z}}^{\geqslant \mu}(\nu,\omega) ,\mathbb{Z}) $$ 
	become  integrable $_{\mathbb{Z}}\mathbf{U}_{1}(\mathfrak{g})$-modules, which are isomorphic to  the $\mathbb{Z}$-forms of $M[>\mu]$ and $M[\geqslant \mu]$ respectively.  Moreover, we have the following commutative diagram
	\[
	\xymatrix{
		\mK^{>\mu}(\lambda^{1},\lambda^{2})|^{\psi}_{v=-1} \ar[d] \ar[r] & \mK^{\geqslant\mu}(\lambda^{1},\lambda^{2})|^{\psi}_{v=-1}\ar[d] \ar[r] & \mK(\lambda^{1},\lambda^{2})|^{\psi}_{v=-1} \ar[d] \\
		\mathbf{H}(\tilde{\mathfrak{Z}}^{> \mu}(\omega),\mathbb{Z}) \ar[r]^{(\ref{i3})} & \mathbf{H}(\tilde{\mathfrak{Z}}^{\geqslant \mu}(\omega),\mathbb{Z}) \ar[r]^{(\ref{i4})} & \mathbf{H}(\tilde{\mathfrak{Z}}(\omega),\mathbb{Z}),
	}
	\]
	where all the vertical maps are $\CC^{s,\omega^{1},\omega^{2}}$. The direct sum of Borel-Moore homology groups
	$$\mathbf{H}(\tilde{\mathfrak{Z}}^{\mu}(\omega),\mathbb{Z})= \bigoplus_{\nu \in \mathbb{N}I} \mathbf{H}^{\mathrm{BM}}_{\mathrm{top}}(\tilde{\mathfrak{Z}}^{ \mu}(\omega)(\nu,\omega) ,\mathbb{Z}) $$
	is isomorphic to the $\mathbb{Z}$-forms of $M[\geqslant \mu]/ M[> \mu]$, which fits into the following commutative diagram of short exact sequences
	\[
	\xymatrix{
		\mK^{>\mu}(\lambda^{1},\lambda^{2})|^{\psi}_{v=-1} \ar[d] \ar[r] & \mK^{\geqslant \mu}(\lambda^{1},\lambda^{2})|^{\psi}_{v=-1}\ar[d] \ar[r]^{\pi_{\geqslant \mu}} & \mK^{\mu}(\lambda^{1},\lambda^{2})(\lambda)|^{\psi}_{v=-1} \ar[d] \\
		\mathbf{H}(\tilde{\mathfrak{Z}}^{> \mu}(\omega),\mathbb{Z}) \ar[r] & \mathbf{H}(\tilde{\mathfrak{Z}}^{\geqslant \mu}(\omega),\mathbb{Z}) \ar[r]^{\pi'_{\geqslant \mu}} & \mathbf{H}(\tilde{\mathfrak{Z}}^{\mu}(\omega),\mathbb{Z}).
	}
	\]
\end{proposition}

	For any irreducible component $Z_{0}$ of $\tilde{\mathfrak{Z}}^{\mu}(wt^{-1}(\mu),\omega)$,  we say its closure $Z=\bar{Z_{0}}$ is strictly dominated by $\mu$. We denote the set of such irreducible components by $S_{2}(\mu)$.  For any irreducible component $Z \in S_{2}(\mu) $, there exists   a unique highest weight vector $v_{Z}$ in $\mathbf{H}^{\textrm{BM}}_{\textrm{top}}(\tilde{\mathfrak{Z}}^{\geqslant \mu}(wt^{-1}(\mu),\omega),\mathbb{Z}) \subseteq \mathbf{H}^{\textrm{BM}}_{\textrm{top}}(\tilde{\mathfrak{Z}}(wt^{-1}(\mu),\omega),\mathbb{Z})$ such that $v_{Z}-[Z] \in \mathbf{H}^{\textrm{BM}}_{\textrm{top}}(\tilde{\mathfrak{Z}}^{> \mu}(wt^{-1}(\mu),\omega),\mathbb{Z})$, then $L_{1}(\mu)^{Z}=\mathbf{U}^{-}(\mathfrak{g}) v_{Z}$ is an irreducible submodule isomorphic to $L_{1}(\mu)$.

\begin{corollary}\label{decrule}
	With the notations above, we have the following decomposition of $\mathfrak{g}$-modules,
	$$L_{1}(\lambda^{1}) \otimes L_{1}(\lambda^{2})=\bigoplus_{\mu}\bigoplus_{Z \in S_{2}(\mu)} L_{1}(\mu)^{Z}.  $$
	 In particular, for any dominant integral weights $\mu$ of $\mathfrak{g}$, the decomposition coefficient $m^{\lambda^{1},\lambda^{2}}_{\mu}$ is determined by  
	$$m^{\lambda^{1},\lambda^{2}}_{\mu}= \dim \mathbf{H}^{\mathrm{BM}}_{\mathrm{top}}(\tilde{\mathfrak{Z}}^{\mu}(wt^{-1}(\mu),\omega),\mathbb{Q}) .$$
\end{corollary}

Now we  take dominant weights $\lambda^{1},\lambda^{2},\cdots ,\lambda^{N}$ and consider the associated $N$-framed quiver.
\begin{definition}
	We say a point $(x,\bar{x},y,z)$ in $\mathbf{E}_{\bV,\mathbf{W}^{\bullet},H^{(N)}}$ is costable if and only if it is in the open complement of $\mathbf{E}^{> 0}_{\bV,\mathbf{W}^{\bullet},H^{(N)}}$. For any locally closed subset $X \subseteq \mathbf{E}_{\bV,\mathbf{W}^{\bullet},H^{(N)}}$, we denote the subset of its costable points by $X^{s\ast }$.
\end{definition}
\begin{remark}
	When $N=1$, Nakajima's stable condition is equivalent to that $(x,\bar{x},y,z)$ doesn't admit nonzero invariant subspace supported on $I$, while to be costable means $(x,\bar{x},y,z)$ doesn't admit nonzero quotient space supported on $I$.  That is why we use costable to call this property. 
\end{remark}

By \cite{MR1865400}, one can inductively define $\mathbb{G}_{m}$-action on $\tilde{\mathfrak{Z}}(\omega)$ to construct Nakajima's tensor product variety for general tensor product $M=L_{1}(\lambda^{1}) \otimes L_{1}(\lambda^{2}) \otimes \cdots \otimes L_{1}(\lambda^{N})$ of $N$ irreducible $\mathfrak{g}$-modules. 
Fix a decomposition $\mathbf{W}=\bigoplus_{1\leqslant l\leqslant N}\mathbf{W}^{l}$, let $\tilde{\mathfrak{Z}}(\omega^{\bullet})=\bigcup_{\nu \in \mathbb{N}I} \tilde{\mathfrak{Z}}(\nu,\omega^{\bullet})$ be the tensor product variety associated to $M=L_{1}(\lambda^{1}) \otimes L_{1}(\lambda^{2}) \otimes \cdots \otimes L_{1}(\lambda^{N})$. The variety $\tilde{\mathfrak{Z}}(\nu,\omega^{\bullet})$ can be identified with a geometric quotient of a locally close subset $\Pi_{\bV,\mathbf{W}^{\bullet}}^{s} \subseteq \mathbf{E}_{\bV,\mathbf{W}^{\bullet},H^{(N)}}$, define $\tilde{\mathfrak{Z}}(\omega^{\bullet})^{s\ast}$ to be the open subset of $\tilde{\mathfrak{Z}}(\omega^{\bullet})$, which is  the disjoint union of the geometric quotients of $\Pi_{\bV,\mathbf{W}^{\bullet}}^{s,s\ast}$ for different $\bV$. Indeed, $\tilde{\mathfrak{Z}}(\nu,\omega^{\bullet})^{s\ast}$ is empty unless $wt(\nu)=0$.

Even though the authors in \cite{fang2025lusztigsheavescharacteristiccycles} only build the characteristic maps for $N=1,2$ cases, their results can be easily generalized to general $N$ by  induction. Hence we have the following proposition for coinvariants.
\begin{proposition}\label{main5}
	Let $M=L_{1}(\lambda^{1}) \otimes L_{1}(\lambda^{2}) \otimes \cdots \otimes L_{1}(\lambda^{N})$ be the tensor product, then the dimension of the coinvariant $M_{\ast}$ is equal to 
	$\dim \mathbf{H}^{\mathrm{BM}}_{\mathrm{top}}(\tilde{\mathfrak{Z}}(\omega^{\bullet})^{s\ast},\mathbb{Q}) $.
\end{proposition}

\end{spacing}

\end{document}